\documentclass[12pt,letterpaper]{article}

\usepackage{amssymb,amsmath,amsthm,amsfonts}
\usepackage[pdftex,colorlinks=true,linkcolor=blue,citecolor=blue,urlcolor=red,unicode=true,hyperfootnotes=true,bookmarksnumbered]{hyperref}
\usepackage[margin=1truein]{geometry}
\usepackage{enumerate}

\usepackage{units}
\usepackage{dsfont}
\usepackage{graphicx}

\makeatletter

\theoremstyle{plain}
\newtheorem{theorem}{\protect\theoremname}[section]
\theoremstyle{definition}
\newtheorem{definition}[theorem]{\protect\definitionname}
\theoremstyle{remark}
\newtheorem{remark}[theorem]{\protect\remarkname}
\theoremstyle{plain}
\newtheorem{prop}[theorem]{\protect\propositionname}
\theoremstyle{remark}
\newtheorem{claim}[theorem]{\protect\claimname}
\theoremstyle{plain}
\newtheorem{lemma}[theorem]{\protect\lemmaname}
\theoremstyle{plain}
\newtheorem{corollary}[theorem]{\protect\corollaryname}
\theoremstyle{remark}
\newtheorem*{claim*}{\protect\claimname}

\@ifundefined{date}{}{\date{}}
\usepackage{tkz-euclide}
\usetikzlibrary{backgrounds,arrows,fit,shapes,calc, chains, decorations.pathmorphing, automata, positioning}
\usepackage{multirow}

\usepackage{bbm}

\@ifundefined{showcaptionsetup}{}{%
 \PassOptionsToPackage{caption=false}{subfig}}
\usepackage{subfig}
\AtBeginDocument{
  
}

\makeatother

\providecommand{\claimname}{Claim}
\providecommand{\corollaryname}{Corollary}
\providecommand{\definitionname}{Definition}
\providecommand{\lemmaname}{Lemma}
\providecommand{\propositionname}{Proposition}
\providecommand{\remarkname}{Remark}
\providecommand{\theoremname}{Theorem}

\begin{document}

\title{First order distinguishability of sparse random graphs}

\author{Tal Hershko\thanks{California Institute of Technology, Pasadena, US; thershko@caltech.edu}, \quad Maksim Zhukovskii\thanks{Department of Computer Science, University of Sheffield, Sheffield S1 4DP, UK; m.zhukovskii@sheffield.ac.uk}}

\date{}

\maketitle

\begin{abstract}
We study the problem of distinguishing between two independent samples $\mathbf{G}_n^1,\mathbf{G}_n^2$ of a binomial random graph $G(n,p)$ by first order (FO) sentences. Shelah and Spencer proved that, for a constant $\alpha\in(0,1)$, $G(n,n^{-\alpha})$ obeys FO zero-one law if and only if $\alpha$ is irrational. Therefore, for irrational $\alpha\in(0,1)$, any fixed FO sentence does not distinguish between $\mathbf{G}_n^1,\mathbf{G}_n^2$ with asymptotical probability 1 (w.h.p.) as $n\to\infty$. We show that the minimum quantifier depth $\mathbf{k}_{\alpha}$ of a FO sentence $\varphi=\varphi(\mathbf{G}_n^1,\mathbf{G}_n^2)$ distinguishing between $\mathbf{G}_n^1,\mathbf{G}_n^2$ depends on how closely $\alpha$ can be approximated by rationals:
\begin{itemize}
\item for all non-Liouville $\alpha\in(0,1)$, $\mathbf{k}_{\alpha}=\Omega(\ln\ln\ln n)$ w.h.p.;
\item there are irrational $\alpha\in(0,1)$ with $\mathbf{k}_{\alpha}$ that grow arbitrarily slowly w.h.p.;
\item $\mathbf{k}_{\alpha}=O_p(\frac{\ln n}{\ln\ln n})$ for all $\alpha\in(0,1)$. 
\end{itemize}
The main ingredients in our proofs are a novel randomized algorithm that generates asymmetric strictly balanced graphs as well as a new method to study symmetry groups of randomly perturbed graphs.

\end{abstract}

\section{Introduction}
    In this paper we focus on the problem of distinguishing between two independent random (simple) graphs by first order (FO) sentences. The vocabulary of FO language of graphs contains the adjacency and the equality relations.  

    Given two graphs $G_1,G_2$ and a FO sentence $\varphi$, we say that $\varphi$ \emph{distinguishes between} $G_1$ and $G_2$ if $G_{1}\models\varphi$ and $G_{2}\not\models\varphi$, or vice versa. We say that  $G_1,G_2$ are \emph{$k$-distinguishable} if there exists a FO sentence $\varphi$ with quantifier depth at most $k$ that distinguishes between them. Recall that the \emph{quantifier depth} of a sentence $\varphi$ is, roughly speaking, the maximum number of nested quantifiers in $\varphi$ (for a formal definition, see~\cite{libkin}). 
    We call the minimum $k$ such that $G_1,G_2$ are $k$-distinguishable {\it the FO distinguishability}, and denote it by $k(G_1,G_2)$. The FO distinguishability was first studied by Spencer and St.\ John for random sequences~\cite{SpencerJohn}. Tight worst-case upper bounds on $k(G_1,G_2)$ for deterministic graphs $G_1,G_2$ were obtained in~\cite{PVV} by Pikhurko, Veith, and Verbitsky. In this paper we study the FO distinguishability of random graphs. Before stating our results, we discuss some motivation of this problem as well as give an overview of its history. 

Distinguishability of random graphs is closely related to zero-one laws,  an important phenomenon in finite model theory. We say that a sequence of random graphs $\left\{ \mathbf{G}_n\right\}_{n=1}^\infty$ satisfies the {\it FO zero-one law} if, for every FO sentence $\varphi$, the limiting probability $\lim_{n\rightarrow\infty}\mathbb{P}(\mathbf{G}_n\models \varphi)$ is either $0$ or $1$. In other words, either $\varphi$ holds with high probability (w.h.p., in what follows) or $\neg\varphi$  holds w.h.p. The celebrated theorem of Glebskii, Kogan, Liogon'kii, Talanov~\cite{glebskii} and Fagin~\cite{fagin} says that the FO zero-one law holds when $\mathbf{G}_{n}\sim G(n,1/2)$, that is, when  $\mathbf{G}_{n}$ is  distributed uniformly on the set of all labelled graphs on $[n]:=\{1,\ldots,n\}$. 

The Bridge Theorem (see ~\cite[Theorem 2.5.1 and Theorem 2.3.1]{spencer}) implies the following relation between zero-one laws and distinguishability: a sequence of random graphs $\left\{ \mathbf{G}_n\right\}_{n=1}^\infty$ satisfies the FO zero-one law if and only if the respective FO distinguishability is unbounded, i.e. for every $k\in\mathbb{N}$,
\begin{equation*}
\lim_{n,m\rightarrow\infty}\mathbb{P}(k(\mathbf{G}_n,\mathbf{G}_m) \geq  k)=1.
\end{equation*}
 In~\cite{KPSV} Kim, Pikhurko, Spencer, and Verbitsky proved a more precise estimation: w.h.p.\ $k\left(\mathbf{G}_{n}^1,\mathbf{G}_{n}^2\right)=\log_2 n+O(\log\log n)$. 
    In~\cite{BZ} Benjamini and the second author of the paper obtained a tight bound: w.h.p. $k\left(\mathbf{G}_{n}^1,\mathbf{G}_{n}^2\right)=\log_2 n-2\log_2\ln n+O(1)$ and is concentrated in 3 consecutive points. They also observed that the same result holds true for the number of variables in infinitary logic $\mathcal{L}^{\omega}_{\infty,\omega}$. 

FO distniguishability is also related to the graph isomorphism problem. Note that two finite graphs are isomorphic if and only if there is no FO sentence that distinguishes between them. Since the truth value of a FO sentence of quantifier depth $k$ on an $n$-vertex graph can be tested in time $O(n^k)$~\cite{libkin}, upper bounds on the FO distinguishability of graphs imply upper bounds on the time complexity of deciding whether or not the graphs are isomorphic. From this perspective, the problem of distinguishing between two independent copies  $\mathbf{G}_{n}^1,\mathbf{G}_{n}^2\sim G(n,1/2)$ arises naturally from the average case analysis of the graph isomorphism problem. However, this approach gives a bound that is far from being optimal: it only shows that  $\mathbf{G}_{n}^1,\mathbf{G}_{n}^2$ can be distinguished in quasi-polynomial time w.h.p. 
 A much more efficient algorithm can be obtained by considering FO logic with counting. Indeed, the results of Babai and Ku\v{c}era~\cite{BK} and Babai, Erd\H{o}s, and Selkow~\cite{BES} (combined with Immerman and Lander's~\cite{IL} logical characterization of color refinement) imply that w.h.p. $\mathbf{G}_n\sim G(n,1/2)$ can be {\it defined} by a FO sentence with counting quantifiers of quantifier depth 4. This implies a polynomial (actually, even linear) time algorithm that distinguishes between $\mathbf{G}_{n}$ and {\it any} non-isomorphic graph w.h.p.

In sparse random graphs, however, the situation is strikingly different. The FO distinguishability can be significantly smaller, and therefore provide much more efficient algorithms for random graphs distinguishing. To state the relevant results, recall the general definition of the binomial random graph: 
$\mathbf{G}_{n}\sim G(n,p)$ is a random graph on $[n]$, in which edges between pairs of vertices appear independently with probability $p$. In~\cite{spencershelah}, Shelah and Spencer studied the validity of the FO zero-one law for $\mathbf{G}_n\sim G(n,n^{-\alpha})$ where $\alpha$ is a positive constant.
    \begin{theorem}[S. Shelah, J. Spencer~\cite{spencershelah}]
    Let $\alpha\in(0,1)$. Then $\mathbf{G}_n\sim G(n,n^{-\alpha})$ obeys a FO zero-one law if and only if $\alpha$ is irrational.
    \label{thm:SS_0-1}
    \end{theorem}

It follows that, for rational $\alpha\in(0,1)$ and two independent copies  $\mathbf{G}_n^1,\mathbf{G}_n^2\sim G(n,n^{-\alpha})$ we have $k(\mathbf{G}_n^1,\mathbf{G}_n^2)=O(1)$ with probability bounded away from 0. Furthermore, it can be even derived that, for every $\varepsilon>0$, $\mathbb{P}(k(\mathbf{G}_n^1,\mathbf{G}_n^2)<C)>1-\varepsilon$ for a certain constant $C=C(\alpha,\varepsilon)$. 

For irrational $\alpha$, the asymptotic behavior of $k(\mathbf{G}_n^1,\mathbf{G}_n^2)$ is more complicated. Authors of~\cite{BZ} suspected that similar methods to those that were applied for dense random graphs might imply $k(\mathbf{G}_n^1,\mathbf{G}_n^2)=O(\ln n)$ w.h.p. for irrational $\alpha$ as well and it was, however, conjectured that actually $k(\mathbf{G}_n^1,\mathbf{G}_n^2)=o(\ln n)$ w.h.p. Our first result shows that this conjecture is true.

    For an irrational $\alpha\in(0,1)$, we let $\mathbf{k}_{\alpha}=k(\mathbf{G}_n^1,\mathbf{G}_n^2)$ for two independent $\mathbf{G}_n^1,\mathbf{G}_n^2\sim G(n,n^{-\alpha})$. For a sequence of random variables $\xi_n$ and a sequence of non-zero constants $a_n$, we write $\xi_n=O_p(a_n)$ when $\xi_n/a_n$ is {\it stochastically bounded}, that is, for every $\varepsilon>0$, there exists $C>0$ such that $\mathbb{P}(|\xi_n/a_n|<C)>1-\varepsilon$ for all $n$.

    \begin{theorem}
    \label{thm:general_upper_bound}
    Let $\alpha\in(0,1)$ be irrational. Then $\mathbf{k}_{\alpha}=O_{p}\left(\frac{\ln n}{\ln\ln n}\right)$.
    \end{theorem}

    On the other hand, our second result shows that it is impossible to get a uniform lower bound approaching infinity: $\mathbf{k}_{\alpha}$ may grow arbitrarily slow. 

    \begin{theorem}
    \label{thm:arbitrary_low_upper_bounds}For every function $f(n)\stackrel{n\rightarrow\infty}{\rightarrow}\infty$
    there exists an irrational $\alpha\in(0,1)$ and an increasing sequence of positive integers $\{n_{t}\}_{t=1}^{\infty}$ such that w.h.p. (as $t\to\infty$) $\mathbf{k}_{\alpha}(n_t)\leq f(n_{t})$.
    \end{theorem}

    We observe that FO distinguishability $\mathbf{k}_{\alpha}$ depends on how well an irrational $\alpha$ can be approximated by rational numbers: the better $\alpha$ is approximable, the closer the behavior of $\mathbf{k}_{\alpha}$ to the case of rational $\alpha$ is.
    This allows us to prove our third result: in contrast to Theorem~\ref{thm:arbitrary_low_upper_bounds}, for almost all $\alpha\in(0,1)$, $\mathbf{k}_{\alpha}$ is at least $\frac{1-o(1)}{\ln 2}\ln\ln\ln n$ w.h.p. Let us recall that the {\it Liouville--Roth irrationality measure} of an irrational number $\alpha$ is the infimum (which can be infinite) of the set of all $d>0$ such that at most finitely many integers $p$ and $q$ satisfy $|\alpha-p/q|\leq q^{-d}$. Due to Roth's theorem~\cite{Roth}, $d=2$ if $\alpha$ is algebraic (while for transcendental numbers $d\geq 2$ or $d=\infty$ --- in the latter case $\alpha$ is called {\it Liouville number}).

    \begin{theorem}
    \label{thm:almost_uniform_lower_bound} Let $d_0>2$ and $\alpha\in(0,1)$ be an irrational number with the Liouville--Roth irrationality measure strictly smaller than $d_0$. Then w.h.p. $\mathbf{k}_{\alpha}\geq\frac{1}{\ln d_0}\ln\ln\ln n$. In particular, for almost all irrational $\alpha\in(0,1)$ (in the Lebesgue measure), w.h.p. $\mathbf{k}_{\alpha}\geq\frac{1-o(1)}{\ln 2}\ln\ln\ln n$.
    \end{theorem}

    The second assertion follows from the fact that the set of $\alpha$ with irrationality measure strictly bigger than 2 has Lebesgue measure 0 due to Khinchin's theorem~\cite{rational_approximations}. In particular, the lower bound $\frac{1-o(1)}{\ln 2}\ln\ln\ln n$ holds true
    for all algebraic irrational $\alpha\in(0,1)$ due to Roth's theorem. 
    In contrast, in our proof of Theorem~\ref{thm:arbitrary_low_upper_bounds}, for a given growing $f(n)$, the corresponding irrational $\alpha$ is defined as a suitable Liouville number.

To prove Theorem \ref{thm:general_upper_bound} we distinguish ${\bf G}_{n}^1$ from ${\bf G}_{n}^2$ via an existential sentence. The strategy is to find a large family of graphs  $\mathcal{F}$ with the property that w.h.p.\ there exists $H\in\mathcal{F}$ which appears as an induced subgraph in ${\bf G}_{n}^1$ but not in ${\bf G}_{n}^2$. Commonly, in order to verify that a representative of a given isomorphism class $H$ appears as an induced subgraph of a random graph, the second moment of the random variable $\mathbf{X}_H$ that counts the number of such appearances is computed. When $H$ has subgraphs that are at least as dense as $H$ itself (i.e., $H$ is not {\it strictly balanced}), the variance of this random variable becomes large and does not allow to apply a concentration inequality. Moreover, symmetries of $H$ negatively affect the expectation of $\mathbb{E}\mathbf{X}_H$. Therefore, we require $\mathcal{F}$ to be a large family of asymmetric strictly balanced graphs with a given density. 

Our last contribution is the existence of a large family of asymmetric strictly balanced graphs as well as a randomized algorithm for sampling such graphs, that we have used to prove Theorem \ref{thm:general_upper_bound}. We prove that for every density $\rho\geq 1+2/n$ (this restriction is tight), bounded from above by a constant, there exists a family of strictly balanced asymmetric graphs on $[n]$ of size $n^{(\rho-o(1))n}$ (see Theorem~\ref{thm:family_of_graphs} in Section~\ref{sec:random_balanced_graph}).
The novel algorithm of sampling {\it random balanced graphs} that are w.h.p. asymmetric that we present in Section~\ref{subsec:Definition} is inspired by the proof of Ruci\'{n}ski and Vince \cite{RV-balanced} of the existence of a strictly balanced graph for every fixed density $\rho\geq 1$. In order to prove that the random balanced graph is asymmetric w.h.p., we develop a new approach for proving asymmetry of randomly perturbed graphs, based on the concept of \emph{alternating cycles}, i.e. cycles whose edges alternate between the edges of a fixed deterministic graph and the random edges. We believe that these results are interesting in their own right --- strictly balanced graphs and their generalizations (as well as their automorphisms) naturally arise in different contexts in random graphs theory, see, e.g.,~\cite{HS,Rucinski,RV-balanced}~and~\cite[Chapter 3]{rucinski_luczak_janson}. 

The paper is organized as follows. In Section~\ref{sec:random_balanced_graph} we define the random balanced graph, study its properties, and then show that it can be used to prove the existence of a large family of strictly balanced asymmetric graphs. The crucial property of the random balanced graph --- asymmetry --- is proved in Section~\ref{sec:Asymmetry_sparse}. The existence of a large family of strictly balanced asymmetric graphs is used in Section~\ref{sec:general_upper_bound} to prove Theorem~\ref{thm:general_upper_bound}. Sections~\ref{sec:low_upper_bound_well_approximable}~and~\ref{sec:Lower-bound} are devoted to proofs of Theorems~\ref{thm:arbitrary_low_upper_bounds}~and~\ref{thm:almost_uniform_lower_bound} respectively. Note that the proof of Theorem~\ref{thm:almost_uniform_lower_bound} requires the random balanced graph and its asymmetry as well. 

\subsubsection*{Notations and conventions}

Throughout the text, we often maintain a convention of denoting random
variables in boldface letters.

For a graph $G$, we denote by $v(G)$ and $e(G)$ the number of vertices and
edges in $G$ respectively. In addition, $\delta(G)$ is the
minimum degree of $G$ and $\Delta(G)$ is the maximum degree of $G$.
For a set of vertices $U$, let $G[U]$ denote the subgraph of $G$ induced
by $U$. 

We will sometimes need to work with multigraphs. We often use asterisks to distinguish multigraphs from simple graphs, unless the distinction is clear from the  context. Recall that a \emph{multigraph} is a pair $G^{*}=(V,E^{*})$ where
$V$ is a set of vertices and $E^{*}$ is a multiset of edges (loops
are allowed). $e(G^{*})$ is the number of edges counted with multiplicities.
Given two multisets $E_{1}^{*},E_{2}^{*}$, we define their sum $E_{1}^{*}+E_{2}^{*}$
as the multiset $E^{*}$ with the property $m_{E^{*}}(e)=m_{E_{1}^{*}}(e)+m_{E_{2}^{*}}(e)$
for every $e$, where $m_{E^{*}}(e)$ denotes the multiplicity of
$e$ in $E^{*}$. For two multigraphs on the same vertex set $G_{1}^{*}=(V,E_{1}^{*})$,
$G_{2}^{*}=(V,E_{2}^{*})$ we define their sum as $G_{1}^{*}+G_{2}^{*}=(V,E_{1}^{*}+E_{2}^{*})$.
Note that $e(G_{1}^{*}+G_{2}^{*})=e(G_{1}^{*})+e(G_{2}^{*})$.  When we talk about (multi)graphs on $n$ vertices, we always refer to labelled graphs with vertex set $[n]$. 

We use the asymptotic notations $f(n)\sim g(n)$ for $f(n)=(1+o(1))g(n)$
and $f(n)\lesssim g(n)$ for $f(n)\leq(1+o(1))g(n)$.

Note that the symbol $\sim$ is also used to denote sampling from a distribution, e.g.\ $\mathbf{G}\sim G(n,p)$. The two meanings of the symbol are distinguished by the context.

\section{The random balanced graph}
\label{sec:random_balanced_graph}

In this section we introduce a random graph model, denoted ${\bf H}(n,m)$,
which we call the \emph{random balanced graph} (with $n$ vertices
and $m$ edges). We then use it to generate a large family of asymmetric,
strictly balanced graphs on $[n]$ with a given density. This will be the key
ingredient in our proof of Theorem \ref{thm:general_upper_bound}
and Theorem \ref{thm:arbitrary_low_upper_bounds}

Recall that a graph $H$ is \emph{balanced} if $d(H)\geq d(H_{0})$
for every subgraph $H_{0}\subseteq H$, where $d(H)=\frac{e(H)}{v(H)}$
is the {\it density} of $H$. Furthermore, $H$ is \emph{strictly balanced}
if $d(H)>d(H_{0})$ for every proper subgraph $H_{0}\subset H$. Also
recall that a graph $H$ is \emph{asymmetric} if its group of automorphisms
is trivial. The same definitions apply for multigraphs.

In Sections \ref{sec:general_upper_bound} and \ref{sec:low_upper_bound_well_approximable}
we will be interested in the case where the density $\frac{m}{n}$
approaches $\frac{1}{\alpha}$ as $n\rightarrow\infty$. Here we consider
a more general setting of a bounded density. The main goal of this
section is to prove the following theorem. 
\begin{theorem}
\label{thm:family_of_graphs}
Fix a constant $c>1$ and let $m=m(n)$ satisfy $n+2\leq m\leq cn$. Then, for every even positive integer $n$, there exists a family $\mathcal{H}(n,m)$
of strictly balanced and asymmetric graphs on $[n]$ and $m$ edges such that $\left|\mathcal{H}(n,m)\right|=\exp\left(-O(n)\right)n^{m}$.
\end{theorem}

\begin{remark}
It makes sense to consider only connected graphs since a disconnected graph is not strictly balanced, and so $m$ should  be at least $n-1$. On the other hand, any tree is strictly balanced. Up to isomorphism, there is only one connected graph with $m=n$ edges which is strictly balanced --- a cycle $C_n$, though it is not asymmetric. Finally, every strictly balanced graph with $m=n+1$ edges is either a cycle with a path between its two vertices, or two cycles joined by a path. Clearly, both graphs are not asymmetric. Thus, the condition $m\geq n+2$ is necessary unless the graph is a tree.
\end{remark}

\begin{remark}
Although we believe that the assertion of Theorem~\ref{thm:family_of_graphs} is true for odd $n$ as well, we state it in this weaker form since an adaptation of our approach to odd $n$ requires a more complicated case-analysis to prove strict balancedness. As for our purposes it is enough to consider even $n$, we omit these technical complications.
\end{remark}

In Subsection~\ref{subsec:Definition} we define ${\bf H}(n,m)$, in Subsection~\ref{subsec:Properties}
we study its properties, 
  and in Subsection~\ref{subsec:The-family} we complete the proof of Theorem~\ref{thm:family_of_graphs}.


\subsection{Definition}
\label{subsec:Definition}

From now on fix $m=m(n)$ and assume that $n$ is even and $n+2\leq m\leq cn$, where $c>1$
is a constant. We begin with a definition of the random balanced \emph{multigraph}
${\bf H}^{*}(n,m)$. It is inspired by Ruci\'{n}ski and Vince's construction
of a strictly balanced graph \cite{RV-balanced}. Write $m=q\cdot\frac{n}{2}+r$ where $q\geq2$ and $0\leq r<\frac{n}{2}$
are integers. ${\bf H}^{*}(n,m)$ is defined as the sum of three different
(multi)graphs, which we call components: the \emph{regular component}
${\bf H}_{r}$ with $(q-2)\frac{n}{2}$ edges, the \emph{Hamiltonian
component} ${\bf H}_{h}$ with $n$ edges and the \emph{balancing
component} ${\bf H}_{b}$ with $r$ edges. 

Let $C$ be a cycle of order $n$; write it as $(v_{1},v_{2},\dots,v_{n},v_{1})$.
For a given $0\leq r\leq n$, define the following set of $r$ \emph{almost
equidistributed} vertices around $C$:
\begin{equation}
R=R(C,r)=\left\{ v_{i}\biggm|1\leq i\leq n,\;\left\lfloor (i-1)\cdot\frac{r}{n}\right\rfloor <\left\lfloor i\cdot\frac{r}{n}\right\rfloor \right\} .\label{eq:R}
\end{equation}
Note that indeed $\left|R\right|=r$. In addition, for any sequence
of $k$ consecutive vertices in $C$, less than $k\cdot\frac{r}{n}+1$
of them belong to $R$. Indeed, a sequence of vertices $v_{i+1},\dots,v_{i+k}$
(with indices taken modulo $n$) has $\left\lfloor (i+k)\cdot\frac{r}{n}\right\rfloor -\left\lfloor i\cdot\frac{r}{n}\right\rfloor $
vertices from $R$, and this is less than $k\cdot\frac{r}{n}+1$.
This property shows that the vertices of $R$ are distributed
almost equally, which is useful for proving balancedness. 
\begin{definition}
\label{def:H*}The random balanced multigraph with $n$ vertices and
$m$ edges ${\bf H}^{*}={\bf H}^{*}(n,m)$ is defined as follows. 
\begin{itemize}
\item Let ${\bf H}_{r}=\left([n],{\bf E}_{r}\right)$ be a uniformly distributed
$(q-2)$-regular (simple) graph on $[n]$.\footnote{We refer to \cite[Chapter 9]{rucinski_luczak_janson} as an introduction
to random regular graphs.}
\item Let ${\bf H}_{h}=\left([n],{\bf E}_{h}\right)$ be a uniformly random Hamilton
cycle, independent of ${\bf H}_{r}$.
\item Let ${\bf H}_{b}=\left([n],{\bf E}_{b}\right)$ be defined as follows.
For every $x\in{\bf R}:=R\left({\bf H}_{h},r\right)$ choose a random vertex ${\bf y}[x]\in[n]$,
such that these $r$ vertices are chosen independently of ${\bf H}_{r}$
and ${\bf H}_{h}$, and are uniformly distributed over all possible
sequences of $r$ \emph{distinct} vertices. Now define ${\bf E}_{b}=\left\{ \left\{ x,{\bf y}[x]\right\} :x\in{\bf R}\right\} $.
We regard ${\bf E}_{b}$ as a multiset: edges may have multiplicity
$2$ and loops are allowed. Thus ${\bf H}_{b}$ may be a multigraph.
\end{itemize}
Finally, define ${\bf H}^{*}={\bf H}_{r}+{\bf H}_{h}+{\bf H}_{b}$. Edges of ${\bf H}_{r}$ are called \emph{regular}, edges of ${\bf H}_{h}$ are called \emph{Hamiltonian}, and edges of ${\bf H}_{b}$ are called \emph{balancing}.
\end{definition}

\begin{remark}
The requirement that the vertices $\{{\bf y}[x]:x\in\mathbf{R}\}$
are all distinct is introduced in order to restrict the maximum degree
of ${\bf H}^{*}$. Indeed, note that ${\bf H}_{r}\cup{\bf H}_{h}$
is a $q$-regular multigraph, and that degrees in ${\bf H}_{b}$ are
$0$, $1$ or $2$. Therefore 
\begin{equation}
q\leq\delta\left({\bf H}^{*}\right)\leq\Delta\left({\bf H}^{*}\right)\leq q+2.\label{eq:min_max_degrees}
\end{equation}
\end{remark}

\begin{remark}
\label{rem:N*}Let $N^{*}=N^{*}(n,m)$ denote the number of possible
values $\left(H_{r},H_{h},H_{b}\right)$ of $\left({\bf H}_{r},{\bf H}_{h},{\bf H}_{b}\right)$.
By definition, $N^{*}=N_{r}N_{h}N_{b}$, where $N_{r}=\exp\left(-O(n)\right)n^{\left(\frac{q}{2}-1\right)n}$ is the number of $(q-2)$-regular graphs on $n$ vertices (see \cite{rucinski_luczak_janson}, Corollary 9.8); $N_{h}=(n-1)!/2$ is the number of cycles on $[n]$; $N_{b}=(n)_{r}$ is the number of sequences of $r$ distinct vertices. Note that $\left({\bf H}_{r},{\bf H}_{h},{\bf H}_{b}\right)$ is uniformly distributed over all $N^{*}$ possible values. However, ${\bf H}^{*}$ itself is \emph{not} uniformly distributed. This is because different multigraphs $H^{*}$ may be represented as a sum $H_{r}+H_{h}+H_{b}$ in a different number of ways.
\end{remark}

\begin{definition}
\label{def:H_and_N}The \emph{random balanced graph} ${\bf H}={\bf H}(n,m)$
is defined as a random (simple) graph whose distribution is the distribution
of ${\bf H}^{*}(n,m)$ conditioned on being simple.
\end{definition}

\begin{remark}
Let $N=N(n,m)$ be the number of possible values $\left(H_{r},H_{h},H_{b}\right)$
of $\left({\bf H}_{r},{\bf H}_{h},{\bf H}_{b}\right)$ such that $H_{r}+H_{h}+H_{b}$
is a simple graph. Conditioning by the event that ${\bf H}^{*}$ is
simple, $\left({\bf H}_{r},{\bf H}_{h},{\bf H}_{b}\right)$ is uniformly
distributed over all $N$ possible values.
\end{remark}

\subsection{Properties\label{subsec:Properties}}

In this subsection we list several important properties of ${\bf H}$.
The main three properties are strict balancedness (Proposition \ref{thm:balanceness}),
a bound on the probability that ${\bf H}$ contains a given set of
edges (Proposition \ref{prop:E_0_existence_bound}), and asymmetry
(Theorem \ref{thm:asymmetry}). 
\begin{prop}
\label{thm:balanceness}
Suppose $H$ is a graph such that $\mathbb{P}\left({\bf H}=H\right)>0$.
Then $H$ is strictly balanced.
\end{prop}

Although the proof of Proposition~\ref{thm:balanceness} resembles the proof of Ruci\'{n}ski and Vince that the graphs that they construct in~\cite{RV-balanced} are strictly balanced, below we present the full proof for completeness. 

\begin{proof}
We actually prove a stronger statement: if $\mathbb{P}\left({\bf H}^{*}=H^*\right)>0$ then $H^*$ is strictly balanced. 

Consider a possible value of $\left(H_{r},H_{h},H_{b}\right)$ of
$\left({\bf H}_{r},{\bf H}_{h},{\bf H}_{b}\right)$ and let $H^*=H_{r}+H_{h}+H_{b}$.
We shall prove that $H^*$ is strictly balanced. That is, we show that
every proper sub-multigraph $H_{0}\subset H^*$ satisfies
\[
\frac{e(H_{0})}{v(H_{0})}<\frac{m}{n}=\frac{q}{2}+\frac{r}{n}.
\]
We now follow a sequence of reductions. 

First, we claim that it suffices to prove strict balancedness of $H'=H_{h}+H_{b}$.
Indeed, suppose that $H'$ is strictly balanced. Consider a set $V_{0}\subset[n]$
and denote $v_{0}=\left|V_{0}\right|$, $H_{0}=H[V_{0}]$ and $H_{0}'=H'[V_{0}]$.
Since $H'$ is strictly balanced, 
\[
\frac{e\left(H_{0}'\right)}{v_{0}}<1+\frac{r}{n}.
\]
Now, $H_{0}$ is obtained from $H_{0}'$ by adding some edges of the
$(q-2)$-regular graph $H_{r}$, which increases the degree of each
vertex by at most $q-2$. Therefore at most $\left(\frac{q}{2}-1\right)v_{0}$
edges are added. Overall
\[
\frac{e(H_{0})}{v_{0}}\leq\frac{\left(\frac{q}{2}-1\right)v_{0}+e\left(H_{0}'\right)}{v_{0}}<\frac{q}{2}+\frac{r}{n}
\]
which proves strict balancedness of $H^*$.

Second, we claim that it suffices to prove strict balancedness of $H''$,
which is obtained from $H'$ by replacing its balancing edges with
loops: each edge $\{x,y[x]\}$ is replaced with a loop at $x$. Indeed,
suppose that $H''$ is strictly balanced. Consider a set $V_{0}\subset[n]$
and denote $v_{0}=\left|V_{0}\right|$, $H_{0}'=H'[V_{0}]$ and $H_{0}''=H''[V_{0}]$.
It is easy to see that $e\left(H_{0}'\right)\leq e\left(H_{0}''\right)$.
From strict balancedness of $H''$, 
\[
\frac{e\left(H_{0}'\right)}{v_{0}}\leq\frac{e\left(H_{0}''\right)}{v_{0}}<1+\frac{r}{n}
\]
which proves strict balancedness of $H'$.

Third, we claim that to prove strict balancedness of $H''$, it suffices
to check only $V_{0}$ which are \emph{segments} of the Hamiltonian
component $H_{h}$; that is, sequences of consecutive vertices on
the cycle. Indeed, assume that we have verified strict balancedness
for segments. Consider a set $V_{0}\subset[n]$ and denote $v_{0}=\left|V_{0}\right|$
and $e_{0}=e\left(H_{0}''\right)$. Notice that $V_{0}$ can be written
as the union of mutually disjoint segments $\tau_{1}\sqcup\tau_{2}\sqcup\dots\sqcup\tau_{t}$
with no Hamiltonian edge between them. Let $v_{i},e_{i}$ denote the
number of vertices and edges in $H''[\tau_{i}]$. From the assumption,
$e_{i}<\left(1+\frac{r}{n}\right)v_{i}$ for every $i$. Then
\[
e_{0}=\sum_{i=1}^{t}e_{i}<\left(1+\frac{r}{n}\right)\sum_{i=1}^{t}v_{i}=\left(1+\frac{r}{n}\right)v_{0}
\]
which proves strict balancedness. 

Finally, it remains to check strict balancedness of $H''$ for segments.
Take any segment $\tau$ with $v_{0}<n$ vertices. The subgraph $H''[\tau]$
contains $v_{0}-1$ Hamiltonian edges and less than $v_{0}\cdot\frac{r}{n}+1$
loops (since the loops are almost equidistributed). Therefore
\[
e_{0}<v_{0}-1+v_{0}\cdot\frac{r}{n}+1=\left(1+\frac{r}{n}\right)v_{0}.
\]
That finishes the proof.
\end{proof}

\begin{prop}
\label{prop:simplicity_probability}There exists a positive constant
$p_{0}$ such that $\mathbb{P}({\bf H}^{*}\;{\rm is}\;{\rm simple})\geq p_{0}$.
\end{prop}

\begin{remark}
Following our notation from Remark \ref{rem:N*} and Definition \ref{def:H_and_N}, let us
note that $\mathbb{P}({\bf H}^{*}\;{\rm is}\;{\rm simple})=N/N^{*}$,
so we can rewrite the inequality as $N\geq p_{0}N^{*}$.
\end{remark}

\begin{proof}[Proof of Proposition~\ref{prop:simplicity_probability}]
Let $A$ be the event that ${\bf H}^{*}$ is simple. Then $A=A_{1}\cap A_{2}$,
where
\begin{enumerate}
\item $A_{1}$ is the event that ${\bf H}_{r}$ and ${\bf H}_{h}$ do not
share common edges.
\item $A_{2}$ is the event that ${\bf H}_{b}$ does not contain loops or
edges with multiplicity $2$ and does not share edges with ${\bf H}_{r}$
or with ${\bf H}_{h}$.
\end{enumerate}
We write $\mathbb{P}(A)=\mathbb{P}(A_{1})\mathbb{P}\left(A_{2}\bigm|A_{1}\right)$
and prove that there are positive constants $p_{1},p_{2}$ such that $\mathbb{P}\left(A_{1}\right)\geq p_{1}$, $\mathbb{P}\left(A_{2}\big|A_{1}\right)\geq p_{2}$.

$\mathbb{P}(A_{1})$ is equal to the probability that the random $(q-2)$-regular
graph ${\bf H}_{r}$ does not contain any edges of a given Hamilton cycle. When $q=2$ it is clearly $1$, so assume $q\geq 3$. We apply \cite[Theorem 1.1]{mckay_2011}, which bounds the probability that a random graph with specified degrees does not contain any edges of a given subgraph. For the special case of regular graphs it yields the following statement. 

\begin{theorem}[B. McKay~\cite{mckay_2011}]
    Let $1\leq d=d(n)\ll n$ be an integer and let $X=X(n)$ be a graph on vertex set $[n]$ and maximum degree $\Delta:=\Delta(X)=o(n)$.  For every even $n$ let $G_{d,X}(n)$ denote the number of labelled $d$-regular graphs on $[n]$ that do not contain any edge of $X$.  Then
    \begin{eqnarray*}
        G_{d,X}(n) & = & \frac{(2E)!}{E!2^{E}(d!)^{n}}\exp\left(-\frac{nd(d-1)}{4E}-\left(\frac{nd(d-1)}{4E}\right)^{2}\right)\\
         &  & \times\exp\left(-\frac{d^2 e(X)}{2E}+O\left(\frac{d^2(d+\Delta)^2}{E}\right)\right).
\end{eqnarray*}
    Here $E=\frac{nd}{2}$ is the number of edges in a $d$-regular graph on $n$ vertices. 
\end{theorem}
In our case, let $d=q-2$ (which is a constant) and let $X$ be a Hamiltonian cycle, so $e(X)=n$ and $\Delta=2$. Also let $\emptyset$ denote the empty graph (on $n$ vertices), so $G_{d,\emptyset}(n)$ is the number of labelled $d$-regular graphs. Then
\begin{equation*}
    \mathbb{P}\left(A_{1}\right)=\frac{G_{d,X}(n)}{G_{d,\emptyset}(n)}=\exp\left(-\frac{nd^2}{nd}+O\left(\frac{1}{nd}\right)\right)\sim e^{-d}=\mathrm{e}^{-(q-2)}.
\end{equation*}
Hence $\mathbb{P}\left(A_{1}\right)$ is bounded from below by a positive constant. 

Now let us bound $\mathbb{P}\left(A_{2}\bigm|A_{1}\right)$. We fix values
${\bf H}_{r}=H_{r}$ and ${\bf H}_{h}=H_{h}$ such that $A_{1}$ holds,
and prove the existence of a constant $p_{2}$ (independent of $H_{r},H_{h}$)
such that
\[
\mathbb{P}\left(A_{2}\mid{\bf H}_{r}=H_{r},\;{\bf H}_{h}=H_{h}\right)\geq p_{2}.
\]

Let $R=\{x_{1},x_{2},\dots,x_{r}\}$ (the labeling may be chosen arbitrarily) be the set of $r$ equidistributed vertices around $H_{h}$. Since $A_{1}$ holds, $H_{r}+H_{h}$ is a simple $q$-regular graph.  The sequence of values ${\bf y}[x_{1}],{\bf y}[x_{2}]\dots,{\bf y}[x_{r}]$
is drawn uniformly from all the $(n)_{r}$ sequences of $r$ distinct
vertices. We provide a lower bound on the number of choices which
satisfy $A_{2}$.

Let us sequentially choose the values $y[x_{1}],y[x_{2}]\dots,y[x_{r}]$,
making sure that $A_{2}$ is satisfied at every step. When we get
to $y[x_{j}]$, there are at least $n-q-j$ possible choices which
assure that $A_{2}$ is still satisfied. We deduce
\begin{multline*}
\mathbb{P}\left(A_{2}\mid{\bf H}_{r}=H_{r},\;{\bf H}_{h}=H_{h}\right)  \geq  \frac{(n-q-1)_{r}}{(n)_{r}}
  \overset{(*)}{\geq} \frac{(n-2c-1)_{r}}{(n)_{r}}\\
  \overset{(**)}{\geq}  \frac{(n-2c-1)_{\frac{n}{2}}}{(n)_{\frac{n}{2}}}
  \overset{(***)}{\geq}  \left(\frac{\frac{n}{2}-2c-1}{\frac{n}{2}}\right)^{\frac{n}{2}} \sim  {\rm e}^{-(2c+1)}.
\end{multline*}
Here $(*)$ follows from $q\leq\frac{2m}{n}\leq2c$, $(**)$ follows
from $r<\frac{n}{2}$, and $(***)$ holds since $\frac{n-2c-i-1}{n-i}$
decreases as a function of $i$. 
\end{proof}
\begin{prop}
\label{prop:E_0_existence_bound}There exists a positive constant $c_{0}$ such that, for every set $E_{0}\subseteq\binom{[n]}{2}$, we have
$\mathbb{P}\left(E_{0}\subseteq E\left({\bf H}\right)\right)\leq\left(c_0/n\right)^{\left|E_{0}\right|}$.
\end{prop}

\begin{proof}
The idea is to partition $E_{0}$ into three disjoint sets
and separately bound the probability that they are subsets of ${\bf E}_{r},{\bf E}_{h},{\bf E}_{b}$. It will be simpler to work with ${\bf H}^{*}$, and
for that we will rely on Proposition \ref{prop:simplicity_probability}.

Formally, fix $E_{0}\subseteq\binom{[n]}{2}$ and let $\mathcal{P}$
be the set of triplets $(E_{r},E_{h},E_{b})$ which form a partition
of $E_{0}$. Note that $\left|\mathcal{P}\right|=3^{\left|E_{0}\right|}$.
For every $P=\left(E_{r},E_{h},E_{b}\right)\in\mathcal{P}$ let $A_{P}$
denote the event that in\emph{ ${\bf H}^{*}$} we have $E_{r}\subseteq{\bf E}_{r}$, $E_{h}\subseteq{\bf E}_{h}$, $E_{b}\subseteq{\bf E}_{b}$. Then, from Proposition \ref{prop:simplicity_probability},
\begin{eqnarray*}
\mathbb{P}\left(E_{0}\subseteq E\left({\bf H}\right)\right)  \leq  \frac{1}{p_{0}}\mathbb{P}\left(E_{0}\subseteq E\left({\bf H}^{*}\right)\right)
  \leq \frac{1}{p_{0}}\sum_{P\in\mathcal{P}}\mathbb{P}(A_{P}).
\end{eqnarray*}

We will now prove that there exists a constant $c_{1}$ such that
$\mathbb{P}\left(A_{P}\right)\leq\left(c_1/n\right)^{\left|E_{0}\right|}$
for every $P\in\mathcal{P}$. Then, taking $c_{0}=\frac{1}{p_{0}}\cdot3c_{1}$
finishes the proof. 

So, let us fix a partition $P=\left(E_{r},E_{h},E_{b}\right)$ of
$E_{0}$. Let $\ell_{0}=\left|E_{0}\right|$ and similarly define
$\ell_{r},\ell_{h},\ell_{b}$. 

\textbf{Step 1.} We apply the following proposition from~\cite{Bollobas-asymmetry}.
\begin{prop}[\cite{Bollobas-asymmetry}, Equation (2)]
    Fix a constant $d\geq 1$ and let $\mathbf{G}_{n,d}$ be a random $d$-regular graph on $[n]$ vertices. Then there exists a constant $c>0$  (depending on $d$ but not on $n$) such that, for every set $E\subseteq \binom{[n]}{2}$, $\mathbb{P}\left(E\subseteq E({\bf G}_{n,d})\right)\leq\left(\frac{c}{n}\right)^{|E|}$.
\end{prop}
In our case, it follows that there exists a positive constant $c_r$ such that  $\mathbb{P}\left(E_{r}\subseteq{\bf E}_{r}\right)\leq\left(\frac{c_{r}}{n}\right)^{\ell_{r}}$.

\textbf{Step 2. }Let us bound $\mathbb{P}\left(E_{h}\subseteq{\bf E}_{h}\right)$.
Recall that ${\bf E}_{h}$ is the set of edges of a random Hamilton
cycle. If $E_{h}$ is not contained in any Hamilton cycle, this
probability is trivially $0$, so assume that it is not the case.

If $\ell_{h}=n$, then $E_{h}$ is already the set of edges of a Hamilton
cycle, and then $\mathbb{P}\left(E_{h}\subseteq{\bf E}_{h}\right)=\frac{2}{(n-1)!}$.

If $\ell_{h}<n$, then the graph $([n],E_{h})$ has exactly $n-\ell_{h}$
connected components, all of them are either paths or isolated vertices
(which are not considered as paths for now). Let $s$ be the number
of paths. Then the number of Hamilton cycles containing $E_{h}$
equals $2^{s-1}(n-\ell_{h}-1)!$. Therefore
\begin{align*}
\mathbb{P}\left(E_{h}\subseteq{\bf E}_{h}\right)  &=  2^{s}\cdot\frac{(n-\ell_{h}-1)!\cdot n}{n!}
  \leq  \frac{2^{\ell_{h}}}{(n-1)_{\ell_{h}}}\\
  &\leq\left(\frac{2{\rm e}}{n-1}\right)^{\ell_{h}}
  =  \left(\frac{n}{n-1}\right)^{\ell_{h}}\left(\frac{2{\rm e}}{n}\right)^{\ell_{h}}
  \lesssim  {\rm e}\left(\frac{2{\rm e}}{n}\right)^{\ell_{h}}.
\end{align*}
Note that the last bound also holds when $\ell_{h}=n$. 

\textbf{Step 3. }Finally, let us bound $\mathbb{P}\left(E_{b}\subseteq{\bf E}_{b}\right)$.
We shall prove that for every possible value $H_{h}$ of \textbf{${\bf H}_{h}$,}
\[
\mathbb{P}\left(E_{b}\subseteq{\bf E}_{b}\bigm|{\bf H}_{h}=H_{h}\right)\leq\left(\frac{2{\rm e}}{n}\right)^{\ell_{b}}.
\]
That will prove $\mathbb{P}\left(E_{b}\subseteq{\bf E}_{b}\right)\leq\left(\frac{2{\rm e}}{n}\right)^{\ell_{b}}.$

Fix a value ${\bf H}_{h}=H_{h}$. It determines the set of equidistributed
vertices $R=R(H_{h},r)$. Again, we arbitrarily enumerate it: $R=\{x_{1},x_{2},\dots,x_{r}\}$.
Recall that ${\bf E}_{b}$ is defined as the (multi)set of edges of
the form $\{x_{j},{\bf y}[x_{j}]\}$, where ${\bf y}[x_{1}],\dots,{\bf y}[x_{r}]$
are drawn uniformly from all possible sequences of $r$ \emph{distinct}
vertices. The number of ways to choose the vertices $y[x]$ such that
$E_{b}\subseteq{\bf E}_{b}$ holds is at most $2^{\ell_{b}}(n-\ell_{b})_{r-\ell_{b}}$.
Indeed, there are at most $2^{\ell_{b}}$ ways to choose a direction
$x_{j}\rightarrow y[x_{j}]$ for every edge of $E_{b}$ and
at most $(n-\ell_{b})_{r-\ell_{b}}$ ways to choose the remaining
$y[x_{j}]$. Overall,
\begin{eqnarray*}
\mathbb{P}\left(E_{b}\subseteq{\bf E}_{b}\bigm|{\bf H}_{h}=H_{h}\right)  \leq 2^{\ell_{b}}\frac{(n-\ell_{b})_{r-\ell_{b}}}{(n)_{r}}
  =  2^{\ell_{b}}\frac{1}{(n)_{\ell_{b}}}{\leq}\left(\frac{2{\rm e}}{n}\right)^{\ell_{b}}.
\end{eqnarray*}
In conclusion, the three steps and the independence between the different
components of ${\bf H}^{*}$ show that indeed there exists a constant
$c_{1}$ such that $\mathbb{P}\left(A_{P}\right)\leq\left(\frac{c_{1}}{n}\right)^{\ell_{0}}$.
That finishes the proof. 
\end{proof}
Finally, we address the asymmetry of ${\bf H}$.
\begin{theorem}
\label{thm:asymmetry}${\bf H}$ is asymmetric w.h.p.
\end{theorem}

\begin{proof} We consider two separate cases $m\geq\frac{3}{2}n$ and $m<\frac{3}{2}n$.

\emph{Dense case:} $m\geq\frac{3}{2}n$. In this case we follow an
argument due to Bollob\'{a}s \cite{Bollobas-asymmetry}, which proves
asymmetry of the random regular graph $G(n,r)$ with $r\geq3$ fixed.
This argument directly generalizes to the following result.
\begin{theorem}[Bollob\'{a}s \cite{Bollobas-asymmetry}]
\label{thm:asymmetry_bollobas} Fix an integer constant $\Delta\geq3$. Let $\{{\bf G}_{n}=([n],\mathbf{E}_n)\}_{n=1}^{\infty}$
be a sequence of random graphs with  $\delta({\bf G}_{n})\geq 3$ and $\Delta({\bf G}_{n})\leq\Delta$.
Assume that there exists a constant $c_{0}>0$ such that $\mathbb{P}\left(E_{0}\subseteq E\left({\bf G}_{n}\right)\right)\leq\left(\frac{c_{0}}{n}\right)^{\left|E_{0}\right|}$
for every $E_{0}\subseteq\binom{[n]}{2}$. 
Then ${\bf G}_{n}$ is asymmetric w.h.p.
\end{theorem}

For the sake of completeness, we give the full proof of Theorem~\ref{thm:asymmetry_bollobas} in Appendix~\ref{appendix:dense_asym}.

Equation (\ref{eq:min_max_degrees}) and Proposition \ref{prop:E_0_existence_bound}
show that the random balanced graph ${\bf H}$ satisfies properties
2 and 3. When $m\geq\frac{3}{2}n$ it also satisfies property 1, and
therefore it is asymmetric w.h.p. 

\emph{Sparse case:} $m<\frac{3}{2}n$. In this case $\delta(\mathbf{G}_n)=2$ and Bollob\'{a}s's argument does not apply. Instead, we develop an entirely different approach for proving asymmetry, generally applicable to randomly perturbed cycles. Indeed, in the sparse case, the regular
component vanishes and ${\bf H}={\bf H}_{h}+{\bf H}_{b}$. That is, ${\bf H}$ is a Hamilton cycle with additional $2\leq r<\frac{n}{2}$ randomly scattered edges.  Our approach is based on the observation that non-trivial automorphisms give rise to certain configurations which are very rare in the sparse case. A key concept in the proof is that of an \emph{alternating cycle}: a cycle which alternates between Hamiltonian edges and balancing edges. Since this proof may be of its own interest and since it is long enough to interrupt the flow of the paper, we present it in the separate Section~\ref{sec:Asymmetry_sparse}.
\end{proof}

\subsection{The family $\mathcal{H}(n,m)$\label{subsec:The-family}}

We are now ready to prove Theorem \ref{thm:family_of_graphs}. Given
$m=m(n)$ which satisfies $n+2\leq m\leq cn$, define $\mathcal{H}(n,m)$
as the set of all asymmetric graphs which are possible values of ${\bf H}={\bf H}(n,m)$.
That is,
\[
\mathcal{H}_{n,m}=\left\{ H\;{\rm is}\;{\rm asymmetric}\bigm|\mathbb{P}\left({\bf H}(n,m)=H\right)>0\right\} .
\]
$\mathcal{H}(n,m)$ is indeed a family of strictly balanced and asymmetric
graphs with $n$ vertices and $m$ edges (by the definition and from Proposition
\ref{thm:balanceness}). Note that it is closed under isomorphism
(since so is the family of all possible values of ${\bf H}$). 

It remains to prove that $\left|\mathcal{H}(n,m)\right|=\exp\left(-O(n)\right)n^{m}$.  We start with an estimation of $N^{*}=N^{*}(n,m)$. Recall that 
\begin{eqnarray*}
N^{*} & = & N_{r}N_{h}N_{b}=\exp\left(-O(n)\right)n^{\left(\frac{q}{2}-1\right)n}\cdot\frac{(n-1)!}{2}\cdot(n)_{r}\\
 & = & \exp\left(-O(n)\right)n^{\left(\frac{q}{2}-1\right)n}n^{n}n^{r}=\exp\left(-O(n)\right)n^{m}.
\end{eqnarray*}
From Proposition \ref{prop:simplicity_probability} we have $N=\Theta(N^{*})$.
Combining that with Theorem \ref{thm:asymmetry}, we deduce that the
number of possible $(H_{r},H_{h},H_{b})$ such that $H_{r}+H_{h}+H_{b}$
is an asymmetric (simple) graph is at most $\exp\left(-O(n)\right)n^{m}$. 
The only remaining issue is that we may overcount graphs:
different values $(H_{r},H_{h},H_{b})$ may have the same sum. To
handle this, the following simple lemma bounds the number of ways
to express a given (simple) graph as $H_{r}+H_{h}+H_{b}$. We allow
very coarse estimations since we only care that the bound is exponential. 
\begin{claim}
\label{lem:B(n)_bound}There is a function $B(n)=\exp\left(O(n)\right)$
such that, for every graph $H\in\mathcal{H}(n,m)$, the number of possible
values $(H_{r},H_{h},H_{b})$ of $({\bf H}_{r},{\bf H}_{h},{\bf H}_{b})$
such that $H=H_{r}+H_{h}+H_{b}$ is at most $B(n)$.
\end{claim}

\begin{proof}
From (\ref{eq:min_max_degrees}) we know that $\Delta(H)\leq q+2$.
Therefore $(q+2)^{n}$ is a trivial (and very coarse) bound on the
number of Hamilton cycles in $H$. The Hamiltonian component $H_{h}$
must be one of them. 

Now fix $H_{h}$, which also fixes the set $R$. For every $x\in R$,
$y[x]$ must be its neighbor in $H$ and can therefore be chosen in
at most $q+2$ ways. So the number of choices for $H_{b}$ is at most
$(q+2)^{r}$. 

The values $H_{h},H_{b}$ uniquely determine $H_{r}$ (as the graph
with all the remaining edges). Overall, the number of choices is at
most
\[
(q+2)^{n+r}\leq\left(2c+2\right)^{\frac{3}{2}n}=:B(n)
\]
where we used the inequalities $q\leq2c$ and $r<\frac{n}{2}$. That
finishes the proof.
\end{proof}
From the claim we immediately get
\[
\left|\mathcal{H}(n,m)\right|\geq\frac{1}{B}\cdot\exp\left(-O(n)\right)n^{m}=\exp\left(-O(n)\right)n^{m}
\]
and so Theorem \ref{thm:family_of_graphs} follows.

\section{The general upper bound}
\label{sec:general_upper_bound}

In this section we prove Theorem \ref{thm:general_upper_bound}. Let
us fix an irrational $\alpha\in (0,1)$ and a function $\omega=\omega(n)\rightarrow\infty$,
and let ${\bf G}_{n}^{1},{\bf G}_{n}^{2}\sim G(n,n^{-\alpha})$ be independent.
We prove that there exists a purely existential FO-sentence $\varphi=\varphi({\bf G}_{n}^{1},{\bf G}_{n}^{2})$,
of quantifier depth at most $\omega\frac{\ln n}{\ln\ln n}$,
that distinguishes between ${\bf G}_{n}^{1},{\bf G}_{n}^{2}$ w.h.p.
To do that, it is sufficient to find a family of graphs $\mathcal{F}$ on $[v]$ where
 $v=\left\lfloor \omega\frac{\ln n}{\ln\ln n}\right\rfloor $ is even and satisfies the following two properties:
\begin{description}
\item [{(A)}] W.h.p., ${\bf G}_{n}^{1}$ contains an induced subgraph which
is isomorphic to some $H\in\mathcal{F}$. 
\item [{(B)}] For any specific graph $H\in\mathcal{F}$, w.h.p.\ ${\bf G}_{n}^{2}$
does not contain an induced subgraph isomorphic to $H$.
\end{description}
Indeed, suppose we have found $\mathcal{F}$ with these properties. Let
$H=H({\bf G}_{n}^{1})$ be the minimum graph from $\mathcal{F}$ (with respect
to some arbitrary ordering) such that ${\bf G}_{n}^{1}$ contains an induced
subgraph isomorphic to $H$. From (A), such a graph $H$ exists w.h.p.
Now let $\varphi=\varphi({\bf G}_{n}^{1})$ be a FO-sentence expressing
the property of containing an induced subgraph isomorphic to $H$,
with quantifier depth at most $v$. From (B),
it distinguishes between ${\bf G}_{n}^{1},{\bf G}_{n}^{2}$ w.h.p.

Our approach is to define $\mathcal{F}$ as a set of typical graphs from the family $\mathcal{H}(v,e)$
from Theorem \ref{thm:family_of_graphs}, with $e$ around $\frac{1}{\alpha}v$.
 More precisely, we set $e=\left\lceil \frac{1}{\alpha}v\right\rceil +1$: the additional $1$ makes the subgraphs dense enough to assure Property (B). 
We get Property (A) in the usual way using Chebyshev's inequality.  To make it work, we need asymmetry and strict balancedness that are provided by Theorem~\ref{thm:family_of_graphs}. For technical reasons, we need to further refine the latter property and make sure that our family comprises graphs that are {\it enhancely balanced}, i.e. 
 small subgraphs
of $H\in\mathcal{F}$ have density slightly below $\frac{1}{\alpha}$.
We will make use of the following proposition.
\begin{prop}[Enhanced Balancedness]
\label{prop:enhanced_balanceness}Consider ${\bf H}={\bf H}(v,e)$
with even $v=\left\lfloor \omega\frac{\ln n}{\ln\ln n}\right\rfloor $
and $e=\left\lceil \frac{1}{\alpha}v\right\rceil +1$. Then there
exist positive constants $\delta_{0},\beta_{0}>0$ (depending only
on $\alpha$) such that the following holds. W.h.p., for every subgraph
$H_{0}$ of ${\bf H}$ with $v(H_{0})\leq\delta_{0}v$, we have $e(H_{0})<\frac{1}{\alpha}v(H_{0})-\beta_{0}$.
\end{prop}


The proof of Proposition \ref{prop:enhanced_balanceness} is postponed
to Subsection \ref{subsec:enhanced_balanceness_proof}. In Subsection
\ref{subsec:The-family-F} we use it to construct a family $\mathcal{F}$
with Properties (A) and (B) and thus complete the proof of Theorem
\ref{thm:general_upper_bound}.

\subsection{A suitable family of subgraphs\label{subsec:The-family-F}}
Let $\mathcal{H}=\mathcal{H}(v,e)$ be the family of graphs from Theorem
\ref{thm:family_of_graphs} with
$v=\left\lfloor \omega\frac{\ln n}{\ln\ln n}\right\rfloor$ , $e=\left\lceil \frac{1}{\alpha}v\right\rceil +1.$ Without loss of generality we may assume that $v$ is even. We define $\mathcal{F}=\mathcal{F}(v,e)$ as the family of graphs $H\in\mathcal{H}$
which additionally satisfy the enhanced balancedness property from
Proposition \ref{prop:enhanced_balanceness}.
Like $\mathcal{H}$, the family $\mathcal{F}$ is closed under isomorphism.
Moreover, from Theorem \ref{thm:family_of_graphs}, Proposition \ref{prop:enhanced_balanceness}, and Claim \ref{lem:B(n)_bound},
we have the following bound on $\left|\mathcal{F}\right|$. 
\begin{claim}
$
\left|\mathcal{F}\right|=\exp\left(-O(v)\right)\left|\mathcal{H}\right|=\exp\left(-O(v)\right)v^{e}.
$
\label{cl:cardinality_F_and_H}
\end{claim}

\begin{proof}
The second equality follows directly from the lower bound on $\left|{\cal H}\right|$
from Theorem \ref{thm:family_of_graphs}, so it remains to prove the
first equality. Let $N$ be the number of possible values $(H_{r},H_{h},H_{b})$
such that $H_{r}+H_{h}+H_{b}$ is simple, and let $N_{{\cal F}}$
be the number of possible values $(H_{r},H_{h},H_{b})$ such that
$H_{r}+H_{h}+H_{b}\in{\cal F}$. Proposition \ref{prop:enhanced_balanceness}
implies $\frac{N_{{\cal F}}}{N}=1-o(1)$. Claim \ref{lem:B(n)_bound}
implies $N_{{\cal F}}=\exp\left(O(v)\right)\left|{\cal F}\right|$.
Therefore $\left|{\cal F}\right|=\exp\left(-O(v)\right)N_{{\cal F}}=\exp\left(-O(v)\right)N=\exp\left(-O(v)\right)\left|{\cal H}\right|.$
\end{proof}

\begin{lemma}
$\mathcal{F}$ satisfies Property (B). 
\end{lemma}

\begin{proof}
Fix $H\in\mathcal{F}$ and let $\mathbf{X}_{H}$
be the number of induced copies of $H$ in ${\bf G}_{2}$. Then
\[
\mathbb{E}\left(\mathbf{X}_{H}\right)=(n)_{v} p^{e}\left(1-p\right)^{\binom{v}{2}-e}\sim n^{v}p^{e}=n^{v-\alpha e}.
\]
since
$v,e=O(\ln n)$ and $p=n^{-\alpha}$. 
From the definition of $e$ we have $v-\alpha e\leq-\alpha$, so overall
$\mathbb{E}\left(\mathbf{X}_{H}\right)=O\left(n^{-\alpha}\right)$. From Markov's
inequality, $\mathbb{P}\left(X_{H}\geq1\right)=o(1)$.
\end{proof}
\begin{lemma}
\label{lem:property_A}$\mathcal{F}$ satisfies property (A).
\end{lemma}

\begin{proof}
Let $\mathbf{X}_{\mathcal{F}}$ count the number
of induced subgraphs $H\subseteq{\bf G}_{n}^{1}$ with $H\in\mathcal{F}$.
We need to prove $\mathbb{P}(\mathbf{X}_{\mathcal{F}}=0)=o(1)$. From Chebyshev's
inequality, it suffices to prove $\frac{{\rm Var}(\mathbf{X}_{\mathcal{F}})}{\left[\mathbb{E}\left(\mathbf{X}_{\mathcal{F}}\right)\right]^{2}}=o(1).$ Write $\mathbf{X}_{\mathcal{F}}=\sum_{H}\mathds{1}_{H\subseteq\mathbf{G}_1}$, where the sum is over
all subgraphs $H$ of the complete graph $K_{n}$ with $H\in F$.
The expected value is
\begin{align*}
\mathbb{E}\left(\mathbf{X}_{\mathcal{F}}\right)  =\sum_{H}\mathbb{P}\left(H\subseteq{\bf G}_{n}^{1}\right)=\binom{n}{v}\left|\mathcal{F}\right| p^{e}\left(1-p\right)^{\binom{v}{2}-e}\sim n^{v-\alpha e}\frac{\left|\mathcal{F}\right|}{v!}.
\end{align*}
The variance is
\begin{equation}
{\rm Var}(\mathbf{X}_{\mathcal{F}})=\sum_{H,H'}\left[\mathbb{E}\left(\mathds{1}_{H\subseteq\mathbf{G}_1,H'\subseteq\mathbf{G}_1}\right)-\mathbb{E}\left(\mathds{1}_{H\subseteq\mathbf{G}_1}\right)\mathbb{E}\left(\mathds{1}_{H'\subseteq\mathbf{G}_1}\right)\right].\label{eq:var(X_F)}
\end{equation}
We decompose the sum in (\ref{eq:var(X_F)}) by considering different
possible intersection patterns of $H,H'$. For every non-negative
integers $v_{0},e_{0}$ let $S_{v_{0},e_{0}}$ denote the sum in (\ref{eq:var(X_F)}),
but only over the pairs $H,H'$ which share exactly $v_{0}$ common
vertices and $e_{0}$ common edges. Then ${\rm Var}(\mathbf{X}_{\mathcal{F}})=\sum_{v_{0},e_{0}}S_{v_{0},e_{0}}$.
When $v_{0}\leq1$, the indicators $\mathds{1}_{H},\mathds{1}_{H'}$
are independent and the corresponding summand is $0$. Taking enhanced
balancedness into account, we may therefore sum over the set $\mathcal{I}$ of pairs $(v_{0},e_{0})$ such that $v_0\geq 2$, $e_0\leq\frac{e}{v}v_0$ and, moreover, $e_0\leq\frac{1}{\alpha}v_0-\beta_0$ whenever $v_{0}\leq\delta_{0}v$. Each $S_{v_{0},e_{0}}$ can be written as follows:
\begin{align*}
S_{v_{0},e_{0}}  = & \binom{n}{v}\binom{v}{v_{0}}\binom{n-v}{v-v_{0}}\cdot T_{v_{0},e_{0}}\\
  & \times\left[p^{2e-e_{0}}(1-p)^{2\left(\binom{v}{2}-e\right)-\left(\binom{v_{0}}{2}-e_{0}\right)}-p^{2e}(1-p)^{2\left(\binom{v}{2}-e\right)}\right].
\end{align*}
Here $T_{v_{0},e_{0}}$ denotes the number of possible choices of
a pair of induced subgraphs $H,H'$ on two given sets of $v$ vertices
with $v_{0}$ common vertices, such that $H,H'\in\mathcal{F}$ and they
share exactly $e_{0}$ edges. Simple estimations yield
\begin{align}
\frac{S_{v_{0},e_{0}}}{[\mathbb{E}\left(\mathbf{X}_{\mathcal{F}}\right)]^{2}}\lesssim\frac{(v!)^{2}}{v_{0}!\left((v-v_{0})!\right)^{2}}\cdot\frac{T_{v_{0},e_{0}}}{\left|\mathcal{F}\right|^{2}}\cdot\frac{1}{n^{v_{0}-\alpha e_{0}}}.
\label{eq:second_moment_X_F}
\end{align}
The following claim bounds the term $T_{v_{0},e_{0}}/\left|\mathcal{F}\right|^{2}$. 
\begin{claim}
\label{lem:T_estimation}For every $(v_{0},e_{0})\in\mathcal{I}$,
$T_{v_{0},e_{0}}/\left|\mathcal{F}\right|^{2}=\exp\left(O(v)\right)\left(\frac{c_{0}}{v}\right)^{e_{0}}$,
where $c_{0}$ is the constant from Proposition \ref{prop:E_0_existence_bound},
and the term $O(v)$ does not depend on $(v_{0},e_{0})$. 
\end{claim}

\begin{proof}
Let us fix two sets of vertices $V_{1},V_{2}$ with $\left|V_{1}\right|=\left|V_{2}\right|=v$
and $V_{0}=V_{1}\cap V_{2}$ having $\left|V_{0}\right|=v_{0}$. We
need to bound $T_{v_{0},e_{0}}$, the number of pairs $H_{1},H_{2}$
such that each $H_{i}$ is a graph on $V_{i}$ with $H_{i}\in\mathcal{F}$,
$H_{1}[V_{0}]=H_{2}[V_{0}]$, and the common induced subgraph contains
exactly $e_{0}$ edges. 

First choose $H_{1}\in\mathcal{F}$ on $V_{1}$ such that $H_{1}\left[V_{0}\right]$
contains exactly $e_{0}$ edges. The number of choices is trivially
bounded by $\left|\mathcal{F}\right|$. Now, given $H_{1}$, we must
choose $H_{2}\in\mathcal{F}$ on $V_{2}$ which contains a given set
of $e_{0}$ edges. The number of choices is bounded by the number of graphs from $\mathcal{H}$ that contain the same set of edges which is, due to Proposition \ref{prop:E_0_existence_bound}, at most $\left(\frac{c_{0}}{v}\right)^{e_{0}}N$,
where $N=N(v,e)$ is again the number of possible values $(H_{r},H_{h},H_{b})$
such that $H_{r}+H_{h}+H_{b}$ is simple. By Claim~\ref{cl:cardinality_F_and_H}, $\left|\mathcal{F}\right|=\exp\left(-O(v)\right)N$.
Therefore
\[
T_{v_{0},e_{0}}\leq\left|\mathcal{F}\right|\cdot\left(\frac{c_{0}}{v}\right)^{e_{0}}\cdot\exp\left(O(v)\right)\left|\mathcal{F}\right|=\exp\left(O(v)\right)\left(\frac{c_{0}}{v}\right)^{e_{0}}\left|\mathcal{F}\right|^{2}.
\]
\end{proof}
We now return to the proof of Lemma \ref{lem:property_A}. It suffices
to prove that the right-hand side in (\ref{eq:second_moment_X_F}) is $o\left(v^{-2}\right)$
(uniformly), since the number of summands is $\Theta\left(v^{2}\right)$.  We have
\begin{eqnarray*}
\frac{S_{v_{0},e_{0}}}{[\mathbb{E}\left(\mathbf{X}_{\mathcal{F}}\right)]^{2}} & = & \frac{(v!)^{2}}{v_{0}!\left((v-v_{0})!\right)^{2}}\cdot\exp\left(O(v)\right)\left(\frac{c_{0}}{v}\right)^{e_{0}}\cdot\frac{1}{n^{v_{0}-\alpha e_{0}}}\\
 & \leq & \frac{\left(\left(v\right)_{v_{0}}\right)^{2}}{\left(\frac{v_{0}}{{\rm e}}\right)^{v_{0}}}\cdot\exp\left(O(v)\right)\left(\frac{c_{0}}{v}\right)^{e_{0}}\cdot\frac{1}{n^{v_{0}-\alpha e_{0}}}\\
 & \leq & \frac{v^{2v_{0}}}{v_{0}^{v_{0}}}\exp\left(O(v)\right)\cdot\left(\frac{c_{0}}{v}\right)^{e_{0}}\cdot\frac{1}{n^{v_{0}-\alpha e_{0}}}.
\end{eqnarray*}
Notice that the last expression is monotonically increasing with respect
to $e_{0}$. Therefore we may bound it for every $2\leq v_{0}\leq v$
only with $e_{0}^{*}$, defined as the maximal $e_{0}$ such that
$(v_{0},e_{0})\in\mathcal{I}$. Note that $e_{0}^{*}=\frac{1}{\alpha}v_{0}+O(1)$,
so
\begin{equation}
\frac{S_{v_{0},e_{0}}}{[\mathbb{E}\left(\mathbf{X}_{\mathcal{F}}\right)]^{2}}=\underset{(*)}{\underbrace{\left(\frac{v^{2-\frac{1}{\alpha}}}{v_{0}}\right)^{v_{0}}\exp\left(O(v)\right)}}\cdot\underset{(**)}{\underbrace{\frac{1}{n^{v_{0}-\alpha e_{0}^{*}}}}}.\label{eq:S_v0e0 bound}
\end{equation}
We bound (\ref{eq:S_v0e0 bound}) separately for small $v_{0}$ and
for large $v_{0}$. Intuitively, for small subgraphs the enhanced
balancedness property promises that $(**)$ is sufficiently small,
while for large subgraphs it is $(*)$ which becomes small.

\textbf{Small Subgraphs. }In this case we assume $v_{0}\leq\frac{c}{\omega}v$
where $c=\frac{\alpha\beta_{0}}{2}$. By definition of $v$,
\[
v_{0}\leq\frac{c}{\omega}\left\lfloor \omega\frac{\ln n}{\ln\ln n}\right\rfloor \leq c\frac{\ln n}{\ln\ln n}.
\]
Of course, this implies $v_{0}\leq\delta_{0}v$, so enhanced balancedness
applies: $e_{0}^{*}\leq\frac{1}{\alpha}v_{0}-\beta_{0}$ and therefore
$\frac{1}{n^{v_{0}-\alpha e_{0}^{*}}}\leq n^{-\alpha\beta_{0}}$. 

As for $(*)$, we use the simple bound
\[
(*)\leq v^{v_{0}}\left[\left(\frac{v}{v_{0}}\right)^{\frac{v_{0}}{v}}\right]^{v}\exp\left(O(v)\right)\leq \exp\left(v_{0}\ln v+O(v)\right).
\]
Without loss of generality, we may assume $\omega=o(\ln\ln n)$, so
$v=o(\ln n)$ and
\begin{eqnarray*}
\exp\left(v_{0}\ln v+O(v)\right)  \leq  \exp\left(c\frac{\ln n}{\ln\ln n}\ln\ln n+o(\ln n)\right)=n^{\frac{\alpha\beta_{0}}{2}+o(1)}.
\end{eqnarray*}
Overall, $S_{v_{0},e_{0}}/\mathbb{E}[\left(\mathbf{X}_{\mathcal{F}}\right)]^{2}\leq n^{-\frac{\alpha\beta_{0}}{2}+o(1)}$
uniformly. In particular it is $o\left(v^{-2}\right)$. 

\textbf{Large Subgraphs. }Now assume $v_{0}\geq\frac{c}{\omega}v$.
In this case we only know $e_{0}^{*}\leq\frac{e}{v}v_{0}\leq\left(\frac{1}{\alpha}+\frac{2}{v}\right)v_{0}$,
so 
\[
(**)=\frac{1}{n^{v_{0}-\alpha e_{0}^{*}}}\leq n^{2\alpha\frac{v_{0}}{v}}.
\]
As for $(*)$, we have
\[
\left(\frac{v^{2-\frac{1}{\alpha}}}{v_{0}}\right)^{v_{0}}\leq\left(\frac{\omega}{c}v^{1-\frac{1}{\alpha}}\right)^{v_{0}}=\left(\frac{\omega}{c}v^{-\xi}\right)^{v_{0}}
\]
where $\xi=\frac{1}{\alpha}-1$ is a positive constant. Combining
both bounds,
\[
\frac{S_{v_{0},e_{0}}}{\mathbb{E}[\left(\mathbf{X}_{\mathcal{F}}\right)]^{2}}\leq\left(\omega v^{-\xi}\cdot C^{\frac{v}{v_{0}}}n^{\frac{2\alpha}{v}}\right)^{v_{0}}
\]
for some constant $C>1$. Note that 
$$
\omega C^{\frac{v}{v_{0}}}n^{\frac{2\alpha}{v}}  \leq\exp\left(O(\omega)+O\left(\frac{\ln n}{v}\right)\right)\leq  \exp\left(O(\omega)+O\left(\frac{\ln\ln n}{\omega}\right)\right).
$$
Again, without loss of generality we may assume $\omega=o(\ln\ln n)$,
so
\[
\omega C^{\frac{v}{v_{0}}}n^{\frac{2\alpha}{v}}=\exp\left(o(\ln\ln n)\right)=v^{o(1)}.
\]
Overall, $S_{v_{0},e_{0}}/\mathbb{E}[\left(\mathbf{X}_{\mathcal{F}}\right)]^{2}\leq\left(v^{-\xi+o(1)}\right)^{v_{0}}$
which is again $o(v^{-2})$.  We have therefore completed the proof of Lemma \ref{lem:property_A}.
\end{proof}

\subsection{Proof of Proposition \ref{prop:enhanced_balanceness}\label{subsec:enhanced_balanceness_proof}}

It remains to prove Proposition \ref{prop:enhanced_balanceness} about
enhanced balancedness. We have ${\bf H}={\bf H}(v,e)$ with $v\rightarrow\infty$
and $e=\left\lceil \frac{1}{\alpha}v\right\rceil +1$. 
Write $\frac{1}{\alpha}=\frac{q}{2}+\xi$ where $q\geq2$ is an integer
and $0<\xi<\frac{1}{2}$. Then $e=q\cdot\frac{v}{2}+r$ where $r=\left\lceil \xi v\right\rceil +1<\frac{v}{2}$. 

\begin{definition}
Let ${\bf H}={\bf H}_{r}+{\bf H}_{h}+{\bf H}_{b}$ be a random balanced
(simple) graph. Let ${\bf H}'={\bf H}_{h}+{\bf H}_{b}$ be
obtained by deleting the regular edges from ${\bf H}$. Let ${\bf H}''$ be 
obtained from ${\bf H}'$ by replacing every balancing edge $\{x,y[x]\}$
with a loop at $x$. That is, ${\bf H}''$ is simply the Hamilton
cycle ${\bf H}_{h}$ with a loop at every vertex of~${\bf R}$. 
\end{definition}

\begin{definition}
For a given set of vertices $V_{0}\subseteq[v]$, its \emph{segment
decomposition }is $V_{0}=\tau_{1}\sqcup\dots\sqcup\tau_{t}$, where
$\tau_{1},\dots,\tau_{t}$ are mutually disjoint segments of the Hamiltonian
component ${\bf H}_{h}$ with no Hamiltonian edge between them. 
\end{definition}

We shall use the following technical claim. 

\begin{claim}
\label{claim:enhanced_balanceness_key}For every positive integer
$M$ there exists a positive constant $\beta$ such that the following
holds. Let $V_{0}\subseteq[v]$ be a set of $v_{0}$ vertices and
let $H_{0}''={\bf H}''[V_{0}]$. If the segment decomposition of $V_{0}$
contains a segment of length at most $M$, then
\[
e(H_{0}'')\leq\left(1+\xi\right)v_{0}+\frac{2v_{0}}{v}-\beta.
\]
\end{claim}

\begin{proof}

Fix a positive integer $M$. Define 
\[
\beta=\min_{v_{0},e_{0}}\left((1+\xi)v_{0}-e_{0}\right)
\]
where the minimum is over all integers $v_{0},e_{0}$ with $1\leq v_{0}\leq M$
and $1\leq e_{0}<\left(1+\xi\right)v_{0}$. This is a minimum over
a finite, fixed set and thus $\beta$ is a positive constant. Note
that $\beta<1$ by definition.

First let $V_{0}=\tau\subseteq[v]$ be a segment with $v_{0}\leq M$
vertices. Let $e_{0}$ be the number of edges in the induced subgraph
${\bf H}''[\tau]$. Recall that ${\bf H}''$ is strictly balanced
and its density is $1+\frac{r}{v}$. Also recall that $r=\left\lceil \xi v\right\rceil +1$,
therefore
\[
1+\frac{r}{v}\in\left[(1+\xi)+\frac{1}{v},(1+\xi)+\frac{2}{v}\right].
\]
Since $M$ is a constant, for a sufficiently large $v$, the inequality
$e_{0}<\left(1+\frac{r}{v}\right)v_{0}$ implies $e_{0}<(1+\xi)v_{0}$
(whenever $1\leq v_{0}\leq M$). In turn, $e_{0}<(1+\xi)v_{0}$ implies
$e_{0}\leq(1+\xi)v_{0}-\beta$ by the definition of $\beta$. 

Now let $\tau$ be a segment of any length. Again, let $v_{0}$ be its
number of vertices and let $e_{0}$ be the number of edges in ${\bf H}''[\tau]$.
In this case we can only use strict balancedness to claim that 
\[
e_{0}<\left(1+\frac{r}{v}\right)v_{0}<\left(1+\xi\right)v_{0}+\frac{2v_{0}}{v}.
\]
Finally, let $V_{0}\subseteq[v]$ be an arbitrary set of vertices.
Let $V_{0}=\tau_{1}\cup\dots\cup\tau_{t}$ be the segment decomposition.
By definition, ${\bf H}''$ contains no edges between any pair $\tau_{i},\tau_{j}$.
Let $v_{0},e_{0}$ denote the number of vertices and edges in ${\bf H}''[V_{0}]$
and let $v_{i},e_{i}$ denote the number of vertices and edges in
${\bf H}''[\tau_{i}]$. If at least one of the segments is of length
$\leq M$, 
\[
e_{0}=\sum_{i=1}^{t}e_{i}<\sum_{i=1}^{t}\left[(1+\xi)v_{i}+\frac{2v_{i}}{v}\right]-\beta=\left(1+\xi\right)v_{0}+\frac{2v_{0}}{v}-\beta.
\]
That finishes the proof of the claim.
\end{proof}

Let us now finish the proof of Proposition \ref{prop:enhanced_balanceness}. Set $M=\left\lceil \frac{4}{\xi}\right\rceil $ and let $\beta$
be the corresponding constant from Claim \ref{claim:enhanced_balanceness_key}.
We prove Proposition \ref{prop:enhanced_balanceness} for $\beta_{0}=\frac{1}{2}\beta$
and $\delta_{0}=\frac{1}{4}\min\left\{ \beta,{\rm e}^{-\frac{2}{\xi}}\right\} $.
These choices will be justified soon; for now, notice that $\beta_{0}+2\delta_{0}\leq\beta<1$. 

Let $V_{0}\subseteq[v]$ be a subset with $\left|V_{0}\right|=v_{0}\leq\delta_{0}v$
and let $H_{0}={\bf H}[V_{0}]$. We begin by identifying several cases
in which the desired inequality $e(H_{0})<\frac{1}{\alpha}v_{0}-\beta_{0}$
holds. The \emph{lost edges} of $V_0$
are the balancing edges $\{x,{\bf y}[x]\}$ with $x\in V_{0}$ and
${\bf y}[x]\not\in V_{0}$. 

\textbf{Fact 1.} Suppose that $V_{0}$ has a lost edge.\textbf{ }Then,
from Claim \ref{claim:enhanced_balanceness_key},
\[ 
e(H_{0})\leq\left(\frac{q}{2}-1\right)v_{0}+e\left(H_{0}''\right)-1\leq\frac{1}{\alpha}v_{0}+\frac{2v_{0}}{v}-1<\frac{1}{\alpha}v_{0}-\beta_{0}.
\]
Here, the first inequality follows from that fact that $H_{0}$ is obtained from ${\bf H}'[V_0]$ be adding some edges of the $(q-2)$-regular
graph ${\bf H}_{r}$, and therefore $e(H_{0})  \leq  \left(\frac{q}{2}-1\right)v_{0}+e({\bf H}'[V_0])$, while $H_{0}''$ is obtained from ${\bf H}'[V_0]$ by replacing its balancing edges with loops, and adding a loop for every lost edge of $V_{0}$. 

\textbf{Fact 2. }Suppose that the segment decomposition of $V_{0}$
contains a segment of length at most $M$. Then, from Claim \ref{claim:enhanced_balanceness_key},
\[
e(H_{0})\leq\frac{1}{\alpha}v_{0}+\frac{2v_{0}}{v}-\beta\leq\frac{1}{\alpha}v_{0}+2\delta_{0}-\beta<\frac{1}{\alpha}v_{0}-\beta_{0}.
\]

\textbf{Fact 3.} Denote $r_{0}=\left|V_{0}\cap R\right|$. Then $e(H_{0})\leq\frac{q}{2}v_{0}+r_{0}$.
In particular, if $r_{0}<\xi v_{0}-\beta_{0}$, we again have
\[
e(H_{0})<\frac{1}{\alpha}v_{0}-\beta_{0}.
\]

These three facts show that in order to finish the proof, it is sufficient
to prove the following statement. W.h.p., for every subset $V_{0}\subseteq[v]$
with $\left|V_{0}\right|\leq\delta_{0}v$, if it is composed only
of segments of length at least $M$ and satisfies $\left|V_{0}\cap R\right|\geq\xi v_{0}-\beta_{0}$,
then it has a lost edge.

For $V_{0}\subseteq[v]$ with $v_{0}=\left|V_{0}\right|$ and $r_{0}=\left|V_{0}\cap R\right|$,
the probability that it has no lost edges is precisely $\frac{(v_{0})_{r_{0}}}{(v)_{r_{0}}}$
in ${\bf H}^{*}$, and $\leq\frac{1}{p_{0}}\frac{(v_{0})_{r_{0}}}{(v)_{r_{0}}}$
in ${\bf H}$ (see Proposition \ref{prop:simplicity_probability}).
Assuming $r_{0}\geq\xi v_{0}-\beta_{0}$, 
\[
\frac{(v_{0})_{r_{0}}}{(v)_{r_{0}}}\leq\left(\frac{v_{0}}{v}\right)^{r_{0}}\leq\left(\frac{v_{0}}{v}\right)^{\xi v_{0}-\beta_{0}}.
\]
Now consider the number of subsets $V_{0}\subseteq[v]$ with $v_{0}$
vertices which are composed only of segments of length at least $M$.
It is bounded by the number of subsets which are composed of at most
$\frac{v_{0}}{M}$ segments, which is $O(1)\cdot\binom{v}{2v_{0}\slash M}$.
By the union bound, the probability that no relevant $V_{0}$ has
a lost edge is
\begin{equation}
\mathbb{P}\left({\rm no}\;V_{0}\;{\rm loses}\;{\rm edges}\right)=O(1)\cdot\sum_{v_{0}=M}^{\delta_{0}v}\binom{v}{2v_{0}\slash M}\left(\frac{v_{0}}{v}\right)^{\xi v_{0}-\beta_{0}}.\label{eq:balanceness_sum}
\end{equation}
It remains to show that the last sum is $o(1)$. The standard bound
$\binom{n}{k}\leq\left(\frac{{\rm e}n}{k}\right)^{k}$ yields
\begin{align*}
(\ref{eq:balanceness_sum}) & \leq\sum_{v_{0}=M}^{\delta_{0}v}\left(\frac{{\rm e}v}{2v_{0}\slash M}\right)^{2v_{0}\slash M}\left(\frac{v_{0}}{v}\right)^{\xi v_{0}-\beta_{0}}\\
 & =\sum_{v_{0}=M}^{\delta_{0}v}\left(\frac{{\rm e}M}{2}\right)^{\frac{2v_{0}}{M}}\left(\frac{v_{0}}{v}\right)^{\xi v_{0}-\frac{2}{M}v_{0}-\beta_{0}}
  \overset{(*)}{\leq}\sum_{v_{0}=M}^{\delta_{0}v}{\rm e}^{v_{0}}\left(\frac{v_{0}}{v}\right)^{\frac{\xi}{2}v_{0}-\beta_{0}}.
\end{align*}
The last inequality $(*)$ follows from the inequality $\left({\rm e}x\right)^{\frac{1}{x}}\leq{\rm e}$
(applied for $x=\frac{M}{2}$) and the fact that $\frac{2}{M}<\frac{\xi}{2}$
(which follows from the definition of $M$). 

We divide the above sum into two: the sum over $M\leq v_{0}\leq\sqrt{v}$
and the sum over $\sqrt{v}\leq v_{0}\leq\delta_{0}v$. First,
\begin{multline*}
\sum_{v_{0}=M}^{\sqrt{v}}{\rm e}^{v_{0}}\left(\frac{v_{0}}{v}\right)^{\frac{\xi}{2}v_{0}-\beta_{0}} \leq \sum_{v_{0}=M}^{\sqrt{v}}{\rm e}^{v_{0}}\left(\frac{1}{\sqrt{v}}\right)^{\frac{\xi}{2}v_{0}-\beta_{0}}\\
  =  {\rm e}^{M}\left(\frac{1}{\sqrt{v}}\right)^{\frac{\xi}{2}M-\beta_{0}}\sum_{i=0}^{\sqrt{v}-M}\left({\rm e}v^{-\frac{\xi}{4}}\right)^{i}
  \lesssim   {\rm e}^{M}\left(\frac{1}{\sqrt{v}}\right)^{\frac{\xi}{2}\cdot\frac{4}{\xi}-1}=o(1).
\end{multline*}

Finally,
\begin{multline*}
\sum_{v_{0}=\sqrt{v}}^{\delta_{0}v}{\rm e}^{v_{0}}\left(\frac{v_{0}}{v}\right)^{\frac{\xi}{2}v_{0}-\beta_{0}}  \leq  \sum_{v_{0}=\sqrt{v}}^{\delta_{0}v}{\rm e}^{v_{0}}\delta_{0}^{\frac{\xi}{2}v_{0}}\cdot\left(\frac{v}{v_{0}}\right)^{\beta_{0}}\\
  \leq  v^{\beta_{0}\slash2}\sum_{v_{0}=\sqrt{v}}^{\delta_{0}M}\left[{\rm e}\delta_{0}^{\xi\slash2}\right]^{v_{0}}
  =  v^{\beta_{0}\slash2}\exp\left(-\Theta(\sqrt{v})\right)=o(1)
\end{multline*}
since ${\rm e}\delta_{0}^{\xi\slash2}<1$ by the definition of $\delta_{0}$.
That finishes the proof.

\section{Low upper bounds for well-approximable $\alpha$}
\label{sec:low_upper_bound_well_approximable}

In this section we prove Theorem \ref{thm:arbitrary_low_upper_bounds}.
Recall that for a given function $f(n)\underset{n\rightarrow\infty}{\rightarrow}\infty$
(which can grow arbitrarily slowly), our goal is to find an irrational
$\alpha\in (0,1)$ and an increasing sequence of positive integers $\{n_{t}\}_{t=1}^{\infty}$
such that w.h.p. two independent copies ${\bf G}_{n_t}^{1},{\bf G}_{n_t}^{2}\sim G(n_{t},n_{t}^{-\alpha})$
can be distinguished by a FO sentence of quantifier depth at most
$f(n_{t})$.  For the rest of this section, let us denote these copies by $\mathbf{G}_{1},\mathbf{G}_{2}$ for notational convenience.

The construction of a suitable irrational $\alpha$ is explicit. The
idea is to take $\alpha$ which is very well-approximable by rational numbers;
the slower $f(n)$ grows, the better the approximations must be. Then,
along a subsequence $\left\{ n_{t}\right\} _{t=1}^{\infty}$ (which
is determined by the sequence of approximations of $\alpha$), distinguishing
between ${\bf G}_{1},{\bf G}_{2}$ is essentially the same as in the
rational case. This idea naturally leads to the theory of Diophantine
approximations, and specifically to the concept of \emph{Liouville
numbers.}

\subsection{Diophantine approximations}

The theory of Diophantine approximations studies approximations of
real numbers by rational numbers. One of its well-known applications
is the celebrated Liouville's theorem, which was used to establish the existence
of transcendental numbers for the first time.
A \emph{Liouville number} is an irrational number $x$ such that for every $d\in\mathbb{N}$,
there exists a rational $\frac{p}{q}$ with $q>1$ such that $\left|x-\frac{p}{q}\right|<\frac{1}{q^{d}}$. Liouville's theorem implies that Liouville numbers are transcendental. For a comprehesive survey on this subject, see \cite{Transcendental_number_theory}.
 Liouville provided $x=\sum_{n=1}^{\infty}\frac{1}{2^{n!}}$ as an
example of a Liouville number. For our purposes, however, we will
need Liouville numbers which are much better approximated.
From now on, when we write a rational number as $\frac{p}{q}$, we
always assume that $q>0$ and that ${\rm gcd}(p,q)=1$. 
\begin{definition}
\label{def:phi-approximable}Let $\varphi:\mathbb{N}\rightarrow(0,\infty)$
be a decreasing function. An irrational number $x$ is
called \emph{$\varphi$-approximable} if, for infinitely many
rational numbers $\frac{p}{q}$, $\left|x-\frac{p}{q}\right|<\varphi(q)$. 
\end{definition}

\begin{lemma}
\label{lem:well-approximable-irrationals}For every descreasing function
$\varphi:\mathbb{N}\rightarrow(0,\infty)$ there exists an irrational
$x\in(0,1)$ which is $\varphi$-approximable.
\end{lemma}

\begin{proof}
First notice that if $\varphi,\psi:\mathbb{N}\rightarrow(0,\infty)$
satisfy $\varphi\leq\psi$ and $x$ is $\varphi$-approximable,
then $x$ is also $\psi$-approximable. Therefore, without loss of
generality, we may assume $\varphi<1$. Recursively define a sequence
of natural numbers as follows: $U_{1} = 1$, $U_{t+1} = \left\lceil -\log_{2}\varphi\left(2^{U_{t}}\right)\right\rceil +1$. Now define $x=\sum_{i=1}^{\infty}2^{-U_{i}}$. Without loss
of generality, we may assume that $\varphi$ is decreasing sufficiently
fast such that $U_{t+1}-U_{t}$ is a strictly increasing sequence.
In that case, $x$ is an irrational number because its binary expansion
is aperiodic. For every $t\in\mathbb{N}$, consider the rational approximation
$\sum_{i=1}^{t}2^{-U_{i}}$ of $x$, which can be written
as $\frac{p_{t}}{q_{t}}$ for $q_{t}=2^{U_{t}}$ and $p_{t}=\sum_{i=1}^{t}2^{U_{t}-U_{i}}$. Then
\[
0<x-\frac{p_{t}}{q_{t}}=\sum_{i=t+1}^{\infty}\frac{1}{2^{U_{i}}}\leq\sum_{i=U_{t+1}}^{\infty}\frac{1}{2^{i}}=\frac{1}{2^{U_{t+1}-1}}\leq\varphi\left(2^{U_{t}}\right)=\varphi(q_{t}).
\]
\end{proof}
As in Section \ref{sec:general_upper_bound}, we distinguish between
two independent ${\bf G}_{1},{\bf G}_{2}\sim G(n_{t},n_{t}^{-\alpha})$
through their subgraphs; however, some details are different.
First, here we count (all) subgraphs instead
of induced subgraphs. We consider subgraphs with densities that approximate $\frac{1}{\alpha}$ well so that the
variables counting copies of these subgraphs asymptotically behave
like independent Poisson variables (as in the case of a rational $\alpha$).
In particular, a single subgraph suffices to distinguish
${\bf G}_{1}$ from ${\bf G}_{2}$ with positive probability. To distinguish
w.h.p., we shall consider $\ell\rightarrow\infty$
different subgraphs and approximate their joint distribution. 

\subsection{Asymptotic Poisson behavior}

We now return to the proof of Theorem \ref{thm:arbitrary_low_upper_bounds}. 
 Let $g:\mathbb{N}\rightarrow\mathbb{N}$ be the inverse of $f$, that is, $g(k)=\min\left\{ n\in\mathbb{N}:f(n)\geq k\right\}$. Without loss of generality, (1) $f:\mathbb{N}\to\mathbb{N}$ is non-decreasing and surjective, (2) $f(n)=o\left(\sqrt{\ln n}/\ln\ln n\right)$, (3) $n/g(n)$ strictly decreases with $n$. Let $\alpha\in(0,1)$ be an irrational number which is $\varphi$-approximable
for $\varphi(q)=1/g(q)$; its existence is guaranteed by Lemma
\ref{lem:well-approximable-irrationals}. Let $\left\{ p_{t}\right\} _{t=1}^{\infty},\left\{ q_{t}\right\} _{t=1}^{\infty}$
be the suitable sequences from the proof of the lemma: increasing
sequences of natural numbers such that, for every $t\in\mathbb{N}$, $p_t,q_t$ are coprime and
\[
0<\alpha-\frac{p_{t}}{q_{t}}\leq\frac{1}{g(q_{t})}.
\]
Write $v_{t}=p_{t}$ and $e_{t}=q_{t}$ and also define $n_{t}=g(v_{t})$.
Our first assumption on $f$ implies that $v_{t}=f(n_{t})$. From
now on we focus on the subsequence $\{n_{t}\}_{t=1}^{\infty}$, so
the underlying parameter is now $t$. For convenience, we often omit
the dependency on $t$ from the notation; all quantities implicitly
depend on $t$ unless we explicitly state that they are fixed.

Let $\mathcal{H}=\mathcal{H}(v,e)$ be the set of graphs from Theorem \ref{thm:family_of_graphs}.
Recall that $\mathcal{H}$ is a family of asymmetric graphs which is
closed under isomorphism. Therefore it contains $\ell=\frac{1}{v!}\left|\mathcal{H}\right|$
isomorphism classes. From the asymptotic estimation on $\left|\mathcal{H}\right|$
given in Theorem \ref{thm:family_of_graphs} it follows that $\ell\underset{t\rightarrow\infty}{\rightarrow}\infty$.
Let $H^{(1)},H^{(2)},\dots,H^{(\ell)}$ be representatives of the
$\ell$ isomorphism classes. For every $1\leq i\leq\ell$ let $\mathbf{X}^{(i)}$
count copies of $H^{(i)}$ in ${\bf G}_{1}$ and let $\mathbf{Y}^{(i)}$ count
copies of $H^{(i)}$ in ${\bf G}_{2}$. In order to complete the proof,
it is sufficient to prove the following.
\begin{prop}
\label{prop:distnguish_H^(i)}W.h.p.\ as $t\rightarrow\infty$, there
exists $1\leq i\leq\ell$ such that either $\mathbf{X}^{(i)}\geq1$ and $\mathbf{Y}^{(i)}=0$,
or $\mathbf{X}^{(i)}=0$ and $\mathbf{Y}^{(i)}\geq1$.
\end{prop}

We postpone the proof of Proposition~\ref{prop:distnguish_H^(i)} to the very end of this section since first we have to prove several auxiliary assertions.

First of all, let us show that $\mathbf{X}^{(1)},\mathbf{X}^{(2)},\dots,\mathbf{X}^{(\ell)}$
asymptotically behave like independent Poisson random variables. We
do that using the method of moments (see \cite{rucinski_luczak_janson},
Subsection 6.1). Note that the method of moments applies only when
the number of variables is fixed, so we start by fixing some $k\in\mathbb{N}$
and considering only the first $k$ variables $\mathbf{X}^{(1)},\mathbf{X}^{(2)},\dots,\mathbf{X}^{(k)}$
(assuming that the underlying $t$ is sufficiently large such that
$k\leq\ell$).

Recall that for summations of Bernoulli random variables, it is easier
to deal with factorial moments rather than usual moments. For a random
variable $\mathbf{Z}$, its $m$-th factorial moment is $\mathbb{E}\left((\mathbf{Z})_{m}\right)$,
where $(\mathbf{Z})_{m}=\mathbf{Z}(\mathbf{Z}-1)\dots(\mathbf{Z}-m+1)$. For a $k$-tuple of variables $\mathbf{Z}_{1},\dots,\mathbf{Z}_{k}$, its $(m_{1},\dots,m_{k})$-th
joint factorial moment equals to $\mathbb{E}\left(\left(\mathbf{Z}_{1}\right){}_{m_{1}}\cdots\left(\mathbf{Z}_{k}\right)_{m_{k}}\right).$
\begin{lemma}
\label{lem:joint_moments}Fix $k\in\mathbb{N}$ and non-negative integers
$m_{1},m_{2},\dots,m_{k}$. Then, with the above definitions, we have
\[
\lim_{t\rightarrow\infty}\mathbb{E}\left(\left(\mathbf{X}^{(1)}\right){}_{m_{1}}\cdots\left(\mathbf{X}^{(k)}\right)_{m_{k}}\right)=1.
\]
\end{lemma}

\begin{proof}
Let us denote $(\vec{\mathbf{X}})_{\vec{m}}=\left(\mathbf{X}^{(1)}\right){}_{m_{1}}\cdots\left(\mathbf{X}^{(k)}\right)_{m_{k}}$ and $m=m_{1}+\dots+m_{k}$.
Then, by decomposing $(\vec{\mathbf{X}})_{\vec{m}}$ into a sum of indicator
random variables, we can write
\begin{equation}
\mathbb{E}(\vec{\mathbf{X}})_{\vec{m}}=\sum_{\vec{H}}\mathbb{P}\left(H_{1}^{(1)},\dots,H_{m_{1}}^{(1)},\dots,H_{1}^{(k)},\dots,H_{m_{k}}^{(k)}\subseteq{\bf G}_{1}\right)\label{eq:factorial_moment_sum}
\end{equation}
where the sum is over all $m$-tuples
\begin{equation}
\vec{H}=\left(H_{1}^{(1)},\dots,H_{m_{1}}^{(1)},\dots,H_{1}^{(k)},\dots,H_{m_{k}}^{(k)}\right)\label{eq:H_tuple}
\end{equation}
such that $H_{1}^{(i)},\dots,H_{m_{i}}^{(i)}$ are distinct copies
of $H^{(i)}$ in the complete graph $K_{n}$ for every $1\leq i\leq k$.
Let us divide the sum in the right hand side of (\ref{eq:factorial_moment_sum})
into two parts: $\mathbb{E}(\vec{\mathbf{X}})_{\vec{m}}=S_{1}+S_{2}$ where
$S_{1}$ is the sum over the $m$-tuples $\vec{H}$ which do not share
any vertices with each other, and $S_{2}$ is the remaining part of
the sum.

We first estimate $S_{1}$. Since $H^{(1)},\dots,H^{(k)}$ are asymmetric,
it is easy to see that
\[
S_{1}=(n)_{mv}p^{me}\sim n^{mv}p^{me}=(n^{v-\alpha e})^{m}.
\]
By the definition of $v,e$, we have $\left|\alpha-v/e\right|\leq 1/g(e)$,
which can also be written as $\left|v-\alpha e\right|\leq e/g(e)$.
Since $n/g(n)$ decreases and $v<e$, we get
$\left|v-\alpha e\right|\leq v/g(v)=f(n)/n$.
Therefore,
\[
S_{1}\sim\exp\left(\ln n\cdot O\left(\frac{f(n)}{n}\right)\right)=\exp(o(1))=1+o(1).
\]

It remains to prove $S_{2}=o(1)$. Let $\mathcal{F}^*$ be the family of
all graphs (up to isomorphism) which are the result of a union of
$m_{i}$ distinct copies of $H^{(i)}$ for every $1\leq i\leq m$,
with at least one shared vertex. For every $F\in\mathcal{F}^*$, let $S_{F}$
be the sum from (\ref{eq:factorial_moment_sum}) but only over the
$m$-tuples $\vec{H}$ whose union is isomorphic to $F$. Then $S_{2}=\sum_{F\in\mathcal{F}^*}S_{F}$. 

Fix $F\in\mathcal{F}^*$. Let $C_{F}$ denote the number of $m$-tuples
$\vec{H}$ as in (\ref{eq:H_tuple}) such that $H_{1}^{(i)},\dots,H_{m_{i}}^{(i)}$ are distinct copies of $H^{(i)}$ and $H_{1}^{(1)}\cup\dots\cup H_{m_{k}}^{(k)}=F$. Then 
\[
S_{F}=\binom{n}{v(F)}C_{F}p^{e(F)}\leq C_{F}n^{v(F)}p^{e(F)}=C_{F}n^{v(F)-\alpha e(F)}.
\]
Since $v(H^{(i)})-\alpha e(H^{(i)})<0$ for all $i\in[k]$, then following the usual argument that is used to prove that the union of intersecting strictly balanced graphs with the same density has higher density, we derive the following claim.
\begin{claim}
\label{claim:union_density}For every $F\in\mathcal{F}^*$, the inequality $v(F)-\alpha e(F)\leq-\frac{c}{v}$ holds for any constant $c<\alpha$.
\end{claim}

\begin{proof}

For a graph $F$ we define $\xi(F)=v(F)-\alpha e(F)$. We begin with
a few simple observations.

\emph{Observation 1.}\textbf{ }If $H$ is isomorphic to one of $H^{(1)},\dots,H^{(k)}$,
we have $\xi\left(H\right)=v-\alpha e<0$. This is because $\frac{v}{e}<\alpha$
by definition. 

\emph{Observation 2. }For every two graphs $F_{1},F_{2}$,
\[
\xi(F_{1}\cup F_{2})=\xi(F_{1})+\xi(F_{2})-\xi(F_{1}\cap F_{2}).
\]

\emph{Observation 3.}\textbf{ }Suppose $H$ is isomorphic to one of
$H^{(1)},\dots,H^{(k)}$ and let $H_{0}\subset H$ be a proper subgraph
with $v_{0}$ vertices and $e_{0}$ edges. Since $H$ is strictly
balanced, we have $\frac{e_{0}}{v_{0}}<\frac{e}{v}.$ Furthermore,
\[
\frac{e}{v}-\frac{e_{0}}{v_{0}}=\frac{ev_{0}-e_{0}v}{v_{0}v}\geq\frac{1}{v_{0}v}.
\]
Denote $\varepsilon=\frac{e}{v}-\frac{1}{\alpha}$ (which is positive).
Then 
\begin{eqnarray*}
\frac{1}{\alpha}+\varepsilon-\frac{e_{0}}{v_{0}} & \geq & \frac{1}{v_{0}v},\\
v_{0}+\alpha\varepsilon v_{0}-\alpha e_{0} & \geq & \frac{\alpha}{v},\\
\xi(H_{0})=v_{0}-\alpha e_{0} & \geq & \frac{\alpha}{v}-\alpha\varepsilon v_{0}.
\end{eqnarray*}
Since $\varepsilon=O\left(\frac{1}{g(v)}\right)$, we get that $\xi(H_{0})\geq\frac{c}{v}$
for any constant $c<\alpha$ (of course, this becomes true when the
underlying $t$ is sufficiently large). 

With these observations, we now prove Claim \ref{claim:union_density}
by induction on $m\geq2$.

We start with the induction base $m=2$. Suppose $F=H_{1}\cup H_{2}$
where $H_{1},H_{2}$ are both from $\mathcal{H}$ and share a vertex.
Then
\[
\xi(F)=\xi(H_{1})+\xi(H_{2})-\xi(H_{1}\cap H_{2})<-\xi(H_{1}\cap H_{2}).
\]
$H_{1}\cap H_{2}$ is isomorphic to some proper subgraph $H_{0}$
of $H_{1}$, therefore
\[
\xi(F)<-\xi(H_{0})\leq-\frac{c}{v}.
\]
Now, suppose the statement is true for $m\geq2$ and prove it for
$m+1$. Let $F=\bigcup_{i=1}^{m+1}H_{i}$ be the union of certain
copies of graphs from $\mathcal{H}$ such that, without loss of generality,
$H_{1},H_{2}$ share a vertex. Let $F'=\bigcup_{i=1}^{m}H_{i}$. Then
\begin{eqnarray*}
\xi(F) & = & \xi(F')+\xi(H_{m+1})-\xi(F'\cap H_{m+1}).
\end{eqnarray*}
$F'\cap H_{m+1}$ is isomorphic to a subgraph $H_{0}$ of $H_{m+1}$
(not necessarily proper this time). If $H_{0}=H_{m+1}$ then $\xi(F)=\xi(F')$.
If $H_{0}\subset H_{m+1}$ then the third observation shows that $\xi(H_{0})>0$,
but $\xi(H_{m+1})<0$ and therefore $\xi(F)<\xi(F')$. From the inductive
assumption, in both cases $\xi(F)\leq\xi(F')<-\frac{c}{v}$ and that
finishes the proof.

\end{proof}

From Claim \ref{claim:union_density}, we can write $S_{F}\leq C_{F}n^{-\frac{c}{v}}$
for any constant $c<\alpha$. Finally, $\sum_{F\in\mathcal{F}^*}C_{F}$
is trivially bounded by $\left((mv)_{v}\right)^{m}$. Overall
\[
S_{2}\leq(mv)^{mv}n^{-\frac{c}{v}}=\exp\left(\Theta\left(v\ln v\right)-\Theta\left(\frac{\ln n}{v}\right)\right)=o(1)
\]
where the last estimation follows from the assumption $f(n)=o\left(\frac{\sqrt{\ln n}}{\ln\ln n}\right)$.
That finishes the proof.
\end{proof}

From Lemma \ref{lem:joint_moments} and the method of moments~\cite[Theorem 6.2]{rucinski_luczak_janson} we get
the following corollary.
\begin{corollary}
\label{cor:poisson_convergence}Fix $k\in\mathbb{N}$; then $\vec{\mathbf{X}}=(\mathbf{X}^{(1)},\dots,\mathbf{X}^{(k)})$
converges in distribution to $\vec{\mathbf{P}}=(\mathbf{P}^{(1)},\dots,\mathbf{P}^{(k)})$ where
$\mathbf{P}^{(1)},\dots,\mathbf{P}^{(k)}\sim{\rm Pois}(1)$ are independent. 
\end{corollary}

We are now ready to prove Proposition \ref{prop:distnguish_H^(i)}.
\begin{proof}[Proof of Proposition \ref{prop:distnguish_H^(i)}]
Let $\vec{\mathbf{I}}=\left(\mathds{1}_{\mathbf{X}^{(1)}\geq 1},\dots,\mathds{1}_{\mathbf{X}^{(\ell)}\geq 1}\right)$ and $\vec{\mathbf{J}}=\left(\mathds{1}_{\mathbf{Y}^{(1)}\geq 1},\dots,\mathds{1}_{\mathbf{Y}^{(\ell)}\geq 1}\right)$.
We need to prove $\mathbb{P}\left(\vec{\mathbf{I}}=\vec{\mathbf{J}}\right)=o(1)$.  For every fixed $k\in\mathbb{N}$, let $\vec{\mathbf{I}}_{k}$  and $\vec{\mathbf{J}}_{k}$ consist of the first $k$ coordinates of $\vec{\mathbf{I}}$ and $\vec{\mathbf{J}}$ respectively. Then
$\mathbb{P}\left(\vec{\mathbf{I}}=\vec{\mathbf{J}}\right)\leq\mathbb{P}\left(\vec{\mathbf{I}}_{k}=\vec{\mathbf{J}}_{k}\right)$. From Corollary \ref{cor:poisson_convergence}, $\mathbb{P}\left(\vec{\mathbf{I}}_{k}=\vec{\mathbf{J}}_{k}\right)\rightarrow\lambda^{k}$
where 
\[
\lambda=\mathbb{P}\left({\rm Pois}(1)=0\right)^{2}+\mathbb{P}\left({\rm Pois}(1)\geq1\right)^{2}={\rm e}^{-2}+\left(1-{\rm e}^{-1}\right)^{2}.
\]
Since this is true for every $k\in\mathbb{N}$, we deduce $\mathbb{P}\left(\vec{\mathbf{I}}=\vec{\mathbf{J}}\right)=o(1)$
and that finishes the proof.
\end{proof}

\section{A lower bound for almost every $\alpha$}
\label{sec:Lower-bound}

In this section we prove Theorem \ref{thm:almost_uniform_lower_bound}.
That is, we fix $d_0>2$ and an irrational $\alpha\in(0,1)$ with Liouville-Roth irrationality measure strictly smaller than $d_0$ and show that w.h.p.\ there is no FO-sentence of quantifier depth less than $\frac{1}{\ln d_0}\ln\ln\ln n$ distinguishing between two independent copies ${\bf G}_{n}^{1},{\bf G}_{n}^{2}\sim G(n,n^{-\alpha})$. To prove this, we present a generalization of  Theorem~\ref{thm:SS_0-1}. The proof of the zero-one law for irrational $\alpha$ relies on the existence of a winning strategy of the second
player in the Ehrenfeucht-Fra\"{i}ss\'{e} game with a bounded number
of rounds. We show that a similar winning strategy can be used by the second player when the number of rounds grows with $n$ sufficiently slowly, provided that $\alpha$ is not too well-approximable by rational numbers. 

From now on we fix $\alpha\in(0,1)$ with Liouville-Roth irrationality measure strictly smaller than $d_0$. Let $d\in(2,d_0)$ be such that $\alpha$ is not $q^{-d}$-approximable, that is,   $\left| x-\frac{p}{q} \right|\geq \frac{1}{q^d}$ for all but finitely many~$\frac{p}{q}$. 

\subsection{The Ehrenfeucht-Fra\"{i}ss\'{e} game}

We begin with a brief overview of the Ehrenfeucht-Fra\"{i}ss\'{e}
game (see Chapter 2 of \cite{spencer} for a more detailed exposition). Given two graphs $G_{1},G_{2}$ (with disjoint vertex sets) and $k\in\mathbb{N}$,
the Ehrenfeucht-Fra\"{i}ss\'{e} game ${\rm EF}(G_{1},G_{2};k)$
is described as follows. The game has two players, called Spoiler
and Duplicator, and consists of $k$ rounds. The graphs $G_{1},G_{2}$
are the ``board'' on which the players make their moves. In the
$i$-th round, Spoiler selects a vertex in either graph (to his choice)
and marks it $i$. Duplicator responds by selecting a vertex in the
other graph, and also marks it $i$. At the end of the game, let $x_{1},\dots,x_{k}$
be the vertices of $G_{1}$ marked $1,\dots,k$ (regardless of who
marked them) and let $y_{1},\dots,y_{k}$ be the vertices of $G_{2}$
marked $1,\dots,k$. Duplicator wins if there exists a partial isomorphism
from $G_{1}$ to $G_{2}$ which maps $x_{i}$ to $y_{i}$ for every
$1\leq i\leq k$. That is, Duplicator wins if $x_{i}=x_{j}\iff y_{i}=y_{j}$
and $x_{i}\sim x_{j}\iff y_{i}\sim y_{j}$ for every $1\leq i,j\leq k$.
The importance of the Ehrenfeucht-Fra\"{i}ss\'{e} game comes from
the following key result, relating it to the FO distinguishability.
\begin{theorem}[Ehrenfeucht~\cite{Ehrenfeucht}]
Let $G_{1},G_{2}$ be two graphs and let $k\in\mathbb{N}$. Then
Duplicator has a winning strategy in ${\rm EF}(G_{1},G_{2};k)$ $\iff$
$k(G_{1},G_{2})>k$.
\end{theorem}

From this result we see that, in order to prove Theorem \ref{thm:almost_uniform_lower_bound},
it suffices to prove that
w.h.p.\ Duplicator has a winning strategy in ${\rm EF}\left({\bf G}_{n}^{1},{\bf G}_{n}^{2};k=\frac{1}{\ln d_0}\ln\ln\ln n\right)$.
Below, we show that the winning strategy of Duplicator
introduced by Shelah and Spencer ({\it look-ahead strategy}) can be used for $k$ rounds as well. We present the strategy and explain the main properties of the random graph that allow Duplicator to use it. Proofs of these properties are postponed to Appendix~\ref{AppendixG}
 since they resemble the proof in the case of constant number of rounds. 
For the original argument, we refer to \cite[Chapters 4--6]{spencer}.

\subsection{Extensions and the closure}

The look-ahead strategy is based on the validity of certain $\forall\exists$-sentences known as extension statements. The original
argument considers statements with a constant number of variables,
but now we allow it to grow with $n$. We begin by generalizing the
concept of safe extensions and their main feature: they w.h.p.\ exist over any tuple of root vertices. 

A \emph{rooted graph} is a graph $H$ with a designated subset of
root vertices $R$. It is denoted by the pair $(R,H)$. We allow $R=\emptyset$
but not $R=V(H)$. Rooted graphs can be used to define graph extensions as follows. Let
$(R,H)$ be a rooted graph and label its vertices $a_{1},\dots,a_{r},b_{1},\dots,b_{v}$
where $a_{i}$ are the roots and $b_{j}$ are the non-roots. Let $G$
be any graph and let $\vec{x}=(x_{1},\dots,x_{r})$ be an $r$-tuple
of distinct vertices of $G$. An $(R,H)$\emph{-extension} of $\vec{x}$
is defined as a $v$-tuple $\vec{y}=(y_{1},\dots,y_{v})$ of distinct
vertices of $G$ such that: (1) $x_{i}\sim y_{j}$ in $G$ whenever $a_{i}\sim b_{j}$ in $H$; (2) $y_{i}\sim y_{j}$ in $G$ whenever $b_{i}\sim b_{j}$ in $H$.  Note that this definition does not consider edges between root vertices,
and allows for additional edges except those specified by $H$. Also
note that the definition implicitly assumes an underlying labeling
of the vertices of $H$. From now on, whenever a rooted graph $(R,H)$
is introduced, we always implicitly assume that it is equipped with
a predetermined labeling.

A rooted graph has three parameters: the number of root vertices $r$,
the number of non-root vertices $v$, and the number of edges (excluding
edges between root vertices) $e$. We call $(v,e)$ the \emph{type}
of the rooted graph. The following definitions depend on the fixed
$\alpha$, and are designed to be used in studying extensions in ${\bf G}\sim G(n,n^{-\alpha})$. 
\begin{definition}
Let $(R,H)$ be a rooted graph of type $(v,e)$. If $v-\alpha e>0$
we say that $(R,H)$ is \emph{sparse}. If $v-\alpha e<0$ we say that
$(R,H)$ is \emph{dense}.
\end{definition}

Since $\alpha$ is irrational, every rooted graph is either sparse
or dense.
\begin{definition}
Let $(R,H)$ be a rooted graph and let $R\subsetneq S\subseteq V(H)$.
We call $\left(R,H[S]\right)$ a \emph{subextension} of $(R,H)$.
We also call $(S,H)$ a \emph{nailextension }of $(R,H)$. 
\end{definition}

Note that $(R,H)$ is always a subextension and a nailextension of
itself.

\begin{definition}
A rooted graph $(R,H)$ is called \emph{safe} if all its subextensions
are sparse, and \emph{rigid} if all its nailextensions are dense.
\end{definition}

The next step is to generalize the Finite Closure Theorem~\cite[Theorem 4.3.2]{spencer} (we rename it as Finite Closure Lemma). We first recall the definition of the {\it $t$-closure}
of a set of vertices.
\begin{definition}
\label{def:t-closure}For a graph $G$, a subset $U\subseteq V(G)$, and $t\in\mathbb{Z}_{\geq 0}$, we define \emph{the $t$-closure
of $U$}, denoted ${\rm cl}_{t}(U)$, as follows: ${\rm cl}_{t}(U)$
is the minimum set of vertices which contains $U$ and is closed with
respect to taking rigid extensions with at most $t$ non-roots. 
\end{definition}

Now consider ${\bf G}\sim G(n,n^{-\alpha}).$ In what follows, we deal with rooted graphs $(R,H)$ of type $(v,e)$ with $r$ roots. Importantly, we now allow
$(R,H)$ to depend on $n$, so $r=r(n)$, $v=v(n)$
and $e=e(n)$. As we shall see, the original argument can be stretched
up to $r,v,e$ which grow as $(\ln n)^{1/(10d)}$. Thus, we denote
$M=M(n)=(\ln n)^{1/(10d)}$ and always assume that $r,v,e=O(M)$.

\begin{lemma}[Bounded Closure Lemma]
There exists a constant $C$
such that w.h.p.\ in ${\bf G}$, for every $r,t\leq M$,
the $t$-closure of every $r$-tuple $\vec{x}$ has size at most $r+Crt^{d}$.
\end{lemma}

Finally, we generalize \cite[Theorem 5.3.1]{spencer} about the
existence of {\it generic extensions}. 
\begin{definition}
\label{def:generic}Let $(R,H)$ be a rooted graph of type $(v,e)$
with $r$ roots. Fix a non-negative integer $t$. For tuples $\vec{x}=(x_{1},\dots,x_{r})$
and $\vec{y}=(y_{1},\dots,y_{v})$, we say that an $(R,H)$-extension
$\vec{y}$ of $\vec{x}$ is \emph{$t$-generic} if: (1) There are no additional edges between the vertices of $\vec{y}$ or
from $\vec{y}$ to $\vec{x}$ other than those specified by $H$. (2) If any $\vec{z}=(z_{1},\dots,z_{s})$ with $s\leq t$ forms a rigid
extension over $\vec{x}\cup\vec{y}$, then there are no edges between
$\vec{z}$ and $\vec{y}$.
\end{definition}

\begin{lemma}
\label{lem:generic_extensions}
W.h.p.\ in ${\bf G}$, for every safe rooted graph $(R,H)$
with $r\leq M$ roots and $v\leq M$ non-roots and for every $t\leq M$,
we have that every $r$-tuple of vertices $\vec{x}$ has a $t$-generic
$(R,H)$-extension $\vec{y}$. 
\end{lemma}

\subsection{Look-ahead strategy}

The look-ahead strategy guarantees a win for Duplicator
in the Ehrenfeucht-Fra\"{i}ss\'{e} game. The previous results allow
us to apply the look-ahead strategy in ${\rm EF}({\bf G}_{n}^{1},{\bf G}_{n}^{2},\frac{1}{\ln d_0}\ln\ln\ln n)$
where ${\bf G}_{n}^{1},{\bf G}_{n}^{2}\sim G(n,n^{-\alpha})$ are independent. Let us recall relevant definitions.
\begin{definition}
Let $\vec{x}=(x_{1},\dots,x_{r})$ be a tuple of vertices in a graph
$G_{1}$ and let $\vec{y}=(y_{1},\dots,y_{r})$ be a tuple of vertices
in a graph $G_{2}$. Also let $t\in\mathbb{Z}_{\geq 0}$. We say
that the $t$-closures ${\rm cl}_{t}(\vec{x})$ and ${\rm cl}_{t}(\vec{y})$
are isomorphic, and write ${\rm cl}_{t}(\vec{x})\cong{\rm cl}_{t}(\vec{y})$,
if there exists a graph isomorphism from ${\rm cl}_{t}\left(\left\{ x_{1},\dots,x_{r}\right\} \right)$
to ${\rm cl}_{t}\left(\left\{ y_{1},\dots,y_{r}\right\} \right)$
that sends $x_{i}$ to $y_{i}$ for every $1\leq i\leq r$. The $t$-type
of a tuple $\vec{x}$ is the isomorphism class of ${\rm cl}_{t}(\vec{x})$ with fixed $\vec{x}$.
\end{definition}

\begin{definition}
Given a sequence of nonnegative integers $0=t_{0},t_{1},\dots,t_{k-1}$,
a \emph{look-ahead $(t_{0},t_{1},\dots,t_{k-1})$-strategy} for Duplicator
in ${\rm EF}(G_{1},G_{2},k)$ is a strategy that satisfies the following
condition. For every $0\leq i\leq k-1$, when there are $i$ rounds
remaining in the game, the $t_{i}$-types of the already marked vertices
are the same in both graphs. 
\end{definition}

The following lemma (see \cite[Chapter 6]{spencer}) summarizes the deterministic part of the original
argument about the existence of a look-ahead
strategy.
\begin{lemma}[\cite{spencer}]
\label{lem:look-ahead}Let $G_{1},G_{2}$ be two graphs and let $\ell,t$
be non-negative integers. Let $x_{1},\dots,x_{\ell}\in G_{1}$ and
$y_{1},\dots,y_{\ell}\in G_{2}$ be two tuples with the same $t$-type.
 Assume
that $u\geq t$ is an integer  such that $t$-closure of any $(\ell+1)$-tuple in $G_{1}$ has at
most $u-1$ non-roots (so at most $\ell+u$ vertices overall).  Also assume that for every rooted graph $(R,H)$ with $r\leq\ell+u$
roots and $v\leq\ell+u$ non-roots, every $r$-tuple in $G_{2}$ has
a $u$-generic $(R,H)$-extension. Then, for every $x_{\ell+1}\in G_{1}$
(representing Spoiler's move) there exists $y_{\ell+1}\in G_{2}$
such that $x_{1},\dots,x_{\ell+1}\in G_{1}$ and $y_{1},\dots,y_{\ell+1}\in G_{2}$
have the same $u$-type.
\end{lemma}

We can now prove the existence of a look-ahead strategy in
${\rm EF}({\bf G}_{n}^{1},{\bf G}_{n}^{2},k)$, and thus complete the proof
of Theorem \ref{thm:almost_uniform_lower_bound}.
\begin{proof}[Proof of Theorem \ref{thm:almost_uniform_lower_bound}]
Let ${\bf G}_{n}^{1},{\bf G}_{n}^{2}\sim G(n,n^{-\alpha})$
be independent and set $k=\frac{1}{\ln d_0}\ln\ln\ln n$. Let $\mathcal{A}$ be the event
that the conclusions of 
Bounded Closure Lemma and Theorem \ref{lem:generic_extensions}
hold in both ${\bf G}_{n}^{1}$ and ${\bf G}_{n}^{2}$. Then $\mathbb{P}\left(\mathcal{A}\right)=1-o(1)$.
We show that, given the event $\mathcal{A}$, Duplicator has a winning
strategy in ${\rm EF}({\bf G}_{n}^{1},{\bf G}_{n}^{2},k)$. 

Construct a sequence $t_{0}\leq t_{1}\leq\dots\leq t_{k-1}$ inductively
as follows: $t_{0}=0$, and given $t_{i}$, we take $t_{i+1}\geq t_i$ such that the $t_{i}$-closure of any $(k-i)$-tuple in both ${\bf G}_{n}^{1}$
and ${\bf G}_{n}^{2}$ has at most $t_{i+1}-1$ non-roots. Note that these are the conditions from Lemma \ref{lem:look-ahead}
with $t=t_{i}$, $\ell=k-i-1$ and $u=t_{i+1}$. From the Bounded
Closure Lemma, a suitable $t_{i+1}$ exists, and can be taken to be
$t_{i+1}=C(k-i)t_{i}^{d}+1$ where $C=C_{\alpha}$ is a constant
(note that it also promises $t_{i+1}\geq t_{i}$ as long as we take
$C\geq1$). From Theorem \ref{lem:generic_extensions} (with $M$
replaced by $2M$), for every rooted graph $(R,H)$ with $r\leq2M$
roots and $v\leq2M$ non-roots and for every $t\leq2M$, we have that
in both ${\bf G}_{n}^{1}$ and ${\bf G}_{n}^{2}$, every $r$-tuple has a
$t$-generic $(R,H)$-extension.

Note that if $t_{i+1}\leq M$ then the assumptions of Lemma \ref{lem:look-ahead}
apply with $t,\ell,u$ as above. Indeed, we have $\ell\leq k$ and
$k=o(M)$, so $\ell+u\leq2M$ and $u\leq M$, and so for every rooted
graph $(R,H)$ with $r\leq\ell+u$ roots and $v\leq\ell+u$ non-roots,
every $r$-tuple has a $u$-generic $(R,H)$-extension. We deduce
that Duplicator can follow a look-ahead
$(t_{0},t_{1},\dots,t_{k-1})$-strategy, provided that $t_{k-1}\leq M$. To show that $t_{k-1}\leq M$ we consider the recurrence relation
defining $t_{i}$. We have $t_{0}=0$, $t_{1}=1$ and for every $1\leq i\leq k-1$,
\[
t_{i+1}=C(k-i)t_{i}^{d}+1\leq Ckt_{i}^{d}+1\leq\tilde{C}kt_{i}^{d}
\]
where $\tilde{C}=2C$. We deduce $t_{i}\leq(\tilde{C}k)^{n_{i}}$
where $n_{1}=0$ and $n_{i+1}=1+dn_{i}$, implying $n_{i}\leq d^{i}$.
Overall we have 
\begin{eqnarray*}
t_{k-1} & \leq & \left(\tilde{C}k\right)^{d^{k}}=\exp\left(d^{k}\ln\left(\tilde{C}k\right)\right)\\
 & = & \exp\left(\exp\left(\frac{\ln d}{\ln d_{0}}\ln\ln\ln n\right)\cdot(1+o(1))\ln\ln\ln\ln n\right)\\
 & = & \exp\left(o(\ln\ln n)\right)=o(M)
\end{eqnarray*}

where we used the assumption $d<d_0$. That finishes the proof.
\end{proof}

\section{Asymmetry: sparse case\label{sec:Asymmetry_sparse}}

In this section we prove the asymmetry of the random balanced graph for the sparse case $m<\frac{3}{2}n$, thus completing the proof of Theorem~\ref{thm:asymmetry}.


Fix $m=m(n)$ with $n+2\leq m<\frac{3}{2}n$ and let ${\bf H}={\bf H}(n,m)$.
Write $m=2\cdot\frac{n}{2}+r$ where $2\leq r<\frac{n}{2}$. The regular
component of ${\bf H}$ is therefore empty and ${\bf H}={\bf H}_{h}+{\bf H}_{b}$.
That is, ${\bf H}$ is a random $n$-cycle with additional
$r$ random balancing edges. Note that the lower bound $m\geq n+2$
cannot be improved: a cycle with one additional edge has
a single non-trivial isomorphism.

We first expose the edges of $\mathbf{H}_{h}$. Due to symmetry, it
is sufficient to prove that w.h.p.\ $\mathbf{H}$ is asymmetric subject
to $\mathbf{H}_{h}=H_{h}:=(1,2,\ldots,n,1)$. Let $R=R(H_{h},r)$
be the set of almost equidistributed vertices. The only random component
is therefore the endpoints $\left\{ {\bf y}[x]:x\in R\right\} $,
which determine the balancing edges $\{x,{\bf y}[x]\}$. 

For a more convenient description of the proof, let us color the Hamiltonian
edges in red and the balancing edges in blue. Since we are working
with the simple graph ${\bf H}$, every edge is assigned with exactly
one color. The red edges are deterministic while the blue edges are
random; see Figure \ref{fig:sparse_random_graph}.


\begin{figure}[h]
\begin{centering}
\includegraphics{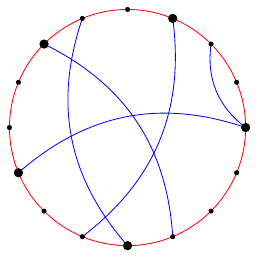}
\par\end{centering}
\caption{\label{fig:sparse_random_graph}A sample of ${\bf H}$ with $n=16$ and $r=5$. Vertices of $R$ are marked by big black circles.}
\end{figure}

The rest of the section is organized as follows.  In Section~\ref{subsec:Rare-configurations} we list several configurations which occur in ${\bf H}$ with probability $o(1)$. A key concept in the proof is that of an \emph{alternating cycle}, which is a cycle with edges of alternating colors. We show that w.h.p.\ ${\bf H}$ contains no two large alternating cycles with equal degree sequences, no two small alternating cycles which are connected by an alternating path, no small alternating cycles with additional internal edges, and does not contain some other rare configurations. In Section~\ref{subsec:Proving-asymmetry} we show that these configurations naturally arise from symmetries of ${\bf H}$, and thus prove that ${\bf H}$ is asymmetric w.h.p.

\subsection{Rare configurations}
\label{subsec:Rare-configurations}

\begin{definition}
An \emph{alternating walk} in ${\bf H}$ is a walk $v_{1},v_{2},\dots,v_{k}$
such that colors of the edges $\{v_{1},v_{2}\},\{v_{2},v_{3}\},\dots,\{v_{k-1},v_{k}\}$
alternate between red and blue. An \emph{alternating path} is an alternating
walk without repeated vertices. An \emph{alternating cycle} is an
alternating walk with $v_{1}=v_{k}$ and no other repetitions. Note that an odd alternating cycle contains two adjacent edges of
the same color.
\end{definition}

Let us call labeled cycles $C=(x_1,\ldots,x_{\ell}),\,C'=(x'_1,\ldots,x'_{\ell})$ {\it equipotent} in $\mathbf{H}$, if, for every $i$, vertices $x_i$ and $x'_i$ have equal degrees.

\begin{prop}
W.h.p. $\mathbf{H}$ does not contain two different (though their sets of vertices may coincide) equipotent alternating cycles of length at least $\ln^2n$. In particular, if $r=o(n)$, then w.h.p. there are no alternating cycles.
\label{prop:alt-cycles-are-small}
\end{prop}

\begin{remark}
\label{rem:R-vertices-and-edges}We make two simple observations:
\begin{itemize}
\item Since $r<\frac{n}{2}$, every red edge has at most one endpoint from
$R$.
\item By definition, every blue edge has at least one endpoint from $R$.
\end{itemize}
\end{remark}

\begin{proof}[Proof of Proposition~\ref{prop:alt-cycles-are-small}]

Let us first assume that $0\leq r\leq\xi n$ for a positive constant $\xi<\frac{1}{2}$ and prove an even stronger statement: either $r=o(n)$, and then w.h.p.\ ${\bf H}$
contains no alternating cycles, or $r=\Theta(n)$, and then w.h.p.\ ${\bf H}$
contains no alternating cycles of length greater than $\ln n$.

We apply Markov's inequality. Let ${\bf X}_{\ell}$ be the number
of alternating cycles of length $\ell$ in ${\bf H}$. Write it as
the sum of indicators: ${\bf X}_{\ell}=\sum_{C}\mathds{1}_{\left\{ C\subseteq{\bf H}\right\} }$
where the sum is over all possible alternating cycles of length $\ell$. 

First, assume $\ell\geq4$ is even. An alternating cycle $C$ of length
$\ell$ must consist of $\frac{\ell}{2}$ red edges and $\frac{\ell}{2}$
blue edges. From Remark \ref{rem:R-vertices-and-edges} it follows
that $C$ contains exactly $\frac{\ell}{2}$ vertices from $R$. Every
edge in $C$ has exactly one endpoint from $R$. It can be therefore
written as 
\begin{equation}
u_{1}\textcolor{red}{\sim}w_{1}\textcolor{blue}{\sim}u_{2}\textcolor{red}{\sim}w_{2}\textcolor{blue}{\sim}\dots\textcolor{blue}{\sim}u_{\frac{\ell}{2}}\textcolor{red}{\sim}w_{\frac{\ell}{2}}\textcolor{blue}{\sim}u_{1}\label{eq:even-alt-cycle}
\end{equation}
where $u_{i}\in R$ and $w_{i}\in R^{c}$.
\begin{claim}
Given $R$, the number of possible alternating cycles of length $\ell$
is at most $2^{\frac{\ell}{2}}(r)_{\frac{\ell}{2}}$.
\end{claim}
\begin{proof}
Consider the following procedure:
\begin{enumerate}
\item Choose a sequence $u_{1},u_{2},\dots,u_{\frac{\ell}{2}}$ of distinct
vertices from $R$.
\item For every $1\leq i\leq\frac{\ell}{2}$, choose one of the two red
edges incident to $u_{i}$, denote it $\{u_{i},w_{i}\}$. 
\end{enumerate}
This procedure defines a single possible alternating cycle (\ref{eq:even-alt-cycle})
of length $\ell$. Every possible alternating cycle can be obtained
from the procedure (perhaps in more than one way). The number of possible
choices is $2^{\frac{\ell}{2}}(r)_{\frac{\ell}{2}}$. This proves
the claim. 
\end{proof}
Given a possible alternating cycle $C$, the probability of $C\subseteq{\bf H}$
is the probability that given $\frac{\ell}{2}$ edges, each
with exactly one endpoint from $R$, appear as blue edges in ${\bf H}$.
This event can be interpreted as determining the values of $\frac{\ell}{2}$
out of the $r$ random vertices $\left\{ {\bf y}[x]:x\in R\right\} $.
In ${\bf H}^{*}$, the probability of this event is $\frac{1}{(n)_{\frac{\ell}{2}}}$.
Therefore, by Proposition $\ref{prop:simplicity_probability}$, in
${\bf H}$ the probability of this event is at most $\frac{p_{0}}{(n)_{\frac{\ell}{2}}}$
where $p_{0}$ is a constant. Overall 
\begin{equation}
\mathbb{E}\left({\bf X}_{\ell}\right)\leq p_{0}2^{\frac{\ell}{2}}\frac{(r)_{\frac{\ell}{2}}}{(n)_{\frac{\ell}{2}}}\leq p_{0}\left(\frac{2r}{n}\right)^{\frac{\ell}{2}}.\label{eq:E(X_l) even bound}
\end{equation}

Now assume $\ell\geq3$ is odd. Write ${\bf X}_{\ell}={\bf X}_{\ell}^{[b]}+{\bf X}_{\ell}^{[r]}$
where ${\bf X}_{\ell}^{[b]}$ is the number of alternating cycles
with $\frac{\ell+1}{2}$ blues edges (\emph{bluish cycles}) and ${\bf X}_{\ell}^{[r]}$
is the number of alternating cycles with $\frac{\ell+1}{2}$ red edges
(\emph{reddish cycles}). 

We start with ${\bf X}_{\ell}^{[b]}$. Let $C$ be a possible bluish
cycle of length $\ell$. It has a single vertex $u_{0}$ incident
to two (potentially) blue edges. From Remark \ref{rem:R-vertices-and-edges}
it follows that $C$ contains exactly $\frac{\ell+1}{2}$ vertices
from $R$. Without loss of generality, we can express any bluish $\ell$-cycle
as
\[
u_{0}\textcolor{blue}{\sim}u_{1}\textcolor{red}{\sim}w_{1}\textcolor{blue}{\sim}u_{2}\textcolor{red}{\sim}w_{2}\textcolor{blue}{\sim}\dots\textcolor{blue}{\sim}u_{\frac{\ell-1}{2}}\textcolor{red}{\sim}w_{\frac{\ell-1}{2}}\textcolor{blue}{\sim}u_{0}
\]
where $u_{i}\in R$, $w_{i}\in R^{c}$. Indeed, since all $y[x]$
have to be distinct, the only common vertex $u_{0}$ of two blue edges
have to be in $R$. Then, exactly one of the two blue edges adjacent
to $u_{0}$ must have both endpoints from $R$. Now we can similarly
claim that the number of possible bluish cycles of length $\ell$
is at most $2^{\frac{\ell-1}{2}}(r)_{\frac{\ell+1}{2}}$. Moreover,
the event $C\subseteq{\bf H}$ now determines the values of $\frac{\ell+1}{2}$
random vertices ${\bf y}[x]$, therefore its probability in ${\bf H}$
is $\leq\frac{p_{0}}{(n)_{\frac{\ell+1}{2}}}$. Overall 
\[
\mathbb{E}\left({\bf X}_{\ell}^{[b]}\right)\leq p_{0}2^{\frac{\ell-1}{2}}\frac{(r)_{\frac{\ell+1}{2}}}{(n)_{\frac{\ell+1}{2}}}\leq\frac{p_{0}}{2}\left(\frac{2r}{n}\right)^{\frac{\ell+1}{2}}.
\]
It remains to estimate the expectation of ${\bf X}_{\ell}^{[r]}$.
Let $C$ be a possible reddish cycle of length $\ell$. From Remark
\ref{rem:R-vertices-and-edges} it follows that the number of $R$-vertices
in $C$ is either $\frac{\ell-1}{2}$ or $\frac{\ell+1}{2}$. Moreover,
$C$ contains a red $2$-path, and at least one of its endpoints must
be from $R$ (otherwise there cannot be at least $\frac{\ell-1}{2}$
$R$-vertices in $C$). Let $xyz$ be the red $2$-path, where $x\in R$.
It is possible that $z\in R$ as well. If $z\not\in R$ then necessarily
$C$ contains exactly $\frac{\ell-1}{2}$ $R$-vertices. Overall we
have three different types of possible reddish cycles; we bound the
expected number of cycles of each type separately. The three types
are demonstrated in Figure \ref{fig:reddish-cycles-types}.

\begin{figure}[h]
\begin{center}

\textcolor{red}{}%
\begin{minipage}[t]{0.3\columnwidth}%
\textcolor{red}{}\subfloat[Type 1.]{\includegraphics[scale=0.8]{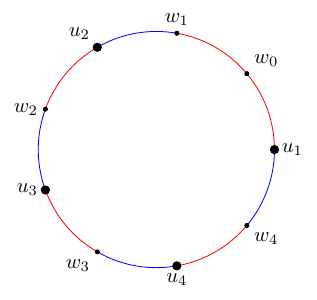}
\textcolor{red}{}}%
\end{minipage}\textcolor{red}{}%
\begin{minipage}[t]{0.3\columnwidth}%
\textcolor{red}{}\subfloat[Type 2.]{\includegraphics[scale=0.8]{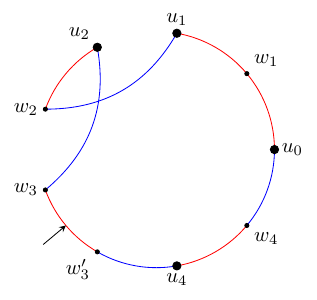}
\textcolor{red}{}}%
\end{minipage}\textcolor{red}{}%
\begin{minipage}[t]{0.3\columnwidth}%
\textcolor{red}{}\subfloat[Type 3.]{\includegraphics[scale=0.8]{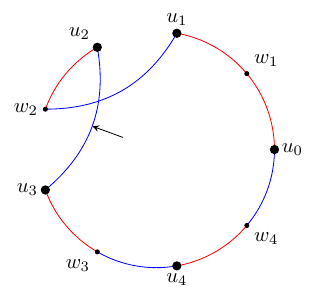}
\textcolor{red}{}}%
\end{minipage}



\end{center}

\caption{Three types of reddish cycles, $\ell=9$. In types 2 and 3, an arrow points at the \textquotedblleft special edge\textquotedblright{} whose choice adds a factor of $\Theta(\ell)$ to the bound. The cycles are drawn such the red edges are fixed, thus emphasizing the process of choosing the blue edges.\label{fig:reddish-cycles-types}}
\end{figure}

\textbf{Type 1.} There are $\frac{\ell-1}{2}$ $R$-vertices and $z\not\in R$.
Cycles of these type behave exactly like even alternating
cycles, with one red edge is replaced by a red $2$-path. In this
case we relabel $x=u_{1},y=w_{0},z=w_{1}$ and then $C$ can be written
as
\[
u_{1}\textcolor{red}{\sim}w_{0}\textcolor{red}{\sim}w_{1}\textcolor{blue}{\sim}u_{2}\textcolor{red}{\sim}w_{2}\textcolor{blue}{\sim}\dots\textcolor{blue}{\sim}u_{\frac{\ell-1}{2}}\textcolor{red}{\sim}w_{\frac{\ell-1}{2}}\textcolor{blue}{\sim}u_{1}
\]
where $u_{i}\in R$ and $w_{i}\in R^{c}$. The number of possible
cycles of this type is bounded by $2^{\frac{\ell-1}{2}}(r)_{\frac{\ell-1}{2}}$. Since
$\mathbb{P}\left(C\subseteq{\bf H}\right)\leq\frac{p_{0}}{(n)_{\frac{\ell-1}{2}}}$, the expected number of reddish cycles of type 1 is at
most $p_{0}\left(\frac{2r}{n}\right)^{\frac{\ell-1}{2}}$.

\textbf{Type 2.} There are $\frac{\ell-1}{2}$ $R$-vertices and $z\in R$.
Then there exists a single red edge which has \emph{no} endpoints
in $R$. Let $e_{0},e_{1},e_{2},\dots,e_{\frac{\ell-1}{2}}$ be
the sequence of red edges in $C$, starting from $e_{0}=\{x,y\},e_{1}=\{y,z\}$
and going in the cyclic order. Let $2\leq j\leq\frac{\ell-1}{2}$ be
such that $e_{j}$ is the red edge with no endpoint from $R$. Write
$e_{j}=\{w_{j},w_{j}'\}$. For $1\leq i\leq\frac{\ell-1}{2}$, $i\neq j$
write $e_{i}=\{u_{i},w_{i}\}$ where $u_{i}\in R,w_{i}\in R^{c}$
(so in particular $y=w_{1}$ and $z=u_{1}$) and also let $x=u_{0}$.
The blue neighbors of $w_{j},w_{j}'$ in $C$ must be $u_{j-1},u_{j+1}$;
without loss of generality, assume the blue edges are $\{u_{j-1},w_{j}\}$
and $\{w_{j}',u_{j+1}\}$. We get the following cycle:
\begin{equation}
\begin{array}{c}
u_{0}\textcolor{red}{\sim}w_{1}\textcolor{red}{\sim}u_{1}\textcolor{blue}{\sim}w_{2}\textcolor{red}{\sim}u_{2}\textcolor{blue}{\sim}\dots\textcolor{blue}{\sim}w_{j-1}\textcolor{red}{\sim}u_{j-1}\\
\textcolor{blue}{\sim}w_{j}\textcolor{red}{\sim}w_{j}'\textcolor{blue}{\sim}u_{j+1}\textcolor{red}{\sim}w_{j+1}\textcolor{blue}{\sim}\dots\textcolor{blue}{\sim}u_{\frac{\ell-1}{2}}\textcolor{red}{\sim}w_{\frac{\ell-1}{2}}\textcolor{blue}{\sim}u_{0}.
\end{array}\label{eq:type-2-cycle}
\end{equation}

The conclusion of this analysis is that the number of possible cycles
of type 2 is at most the number of choices in the following procedure:
\begin{enumerate}
\item Choose an index $2\leq j\leq\frac{\ell-1}{2}$. 
\item Choose a sequence of $\frac{\ell-3}{2}$ distinct vertices from $R$
and denote them $u_{i}$ for $i\in\{0\}\cup\left\{ 2,3,\dots\frac{\ell-1}{2}\right\} \setminus\{j\}.$
\item For $u_{0}$ choose one of the two red 2-paths starting from it, denote
it $\{u_{0},w_{1},u_{1}\}$. If $u_{1}\not\in R$, halt.
\item For every other $u_{i}$ choose one of the two red edges incident
to it, denote it $\{u_{i},w_{i}\}$. 
\item Choose an additional red edge with no endpoints in $R$. Also choose
one of its vertices to be denoted $w_{j}$, and denote the second
one $w_{j}'$. 
\end{enumerate}
The procedure indeed uniquely defines a possible cycle of type 2:
the one described in (\ref{eq:type-2-cycle}). As explained, every
possible cycle of type 2 can be obtained from it. The number of choices
in the procedure, which bounds the number of possible cycles of type
2, is at most $\ell\cdot(r)_{\frac{\ell-3}{2}}\cdot2^{\frac{\ell-3}{2}}\cdot2n$.
Like before, we still have $\mathbb{P}\left(C\subseteq{\bf H}\right)\leq\frac{p_{0}}{(n)_{\frac{\ell-1}{2}}}$.
In conclusion, the expected number of reddish cycles of type 2 is
at most
\[
2p_{0}\ell\cdot2^{\frac{\ell-3}{2}}\frac{n(r)_{\frac{\ell-3}{2}}}{(n)_{\frac{\ell-1}{2}}}=2p_{0}\ell\cdot2^{\frac{\ell-3}{2}}\frac{(r)_{\frac{\ell-3}{2}}}{(n-1)_{\frac{\ell-3}{2}}}\leq2p_{0}\ell\cdot\left(\frac{2r}{n-1}\right)^{\frac{\ell-3}{2}}.
\]
\textbf{Type 3.} There are $\frac{\ell+1}{2}$ $R$-vertices and $z\in R$.
In this case all red edges have exactly one endpoint from $R$. Moreover,
there is a single blue edge which has both endpoints from $R$. The
analysis is now similar to the previous type. Again, let $e_{0},e_{1},e_{2},\dots,e_{\frac{\ell-1}{2}}$
be the sequence of red edges in $C$, starting from $e_{0}=\{x,y\},e_{1}=\{y,z\}$
and going in cyclic order. For $1\leq i\leq\frac{\ell-1}{2}$ write
$e_{i}=\{u_{i},w_{i}\}$ where $u_{i}\in R,w_{i}\in R^{c}$ and also
let $x=u_{0}$. There exists $1\leq j\leq\frac{\ell-1}{2}$ such that
the blue edge connecting $e_{j}$ with $e_{j+1}$ has both endpoints
from $R$ (here we let $e_{\frac{\ell+1}{2}}=e_{0}$). We get the
following cycle:
\begin{equation}
\begin{array}{c}
u_{0}\textcolor{red}{\sim}w_{1}\textcolor{red}{\sim}u_{1}\textcolor{blue}{\sim}w_{2}\textcolor{red}{\sim}u_{2}\textcolor{blue}{\sim}\dots\textcolor{blue}{\sim}w_{j}\textcolor{red}{\sim}u_{j}\\
\textcolor{blue}{\sim}u_{j+1}\textcolor{red}{\sim}w_{j+1}\textcolor{blue}{\sim}\dots\textcolor{blue}{\sim}u_{\frac{\ell-1}{2}}\textcolor{red}{\sim}w_{\frac{\ell-1}{2}}\textcolor{blue}{\sim}u_{0}.
\end{array}\label{eq:type-3-cycle}
\end{equation}
So the number of possible cycles of type 3 is at most $\ell\cdot(r)_{\frac{\ell-1}{2}}\cdot2^{\frac{\ell-1}{2}}$.
As for $\mathbb{P}\left(C\subseteq{\bf H}\right)$, it is now bounded
by $\frac{2p_{0}}{(n)_{\frac{\ell-1}{2}}}$; the extra factor of $2$
comes from the blue edge $\{u_{j},u_{j+1}\}$, which could come from
$u_{j+1}=y[u_{j}]$ or from $u_{j}=y[u_{j+1}]$. In conclusion, the
expected number of reddish cycles of type 3 is at most
\begin{eqnarray*}
2p_{0}\ell\cdot2^{\frac{\ell-1}{2}}\frac{(r)_{\frac{\ell-1}{2}}}{(n)_{\frac{\ell-1}{2}}} & \leq & 2p_{0}\ell\cdot\left(\frac{2r}{n}\right)^{\frac{\ell-1}{2}}.
\end{eqnarray*}

In summary, we have proved the following bound for every $\ell\geq3$:
\begin{equation}
\mathbb{E}\left({\bf X}_{\ell}\right)=O\left(\ell\left(\frac{2r}{n}\right)^{\frac{\ell-3}{2}}\right).\label{eq:E(X_l) bound}
\end{equation}
Assume $r=\Theta(n)$. Then the expected number of alternating
cycles of length $\geq\ln n$ is
\[
\sum_{\ell\geq\ln n}\mathbb{E}\left({\bf X}_{\ell}\right)=O\left(\sum_{\ell\geq\ln n}\ell\left(\frac{2r}{n}\right)^{\frac{\ell-3}{2}}\right)=O\left(\sum_{\ell\geq\ln n}\ell\left(2\xi\right)^{\frac{\ell-3}{2}}\right)=o(1)
\]
since $2\xi<1$. Then we are done by Markov's inequality. 

Now assume $r=o(n)$. Then reddish cycles of type 2 or 3 are
impossible, because there is no red path of length $2$ with both
endpoints in $R$. In that case we can improve (\ref{eq:E(X_l) bound})
and write $\mathbb{E}\left({\bf X}_{\ell}\right)=O\left(\ell\left(\frac{2r}{n}\right)^{\frac{\ell-1}{2}}\right)$.
Then the expected number of alternating cycles is
\[
O\left(\sum_{\ell\geq3}\ell\left(\frac{2r}{n}\right)^{\frac{\ell-1}{2}}\right)=o(1).
\]
Again, we apply Markov's inequality and complete the proof in the case when $r$ is bounded away from $n/2$.

Finally, assume that $r<n/2$ and $r=(\frac{1}{2}-o(1))n$. We apply the bound (\ref{eq:E(X_l) bound}), noting that its proof does not use the fact that $r$ is bounded away from $\frac{n}{2}$. The bound can now be written more neatly as
\begin{equation}
\mathbb{E}\left({\bf X}_{\ell}\right)=O\left(\ell\left(\frac{2r}{n}\right)^{\frac{\ell}{2}}\right).\label{eq:E(X_l) neat}
\end{equation}
Let us first prove that w.h.p.\ there are no alternating cycles of size at least $2n/7$. It immediately follows from the union bound and ~\eqref{eq:E(X_l) neat}. Note that the power of $(2r/n)$ in this bound is, in particular, due to the inequality $(r)_{\lfloor\ell/2\rfloor}/(n)_{\lfloor\ell/2\rfloor}\leq (r/n)^{\lfloor\ell/2\rfloor}$. When $\ell$ is linear $n$ a much stronger inequality holds true:
\begin{align*}
\frac{(r)_{\lfloor\ell/2\rfloor}}{(n)_{\lfloor\ell/2\rfloor}} & =
\frac{r!(n-\lfloor\ell/2\rfloor)!}{n!(r-\lfloor\ell/2\rfloor)!}\lesssim\frac{r^r(n-\lfloor\ell/2\rfloor)^{n-\lfloor\ell/2\rfloor}}{n^n(r-\lfloor\ell/2\rfloor)^{r-\lfloor\ell/2\rfloor}}\\
&=\left(\frac{r}{n}\right)^{\lfloor\ell/2\rfloor}\left(1+\frac{\lfloor\ell/2\rfloor}{r-\lfloor\ell/2\rfloor}\right)^{r-\lfloor\ell/2\rfloor}\left(1-\frac{s}{n-s}\right)^{n-\lfloor\ell/2\rfloor}\\
&=\left(\frac{r}{n}
\left[\left(\frac{1}{1-2x}\right)^{(1-2x)/(2x)}(1-x)^{(1-x)/x}\right]
\right)^{\lfloor\ell/2\rfloor}
,
\end{align*}
where $x=\lfloor \ell/2\rfloor/n\in[1/7,1/2)$. Note that the function of $x$ in the above bound decreases with $x$, and so it is at most $(7/5)^{5/2}(6/7)^6:=\alpha<1$. Eventually, we get the following refined bound:
$$
\mathbb{E}\left({\bf X}_{\ell}\right)=O\left(\ell\left(\alpha\frac{2r}{n}\right)^{\frac{\ell}{2}}\right),
$$
and then Markov's inequality and the union bound over $\ell\geq\frac{2}{7}n$ implies that indeed w.h.p. there are no alternating cycles of length at least $2n/7$. This implies that w.h.p. any union of two alternating cycles has at most $\frac{3}{7}n+O(1)$ vertices from $R$.

Assume that different $C=(x_1,\ldots,x_{\ell})$ and $C'=(x'_1,\ldots,x'_{\ell})$ are equipotent in $\mathbf{H}$. Let as assume that, for a certain $i\in[\ell]$, $x_i=x'_{i}\in R$ and, at the same time, $x_{i+2}=x'_{i+2}\in R$ as well. Due to the description of types of alternating cycles, it may only happen when $x_i$ and $x_{i+2}$ are joined in $C$ by the path $x_ix_{i+1}x_{i+2}$  with edges of different color, and the same is true in $C'$ (an entirely blue $x_ix_{i+1}x_{i+2}$ might have two consecutive vertices from $R$, and the third vertex should not belong to $R$). Moreover, both $x_i$ and $x_{i+2}$ should be adjacent to at least one blue vertex both in $C$ and $C'$. Since colours of $\{x_i,x_{i+1}\}$ and $\{x'_i,x'_{i+1}\}$ are different, $x_{i+1}\neq x'_{i+1}$. This may only happen when either $x_i$ or $x_{i+2}$ belongs to two blue edges in $C\cup C'$. But then one of these two blue edges joins two vertices from $R$. There can be only constantly many such $i$ due to the description of types of alternating cycles.
 
Then there exists a set $\mathcal{J}\subset[\ell]$ of size at least $\frac{1-o(1)}{4}\ell$ such that, for every $i\in\mathcal{J}$, $x_i\in R$, and $x_i\neq x'_i$. From this it immediately follows that there exists $\mathcal{\tilde J}\subset\mathcal{J}$ of size at least $\frac{1-o(1)}{8}\ell$ such that sets $\mathcal{R}:=\{x_i,i\in\mathcal{\tilde J}\}$ and $\mathcal{R}':=\{x'_i,i\in\mathcal{\tilde J}\}$ are disjoint and $\mathcal{R}\subset R$. 



Let us now finish the proof. Consider two (not necessarily disjoint) sets $\{x_1,\ldots,x_{\ell}\}$ and $\{x'_1,\ldots,x'_{\ell}\}$. Assume that the event saying that $C=(x_1,\ldots,x_{\ell})$ and $C'=(x'_1,\ldots,x'_{\ell})$ are alternating cycles in $\mathbf{H}$ holds. We may assume that at least $(1-o(1))\frac{n}{14}$ vertices of $R$ do not belong to $C\cup C'$. Denote the set of these vertices of $R$ that do not belong to any of the cycles by $R'$. Almost all (but constantly many) vertices of $R$ that belong to the union of the cycles still can play the role of $y[x]$ for $x\in R'$ (unless they do not belong to blue edges with both endpoints from $R$). We let $\tilde R$ be the set of such vertices from $R\cap (V(C)\cup V(C'))$. Moreover, let $\mathcal{\tilde R}=\tilde R\cap\mathcal{R}$, $\mathcal{\tilde R}'=\tilde R\cap\mathcal{R}'$. Recall that $|\mathcal{\tilde R}|=\frac{1-o(1)}{8}\ell$.

Note that the set $\mathcal{Y}$ of $y[x]$, $x\in R'$, is a uniformly random subset of $[n]\setminus(V(C)\cup V(C'))\cup\tilde R$ of size $|R'|$.
Let $X$ be the number of vertices in $\mathcal{\tilde R}\cap\mathcal{Y}$. Due to the Hoeffding tail bounds for the hypergeometric distribution (see, e.g.~\cite[Theorem 2.10]{rucinski_luczak_janson}), $X$ is at least $\frac{1-o(1)}{112}\ell$ with probability $1-\exp(-\Omega(\ell))$. Let $\varphi$ be the bijection from $\mathcal{R}$ to $\mathcal{R}'$ that sends $x_i$ to $x'_i$. If, for a certain $x\in\mathcal{Y}\cap\mathcal{\tilde R}$, $\varphi(x)\notin \mathcal{\tilde R}'$, then cycles $C$ and $C'$ cannot be equipotent. Thus, we get that the event that $C$ and $C'$ are equipotent implies that $\varphi|_{\mathcal{Y}\cap\mathcal{\tilde R}}$ is a bijection between  $\mathcal{Y}\cap\mathcal{\tilde R}$ and $\mathcal{Y}\cap\mathcal{\tilde R}'$. Since the probability that $\mathcal{Y}\cap\mathcal{\tilde R}'$ coincides with a fixed subset of $\mathcal{\tilde R}'$ is $\exp(-\Omega(\ell))$, we get the statement of Proposition~\ref{prop:alt-cycles-are-small} due to the bound~\eqref{eq:E(X_l) neat} and the union bound over the choices of $\ell$, $x_1$, $x_1'$, and the directions in both cycles $C,C'$.

\end{proof}

In the case $r=o(n)$, Proposition \ref{prop:alt-cycles-are-small}
rules out the existence of alternating cycles, which is already sufficient for a proof that ${\bf H}$ is asymmetric w.h.p. (as would be clear from the rest of the proof). In the case $r=\Theta(n)$,
however, we must consider
more specific configurations. Thus, for the rest of this subsection,
let us assume that $r\geq cn$ where $c$ is a positive constant.

\begin{definition}
For two non-empty sets $U_{1},U_{2}\subseteq[n]$, the \emph{red
distance} between them is the minimal length of a red path connecting a vertex
from $U_{1}$ to a vertex from $U_{2}$. 
\end{definition}

\begin{definition}
For a set $U\subseteq[n]$, $\left|U\right|\geq2$, a \emph{red segment}
of $U$ of length $t$ is a red path with endpoints $x_{0},x_{t}\in U$
and internal vertices $x_{1},\dots,x_{t-1}\not\in U$.
\end{definition}

\begin{prop}
\label{prop:alt-cycles_all-in-1}
W.h.p.\ in ${\bf H}$, there
are no alternating cycles of length at most $\ln^2n$ 
\begin{enumerate}
\item with an additional alternating path of length between $2$ and $\ln^2 n$ connecting two (not necessarily distinct) vertices of the cycle;
\item with additional internal edges;
\item with another alternating cycle of length at most $\ln^2n$ at red distance less than $\sqrt{n}$;
\item with a red segment of length between $2$ and $\sqrt{n}$.
\end{enumerate}
\end{prop}

We will need the following claim.

\begin{claim}
\label{claim:alternating_paths}Fix two distinct vertices $x,y$ and
let $t\geq2$. Let ${\bf X}_{t}[x,y]$ be the number of alternating
paths of length $t$ which have $x,y$ as the two endpoints. Then
\[
\mathbb{E}({\bf X}_{t}[x,y])=O\left(\frac{1}{n}\right)\cdot t\left(\frac{2r}{n}\right)^{\frac{t}{2}}.
\]
\end{claim}

\begin{proof}
The proof involves case analysis very similar to that of the previous
proof. We avoid a repetition of the same details, and instead only briefly
describe the different possible types of alternating paths (see also
Figure \ref{fig:paths-types}).

Again, we distinguish between three cases: even paths, bluish odd
paths, and reddish odd paths. Each case can be subdivided into types
analogous to the types of reddish cycles from the previous proof.
\begin{description}
\item [{Even~path}] $P$ consists of $\frac{t}{2}$ red edges
and $\frac{t}{2}$ blue edges. Without loss of generality $x$ is
incident to a blue edge. 
\begin{description}
\item [{Type~E1.}] $x\not\in R$ and $\left|V(P)\cap R\right|=\frac{t}{2}$.
\item [{Type~E2.}] $x\in R$ and $\left|V(P)\cap R\right|=\frac{t}{2}$.
\item [{Type~E3.}] $x\in R$ and $\left|V(P)\cap R\right|=\frac{t}{2}+1$.
\end{description}
\item [{Bluish~path}] $P$ consists of $\frac{t-1}{2}$ red
edges and $\frac{t+1}{2}$ blue edges. At least one endpoint of $P$
must be in $R$, due to Remark \ref{rem:R-vertices-and-edges}. Without
loss of generality assume $x\in R$.
\begin{description}
\item [{Type~B1.}] $y\not\in R$ and $\left|V(P)\cap R\right|=\frac{t+1}{2}$.
\item [{Type~B2.}] $y\in R$ and $\left|V(P)\cap R\right|=\frac{t+1}{2}$.
\item [{Type~B3.}] $y\in R$ and $\left|V(P)\cap R\right|=\frac{t+3}{2}$.
\end{description}
\item [{Reddish~path}] $P$ consists of $\frac{t+1}{2}$ red
edges and $\frac{t-1}{2}$ blue edges. In this case we exclude the
endpoints of $P$ from the count of $R$-vertices, since they are
not incident to any blue edge in the path. 
\begin{description}
\item [{Type~R2.}] $\left|\left(V(P)\cap R\right)\setminus\left\{ x,y\right\} \right|=\frac{t-1}{2}$.
\item [{Type~R3.}] $\left|\left(V(P)\cap R\right)\setminus\left\{ x,y\right\} \right|=\frac{t+1}{2}$.
\end{description}
\end{description}
For all types, the expected number of paths is determined (up to a
$\Theta(1)$-factor) by combining a factor of $2r$ for every \emph{internal}
red edge, a factor of $\frac{1}{n}$ for every blue edge, and an additional
factor of $t$ in types numbered $2$ and $3$ (coming from the choice
of one ``special edge'', indicated by an arrow in Figure \ref{fig:paths-types}). 
\end{proof}

\begin{figure}[!h]
\begin{center}

\textcolor{red}{}\subfloat[Type E1.]{\includegraphics{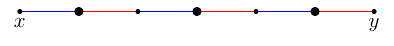}
\textcolor{red}{}}

\textcolor{red}{}\subfloat[Type E2.]{\includegraphics{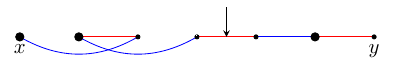}
\textcolor{red}{}}

\textcolor{red}{}\subfloat[Type E3.]{\includegraphics{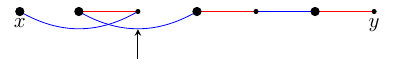}
\textcolor{red}{}}

\textcolor{red}{}\subfloat[Type B1.]{\includegraphics{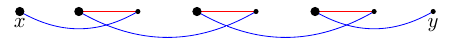}
\textcolor{red}{}}

\textcolor{red}{}\subfloat[Type B2.]{\includegraphics{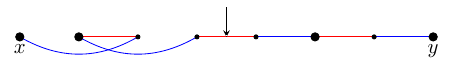}
\textcolor{red}{}}

\textcolor{red}{}\subfloat[Type B3.]{\includegraphics{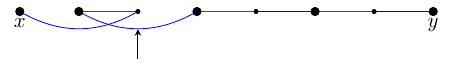}
\textcolor{red}{}}

\textcolor{red}{}\subfloat[Type R2.]{\includegraphics{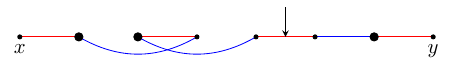}
\textcolor{red}{}}

\textcolor{red}{}\subfloat[Type R3.]{\includegraphics{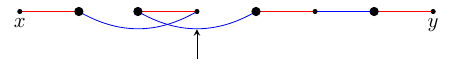}
\textcolor{red}{}}

\end{center}

\caption{\label{fig:paths-types}All possible types of alternating paths connecting a given pair of vertices $x,y$. Vertices of $R$ are marked by big black circles (excluding the endpoints in the red cases, as described in the proof). An arrow marks the \textquotedblleft special edge\textquotedblright , which is again either a red edge with no $R$-endpoints or a blue edge with two $R$-endpoints.}
\end{figure}

\begin{proof}[Proof of Proposition~\ref{prop:alt-cycles_all-in-1}.1.]

Let ${\bf X}_{\ell,t}$ be the number of pairs $(C,P)$,
where $C$ is an alternating cycle of length $\ell\leq\ln^2 n$ and $P$ is an alternating path of length $t\leq\ln^2 n$ which shares its two endpoints with $C$ (internally
disjoint from $C$). A pair $(C,P)$ can be chosen by first choosing an alternating cycle
$C$, then two vertices $x,y$ from $C$, and then an alternating
path connecting $x,y$. From (\ref{eq:E(X_l) neat}) and Claim \ref{claim:alternating_paths},
we get
$$
\mathbb{E}\left({\bf X}_{\ell,t}\right)  =  O\left(\ell\left(\frac{2r}{n}\right)^{\frac{\ell}{2}}\ell^{2}\frac{1}{n}t\left(\frac{2r}{n}\right)^{\frac{t}{2}}\right)
 =  \frac{1}{n}O\left(\ell^{3}\left(\frac{2r}{n}\right)^{\frac{\ell}{2}} t\left(\frac{2r}{n}\right)^{\frac{t}{2}}\right).
$$
Summing over $3\leq\ell\leq\ln^2 n$ and $t\leq\ln^2 n$, and using the inequality
$\frac{2r}{n}<1$, we deduce $\sum_{\ell,t}\mathbb{E}\left({\bf X}_{\ell,t}\right)=o(1)$. 
We then apply Markov's inequality and finish the proof. 
\end{proof}

\begin{proof}[Proof of Proposition~\ref{prop:alt-cycles_all-in-1}.2.]
For every $4\leq\ell\leq\ln^2 n$, let ${\bf X}_{\ell}^{[r]}$ be the
number of pairs $(C,e)$ where $C$ is an alternating cycle of length
$\ell$ and $e$ is an additional internal red edge. Similarly let
${\bf X}_{\ell}^{[b]}$ be the number of $(C,e)$ where now $e$ is
an additional internal blue edge. The idea is that an additional red
edge reduces the number of possible choices by a factor of $\ell/r$, and
an additional blue edge reduces the probability of the existence in
${\bf H}$ by a factor of $\ell/n$.

We begin with ${\bf X}_{\ell}^{[b]}$. First notice that an additional
blue edge in $C$ is possible only when $\ell$ is odd and $C$ is
a reddish cycle of type 3 (as defined in the proof of Proposition
\ref{prop:alt-cycles-are-small}). Indeed, in all other cases, the
number of blue edges in $C$ equals the number of $R$-vertices in
$C$, and therefore no additional blue edges may exist. The general
form of a reddish cycle of type 3 is described in (\ref{eq:type-3-cycle}).
Using that notation, the additional blue edge $e$ must have either
$u_{j}$ or $u_{j+1}$ as one of its endpoints. The number of possible
choices for $(C,e)$ is therefore at most $2\ell\cdot\ell\cdot(r)_{\frac{\ell-1}{2}}\cdot2^{\frac{\ell-1}{2}}$
(an extra factor of $2\ell$ is added due to $e$). However, the probability
that $(C,e)$ exists in ${\bf H}$ is now $\frac{1}{(n)_{\frac{\ell+1}{2}}}$.
Overall
\[
\mathbb{E}\left({\bf X}_{\ell}^{[b]}\right)=O\left(\ell^{2}\cdot2^{\frac{\ell-1}{2}}\frac{(r)_{\frac{\ell-1}{2}}}{(n)_{\frac{\ell+1}{2}}}\right)=O\left(\frac{\ell^{2}}{n}\cdot\left(\frac{2r}{n}\right)^{\frac{\ell}{2}}\right).
\]
Now consider ${\bf X}_{\ell}^{[r]}$. We focus on the case where $\ell$
is even; the odd case is handled similarly and gives an additional
factor of $\ell$.

In any pair $(C,e)$ the cycle can be labelled as in $C$
as in (\ref{eq:even-alt-cycle}) in a way such that $e$ has $u_{1}$
as an endpoint. There are at most $\ell$ ways to choose the other
endpoint of $e$. Moreover, $e$ determines a red path of length 3
which contains two red edges of $C$, and reduces the
number of choices by a factor of $r$. Therefore, a bound on $\mathbb{E}\left({\bf X}_{\ell}^{[r]}\right)$
can be obtained by multiplying (\ref{eq:E(X_l) even bound}) by $\frac{\ell}{r}$,
which yields
\[
\mathbb{E}\left({\bf X}_{\ell}^{[r]}\right)=O\left(\frac{\ell}{r}\cdot\left(\frac{2r}{n}\right)^{\frac{\ell}{2}}\right)=O\left(\frac{\ell}{n}\cdot\left(\frac{2r}{n}\right)^{\frac{\ell}{2}}\right).
\]
Recalling that $\frac{2r}{n}<1$ and combining the above,
\[
\sum_{4\leq\ell\leq\ln^2 n}\mathbb{E}\left({\bf X}_{\ell}^{[b]}+{\bf X}_{\ell}^{[r]}\right)=O\left(\frac{1}{n}\right)\sum\ell^{2}\left(\frac{2r}{n}\right)^{\frac{\ell}{2}}=O\left(\frac{\ln^5 n}{n}\right).
\]
We apply Markov's inequality and finish the proof. 
\end{proof}

\begin{proof}[Proof of Proposition~\ref{prop:alt-cycles_all-in-1}.3.]
We have already proved that there are no two overlapping
alternating cycles w.h.p. Fix $3\leq\ell,\ell'\leq\ln^2 n$ and $t\leq\sqrt{n}$.
Let ${\bf X}_{\ell,\ell',t}$ be the number of triplets $(C,C',P)$
where $C,C'$ are disjoint alternating cycles of lengths $\ell,\ell'$
respectively, and $P$ is a red path connecting them. The idea is
that the red path $P$ eliminates one degree of freedom in the choice
of cycles, therefore reduces the number of choices by a factor of
$r$. Let us handle the case when $\ell,\ell'$ are both even; the
odd case is handled similarly, and gives an additional factor of $\ell\ell'$.

Note that $P$ together with two red edges from $C$ and $C'$ produce
a red path $P'$ of length $t+2$, with two vertices from $R\cap(V(C)\cup V(C'))$
in it. There are $O(r)$ ways to choose the path $P'$, and this choice
determines a red edge of $C$ and a red edge of $C'$. The number
of choices of $C$ and of $C'$ are then reduced by a factor of $r$
each. Combining these factors with (\ref{eq:E(X_l) even bound}),
we obtain the bound
\[
\mathbb{E}\left({\bf X}_{\ell,\ell',t}\right)=O\left(\frac{1}{n}\cdot\left(\frac{2r}{n}\right)^{\frac{\ell}{2}+\frac{\ell'}{2}}\right).
\]
Summing over $\ell,\ell',t$, and adding a factor of $\ell\ell'$
for the odd case, we have
\begin{align*}
\sum_{\ell,\ell'\leq\ln^2 n,\;t\leq\sqrt{n}}\mathbb{E}\left({\bf X}_{\ell,\ell',t}\right) & =\sqrt{n}\cdot O\left(\frac{1}{n}\right)\cdot\sum_{\ell,\ell'\leq\ln^2 n}\ell\ell'(2\xi)^{\frac{\ell}{2}+\frac{\ell'}{2}}\\
&=O\left(\frac{\ln^6n}{\sqrt{n}}\right)=o(1).
\end{align*}
\end{proof}

\begin{proof}[Proof of Proposition~\ref{prop:alt-cycles_all-in-1}.4.]
Fix $\ell\leq\ln^2 n$ and $2\leq t\leq\sqrt{n}$. Let ${\bf X}_{\ell,t}$
count pairs $(C,P)$ where $C$ is an alternating cycle of length
$\ell$ and $P$ is a red segment of $C$ of length $t$. As above,
we only consider the case when $\ell$ is even; the odd case is handled
similarly, and gves an additional factor of $\ell$.

Note that the endpoints of $P$ are not adjacent by a red edge ---
otherwise this edge would form with $P$ the entire Hamilton cycle.
We apply the usual argument: $P$ together with the two red edges
from $C$ produce a red path $P'$ of length $t+2$ with two vertices
from $R\cap V(C)$ in it. There are $O(r)$ ways to choose the path
$P'$, and this choice determines two red edges of $C$. There are
$O(\ell)$ ways to choose their relative position in $C$. The number
of choices of $C$ is therefore multiplied by a factor of $\frac{\ell}{r^{2}}$.
Combining with (\ref{eq:E(X_l) even bound}), we obtain the bound
\[
\mathbb{E}\left({\bf X}_{\ell,t}\right)=O\left(r\cdot\frac{\ell}{r^{2}}\left(\frac{2r}{n}\right)^{\frac{\ell}{2}}\right)=O\left(\frac{\ell}{n}\cdot\left(\frac{2r}{n}\right)^{\frac{\ell}{2}}\right).
\]
Summing over $\ell,t$, and adding a factor of $\ell\ell'$ for the
odd case, we have
\begin{align*}
\sum_{3\leq\ell\leq\ln^2 n,\;2\leq t\leq\sqrt{n}}\mathbb{E}\left({\bf X}_{\ell,t}\right) &=\sqrt{n}\cdot O\left(\frac{1}{n}\right)\cdot\sum_{3\leq\ell\leq\ln^2 n}\ell^{2}(2\xi)^{\frac{\ell}{2}}\\
&=O\left(\frac{\ln^5 n}{\sqrt{n}}\right)=o(1).
\end{align*}
\end{proof}

Finally, we will need the following property.

\begin{prop}
\label{prop:paths-same-degree}W.h.p.\ in ${\bf H}$, there exist
no two disjoint red paths $P=(x_{1},x_{2},\dots,x_{t})$ and $P'=(x_{1}',x_{2}',\dots,x_{t}')$
of length $t\geq\sqrt{n}$ such that ${\rm deg}_{{\bf H}}(x_{i})={\rm deg}_{{\bf H}}(x_{i}')$
for every $1\leq i\leq t$. 
\end{prop}

\begin{proof}
We prove the proposition in the random multigraph ${\bf H}^{*}$.
Then, from Proposition \ref{prop:simplicity_probability} it is also
true in ${\bf H}$. Moreover, it is sufficient to prove the proposition
only for $t=\sqrt{n}$ due to monotonicity. 

The number of disjoint labelled red paths $P=(x_{1},x_{2},\dots,x_{t})$
and $P'=(x_{1}',x_{2}',\dots,x_{t}')$ is at most $(2n)^{2}=4n^{2}$.
Let us fix such a pair $(P,P')$ and bound $\mathbb{P}(\mathcal{D})$,
where 
\[
\mathcal{D}=\mathcal{D}_{P,P'}=\left\{ {\rm deg}(x_{i})={\rm deg}(x_{i}')\quad\forall1\leq i\leq t\right\} .
\]
Note that ${\rm deg}(x)$ refers to the degree in ${\bf H}^{*}$. 

We begin by defining some notation. Let ${\bf Y}=\left\{ y[x]:x\in R\right\} $.
In ${\bf H}^{*}$, by definition, ${\bf Y}$ is distributed uniformly
over all subsets of $[n]$ of size $r$. For every $v\in[n]$, 
${\rm deg}(v)=2+\mathds{1}_{v\in{\bf Y}}+\mathds{1}_{v\in R}$. Let ${\bf z}=\left(\mathds{1}_{x_1\in{\bf Y}},\dots,\mathds{1}_{x_t\in{\bf Y}}\right)$ be a random binary vector, whose value precisely determines the
intersection of $P$ with ${\bf Y}$. We bound $\mathbb{P}\left(\mathcal{D}\right)$
by conditioning by the value of ${\bf z}$. 

So, let $z=\left(I_{x_{1}},I_{x_{2}},\dots,I_{x_{t}}\right)$ be a
binary vector and consider $\mathbb{P}\left(\mathcal{D}\bigm|{\bf z}=z\right)$.
Let $j=\sum_{i=1}^{t}I_{x_{i}}$ (the number of $1$-entries in $z$).
Then, given ${\bf z}=z$, the subset $\tilde{{\bf Y}}={\bf Y}\setminus P$
is distributed uniformly over all subsets of $[n]\setminus P$ of
size $r-j$. The event $\mathcal{D}$ is saying that $\mathds{1}_{x'_i\in{\bf Y}}=I_{x_{i}}+\mathds{1}_{x_i\in R}-\mathds{1}_{x'_i\in R}$
for every $1\leq i\leq t$. If there exists $i$ such that $\mathds{1}_{x'_i\in{\bf Y}}=I_{x_{i}}+\mathds{1}_{x_i\in R}-\mathds{1}_{x'_i\in R}\not\in\{0,1\}$,
then of course $\mathbb{P}\left(\mathcal{D}\bigm|{\bf z}=z\right)=0$.
Otherwise, we can partition the vertices of $P'$ into 
\begin{eqnarray*}
B & = & \left\{ x_{i}':\,I_{x_{i}}+\mathds{1}_{x_i\in R}-\mathds{1}_{x'_i\in R}=1\right\} ,\\
C & = & \left\{ x_{i}':\,I_{x_{i}}+\mathds{1}_{x_i\in R}-\mathds{1}_{x'_i\in R}=0\right\} .
\end{eqnarray*}
Then $\mathcal{D}$ is the event that $B\subseteq\tilde{{\bf Y}}$ and
$\tilde{{\bf Y}}\cap C=\emptyset$. 
\begin{eqnarray*}
\mathbb{P}(\mathcal{D}\mid{\bf z}=z) & = & \frac{{n-2t \choose r-j-t}}{{n-t \choose r-j}}\leq\frac{{n-2t \choose r-j}}{{n-t \choose r-j}}=\frac{(n-2t)_{r-j}}{(n-t)_{r-j}}\leq\frac{(n-t)_{r-j}}{(n)_{r-j}}\\
 & = & O\left(\prod_{i=0}^{r-1}\left(1-\frac{t}{n-i}\right)\right)=O\left(\exp\left(-\sum_{i=0}^{r-1}\frac{t}{n-i}\right)\right).
\end{eqnarray*}
Now, since $r\geq cn$ for a positive $c$,
\begin{align*}
\sum_{i=0}^{r-1}\frac{1}{n-i}=\frac{1}{n}\sum_{i=0}^{r-1}\frac{1}{1-\frac{i}{n}}
&\geq(1+o(1))\int_{0}^{c}\frac{{\rm d}x}{1-x}\\
&=(1+o(1))\ln\left(\frac{1}{1-c}\right).
\end{align*}
We deduce
\[
\mathbb{P}(\mathcal{D}\mid{\bf z}=z)\leq\exp\left(-\Theta(t)\right)=\exp\left(-\Theta\left(\sqrt{n}\right)\right).
\]
This bound is true for every value $z$ (uniformly), therefore it
also holds for $\mathbb{P}\left(\mathcal{D}\right)$. Now, taking a union
bound over the pairs $(P,P')$, we have
\[
\mathbb{P}\left(\bigcup_{(P,P')}\mathcal{D}_{P,P'}\right)\leq4n^{2}\cdot\exp\left(-\Theta\left(\sqrt{n}\right)\right)=o(1)
\]
and that finishes the proof.
\end{proof}


\subsection{Proof of asymmetry}
\label{subsec:Proving-asymmetry}

Let $\Sigma_{n}$ be the group of permutations of $[n]$. We need
to prove that w.h.p., for every $\sigma\in\Sigma_{n}\setminus\{{\rm id}\}$,
$\sigma$ is not an automorphism of ${\bf H}$. Let $\Sigma_{n}^{[h]}\subset\Sigma_{n}$
be the subgroup of $2n$ automorphisms of $H_{h}=(1,2,\dots,n,1).$
Note that these are exactly the permutations which preserve edge colors. 

We begin by proving that w.h.p., no permutation $\sigma\in\Sigma_{n}\setminus\Sigma_{n}^{[h]}$
is an automorphism of ${\bf H}$. Suppose there exists $\sigma\in\Sigma_{n}\setminus\Sigma_{n}^{[h]}$
which is an automorphism of ${\bf H}$. Then there exists a red edge
$\{x_1,x_2\}$ such that $\sigma(\{x_1,x_2\})$ is blue. Consider the two red edges incident to $\sigma(x_{2})$; write them
as $\{\sigma(x_{2}),\sigma(x_3)\}$ and $\{\sigma(x_{2}),\sigma(x'_3)\}$.
Since $x_{2}$ is also incident to exactly two red edges, and $\{x_{1},x_{2}\}$
is red, one of $\{x_{2},x_{3}\},\{x_{2},x_{3}'\}$ must be blue. Assuming that $\{x_2,x_3\}$ is blue, and applying the same argument but for $\sigma^{-1}$ applied to $\{\sigma(x_2),\sigma(x_3)\}$, we would get that there exists a red edge $\{x_3,x_4\}$ such that $\{\sigma(x_3),\sigma(x_4)\}$ is blue. Proceeding by induction,
we obtain an alternating walk such that its image under $\sigma$
is also an alternating walk, with colors switched. Since the graph
is finite, at some point the walk meets itself, and we would get an alternating cycle $C$ such that $\sigma(C)$ is also an alternating
cycle, with edge colors switched by $\sigma$.
If $r=o(n)$, we are done: from Proposition \ref{prop:alt-cycles-are-small}
we know that w.h.p.\ ${\bf H}$ contains no alternating cycles. So
assume $r\geq cn$ for a positive constant $c$. Since labeled $C$ and $\sigma(C)$ are different, due to Proposition~\ref{prop:alt-cycles-are-small}, we may assume that $C$ has length at most $\ln^2n$.

Let $u_{1}$ be a vertex in $C$, incident to a red edge and a blue
edge in $C$, such that $\sigma(u_{1})\neq u_{1}$. It is obvious that such a vertex exists when $C$ is even. If $C$ is odd, then it has a unique vertex adjacent to two edges of the same color. If all other vertices are fixed, then there exists an edge in $C$ that maps to itself, and so its image has the same color --- a contradiction.

Let $\{u_{1},v\}$ be the red edge incident to $u_{1}$ in $C$. Let
$\{u_{1},u_{2}\}$ be the second red edge incident to $u_{1}$. From Proposition \ref{prop:alt-cycles_all-in-1}, we may assume that $u_{2}\not\in C$, $\{\sigma(u_{1}),\sigma(u_{2})\}$ is red, and
all four vertices $u_{1},u_{2},\sigma(u_{1}),\sigma(u_{2})$ must
be distinct. We can now continue inductively, constructing two disjoint red paths
$(u_{1},u_{2},u_{3},\dots)$ and $(\sigma(u_{1}),\sigma(u_{2}),\sigma(u_{3}),\dots)$.
Since Proposition~\ref{prop:alt-cycles_all-in-1}
forbids any red distances or red segments of length at most $\sqrt{n}$,
in particular we can proceed until the red paths are of length $t\geq\sqrt{n}$.
We get two directed disjoint red paths of length at least $\sqrt{n}$
with the same sequence of degrees, so Proposition \ref{prop:paths-same-degree}
applies. We see that the assumption that ${\bf H}$ is $\sigma$-symmetric
leads to the existence of one of the rare configurations from the
previous subsection in ${\bf H}$, which happens with probability
$o(1)$.

Finally, it remains to show that w.h.p., no permutation $\sigma\in\Sigma_{n}^{[h]}\setminus\left\{ {\rm id}\right\} $
is an automorphism of ${\bf H}$. Recall that these are the permutations
which preserve edge colors. This time we separately consider the cases
$r\geq30\ln n$ and $r\leq30\ln n$.

Start with $r\geq30\ln n$. We show that for a given $\sigma\in\Sigma_{n}^{[h]}\setminus\left\{ {\rm id}\right\} $,
the probability that it is an automorphism of ${\bf H}$
is $o\left(1/n\right)$. Then, it would only remain to take
a union bound over $2n-1$ possible $\sigma$. 

We construct a set of vertices $S$ (depending only on $\sigma$)
with $\left|S\right|=\left\lfloor \frac{n}{3}\right\rfloor $ and
$\left|\sigma(S)\cap R\right|\geq\left\lfloor \frac{r}{3}\right\rfloor $
and $S\cap\sigma(S)=\emptyset$. The construction is done with the
following algorithm.
\begin{enumerate}
\item Start with $S=\emptyset$. 
\item Choose a vertex $x\in\sigma^{-1}(R)$ such that $\sigma(x)\neq x$
and $x\not\in S\cup\sigma(S)\cup\sigma^{-1}(S)$. Update $S\leftarrow S\cup\{x\}$.
Repeat this step $\left\lfloor \frac{r}{3}\right\rfloor $ times.
\item Choose a vertex $x\in[n]$ such that $\sigma(x)\neq x$ and $x\not\in S\cup\sigma(S)\cup\sigma^{-1}(S)$.
Update $S\leftarrow S\cup\{x\}$. Repeat this step $\left\lfloor \frac{n}{3}\right\rfloor -\left\lfloor \frac{r}{3}\right\rfloor $
times.
\end{enumerate}
Since $\sigma$ is either a rotation or a reflection of the Hamiltonian
cycle $H_{h}$, it has at most $2$ fixed points. Therefore, at the
$i$-th step there are at most $2+3(i-1)=3i-1$ forbidden vertices.
This shows that indeed we can repeat the second step $\left\lfloor \frac{r}{3}\right\rfloor $
times, and then repeat the third step $\left\lfloor \frac{n}{3}\right\rfloor -\left\lfloor \frac{r}{3}\right\rfloor $
times. By construction $S\cap\sigma(S)=\emptyset$, $\left|\sigma(S)\cap R\right|\geq\left\lfloor \frac{r}{3}\right\rfloor \geq\left\lfloor 10\ln n\right\rfloor $
and $\left|S\right|=\left\lfloor \frac{n}{3}\right\rfloor $. Also let $T\subseteq S$ such that $\sigma(T)\subseteq\sigma(S)\cap R$
and $\left|T\right|=\left\lfloor 10\ln n\right\rfloor $.

Let ${\bf B}\left[T,S\right]$
denote the set of blue edges with an endpoint in $T$ and the other
endpoint in $S$. If $\sigma$ is an automorphism
of ${\bf H}$ then the event $\sigma\left({\bf B}[T,S]\right)={\bf B}[\sigma(T),\sigma(S)]$ holds. 
So it remains to show that, for every possible value $B$ of ${\bf B}[T,S]$,
\[
\mathbb{P}\left({\bf B}[\sigma(T),\sigma(S)]=\sigma(B)\bigm|{\bf B}[T,S]=B\right)=o\left(1/n\right)
\]
uniformly over $B$. 

Consider the vertices $\left\{ {\bf y}[x]:x\in\sigma(T)\right\} $.
Given ${\bf B}[T,S]=B$, each ${\bf y}[x]$ can still be any vertex
from $\sigma(S)$ (since $S\cap\sigma(S)=\emptyset$). However, the
event ${\bf B}[\sigma(T),\sigma(S)]=\sigma(B)$ restricts ${\bf y}[x]$
to at most $2$ possible values from $\sigma(S)$ (since there are at most two blue edges between $x$ and $S$). Recall that $\left|S\right|=\left\lfloor \frac{n}{3}\right\rfloor $ and that $\left|T\right|=\lfloor 10\ln n\rfloor$,
so given ${\bf B}[T,S]=B$, the event ${\bf B}[\sigma(T),\sigma(S)]=\sigma(B)$
restricts the number of values of each ${\bf y}[x]$ by a factor of
$\frac{2}{3}+o(1)$. We therefore get the following coarse bound:
\[
\mathbb{P}\left({\bf B}[\sigma(S)]=\sigma(B)\bigm|{\bf B}[S]=B\right)\leq\left(2/3+o(1)\right)^{\left|T\right|}=o\left(1/n\right).
\]

In the case $r\leq30\ln n$ we use a different argument. Fix $\sigma\in\Sigma_{n}^{[h]}\setminus\left\{ {\rm id}\right\} $.
If $\sigma$ is an automorphism of ${\bf H}$, one of the following
events must hold:
\begin{enumerate}
\item $E_{\sigma}^{(1)}$, the event that $\sigma(e)=e$ for every blue
edge $e$.
\item $E_{\sigma}^{(2)}$, the event that there exist two blue edges $e_{1}\neq e_{2}$
such that $\sigma(e_{1})=e_{2}$. 
\end{enumerate}
The first event fully determines the values ${\bf y}[x]$ for every
$x\in R$ (necessarily ${\bf y}[x]=\sigma(x)$, except when $\sigma$
is a reflection and there are two fixed points $x_{1},x_{2}$, in
which case ${\bf y}[x_{1}]=x_{2}$ and ${\bf y}[x_{2}]=x_{1}$). Therefore
$\mathbb{P}\left(E_{\sigma}^{(1)}\right)\leq p_{0}/(n)_{r}$. 

The second event implies the following event $E$ (which is not dependent
on $\sigma$). For every $x\in R$ let ${\bf d}(x)$ measure the red
distance between $x$ and ${\bf y}[x]$. Let $E$ be the event that
there exist two distinct vertices $x_{1},x_{2}\in R$ such that ${\bf d}(x_{1})={\bf d}(x_{2})$.
It is easy to see that $\mathbb{P}\left(E\right)=\binom{r}{2}\cdot O\left(1/n\right).$
Overall, by the union bound, the probability that ${\bf H}$ is $\sigma$-symmetric
in this case is $o(1)$ since $r\geq2$ and $r=O\left(\ln n\right)$.
That finishes the proof of Theorem \ref{thm:asymmetry} for the case $m<\frac{3}{2}n$. 

\section{Final remarks}
\label{sec:conclusions}

In this paper, we study distinguishability of sparse random graphs by FO sentences. We proved a general upper bound on the minimum quantifier depth of a sentence that distinguishes between independent samples of $G(n,n^{-\alpha})$ that answers the question from~\cite{BZ}, developed a new tool --- random balanced graph --- and studied its properties. We observed a new phenomenon: the FO distinguishability between independent samples of $G(n,n^{-\alpha})$ depends on how well $\alpha$ can be approximated by rational numbers. In particular, from Theorem~\ref{thm:almost_uniform_lower_bound} it follows that the lower bound $\mathbf{k}_{\alpha}\geq\frac{1-o(1)}{\ln 2}\ln\ln\ln n$ w.h.p. holds true for almost all $\alpha\in(0,1)$ 
  --- in particular, for all algebraic irrational $\alpha\in(0,1)$. 

We believe that the general upper bound on $\mathbf{k}_{\alpha}$ in Theorem~\ref{thm:general_upper_bound} can be improved. A possible way to do it is to consider existential sentences that describe existence of a tuple of graphs with a common set of roots defined in such a way that, as soon as, intersections between these graphs become large, the density of their union becomes non-typical. We suspect that 
 a further increase of the width of syntax trees of existential sentences by branching would improve the bound up to $O(\ln\ln n)$, but we do not see a promising way to reach the lower bound in Theorem~\ref{thm:almost_uniform_lower_bound}. 

A more general and involved problem that is very interesting to address is to get a tight dependency between typical values of $\mathbf{k}_{\alpha}$ and an accuracy of approximations of $\alpha$ by rational numbers.

\bibliography{bibliography} 
\bibliographystyle{plain}

\appendix

\section{Asymmetry: dense case\label{sec:Asymmetry_dense}}
\label{appendix:dense_asym}

In this section we prove Theorem~\ref{thm:asymmetry_bollobas}. 
From now on, suppose that $\{{\bf G}_{n}\}_{n=1}^{\infty}$ is a sequence
of random graphs with the following three properties:
\begin{description}
\item [{(P1)}] $\delta({\bf G}_{n})\geq3$.
\item [{(P2)}] There exists a constant $\Delta\geq3$ such that $\Delta({\bf G}_{n})\leq\Delta$ for every $n$.
\item [{(P3)}] There exists a constant $c_{0}>0$ such that $\mathbb{P}\left(E_{0}\subseteq E\left({\bf G}_{n}\right)\right)\leq\left(\frac{c_{0}}{n}\right)^{\left|E_{0}\right|}$
for every $E_{0}\subseteq\binom{[n]}{2}$ and every $n$. 
\end{description}
We need to show that ${\bf G}_{n}$ is asymmetric w.h.p.

\subsection{Orbits Analysis}

Let $\Sigma_{n}$ denote the group of permutations on $[n]$. For
every $\sigma\in\Sigma_{n}$ let $P_{\sigma}$ denote the probability
that $\sigma$ is an automorphism of ${\bf G}_{n}$. Then it suffices
to show that
\begin{equation}
\sum_{\sigma\in\Sigma_{n}\setminus\{{\rm id}\}}P_{\sigma}=o(1).\label{eq:sum of P_sigma}
\end{equation}

To bound $P_{\sigma}$ for a given permutation $\sigma\in\Sigma_{n}$,
we shall analyize the orbits of edges under the action of $\sigma$.
\begin{definition}
Let $E_{n}=\binom{[n]}{2}$. We define a group action of $\Sigma_{n}$
on $E_{n}$ as follows: for every $\sigma\in\Sigma_{n}$ and $\{x,y\}\in E_{n}$,
\[
\sigma\left(\{x,y\}\right)=\{\sigma(x),\sigma(y)\}.
\]
Given a permutation $\sigma\in\Sigma_{n}$ and a pair $e\in E_{n}$,
we define ${\rm Orb}_{\sigma}(e)$ as the orbit of $e$ under the
action of the subgroup $\left\langle \sigma\right\rangle $ generated
by $\sigma$. That is, the $\sigma$-orbit ${\rm Orb}_{\sigma}(e)=\left\{ \sigma^{k}(e):k\in\mathbb{Z}\right\} \subseteq E_{n}$. 
\end{definition}

Note that for every $\sigma$, the quotient set 
is
\[
E_{n}/\left\langle \sigma\right\rangle =\left\{ {\rm Orb}_{\sigma}(e):e\in E_{n}\right\}
\]
defines a partition of $E_{n}$. 

The relation between asymmetry of a graph and orbits is given in the following
simple claim.
\begin{claim}
$\sigma$ as an automorphism of a graph $G=([n],E)$  $\iff$ $E$
does not split $\sigma$-orbits (that is, for every $\sigma$-orbit
$O$, either $O\subseteq E$ or $O\cap E=\emptyset$). 
\label{cl:appA_asym}
\end{claim}

The $\sigma$-orbits in the set of edges $E_{n}$ are closely related
to the $\sigma$-orbits in the vertex set $[n]$, which are simply sets of vertices of cycles in the cycle decomposition of $\sigma$. The following proposition is straightforward.
\begin{prop}
\label{prop:orbits_and_cycles}Let $e=\{x_{1},x_{2}\}\in E_{n}$, $x_{1}\in C_{1}$ and $x_{2}\in C_{2}$ where $C_{1},C_{2}$
are cycles from the cycle decomposition of $\sigma$. Denote $\ell_{1}=\left|C_{1}\right|$
and $\ell_{2}=\left|C_{2}\right|$. 
\begin{itemize}
\item If $C_{1}\neq C_{2}$ then $\left|{\rm Orb}_{\sigma}(e)\right|={\rm lcm}(\ell_{1},\ell_{2})$. 
\item Now suppose $C_{1}=C_{2}$ and let $\ell=\ell_{1}=\ell_{2}$. Then
$\left|{\rm Orb}_{\sigma}(e)\right|=\ell$, unless $\ell$ is even
and $x,y$ are opposite in the cycle; in that case $\left|{\rm Orb}_{\sigma}(e)\right|=\frac{\ell}{2}$.
\end{itemize}
\end{prop}

Proposition \ref{prop:orbits_and_cycles} motivates us to classify
vertices according to their cycle length in $\sigma$. As it turns
out, to properly bound $P_{\sigma}$ for $\sigma\in\Sigma_{n}\setminus\{{\rm id}\}$
we must consider $2$-cycles and $3$-cycles separately from the longer
cycles.

From now on let us fix a permutation $\sigma\in\Sigma_{n}$; all the
following definitions are with respect to the action of $\left\langle \sigma\right\rangle $.
\begin{definition}
We define a partition $[n]=S_{1}\cup S_{2}\cup S_{3}\cup S_{4}$ as
follows. $S_{1},S_{2},S_{3}$ are the sets of vertices in $1$-cycles
(fixed points), $2$-cycles and $3$-cycles of $\sigma$, respectively.
$S_{4}$ is the set of remaining vertices: vertices in $\ell$-cycles
for $\ell\geq4$. We also write $S=S_{2}\cup S_{3}\cup S_{4}$, $s_{i}=\left|S_{i}\right|$
for $i\in\{1,2,3,4\}$ and $s=\left|S\right|$.
\end{definition}

Note that $\sigma\neq{\rm id}$ is equivalent to $S\neq\emptyset$.
\begin{definition}
We classify pairs $e\in E_{n}$ as follows. Suppose $e=\left\{ x,y\right\} $
where $x\in S_{i}$ and $y\in S_{j}$ for some $i,j\in\{1,2,3,4\}$. 
\begin{itemize}
\item If $x,y$ are not opposite vertices in an even cycle,
we say that $e$ is a\emph{ pair of type $(i,j)$} (and also of type
$(j,i)$). 
\item Now suppose that $x,y$ are opposite vertices in an
even cycle of length $\ell$. Note that then $i=j={\scriptscriptstyle \begin{cases}
2 & \ell=2\\
4 & \ell\geq4
\end{cases}}$. In this case we say that $e$ is a \emph{pair of type $(i,0)$ }(and
also of type $(0,i)$).
\end{itemize}
A pair of type $(i,j)$ is also called an \emph{$(i,j)$-pair}.
\end{definition}

\begin{remark}
We make several remarks about the last definition. 
\begin{itemize}
\item The types $(i,j)$ and $(j,i)$ are identical.
\item Two pairs in $E_{n}$ from the same $\sigma$-orbit have the same
type. An orbit of an $(i,j)$-pair is called an \emph{$(i,j)$-orbit. }
\item There are 12 distinct possible types, arranged in the table below.
\[
\begin{array}{cccc}
 & (0,2), &  & (0,4),\\
(1,1), & (1,2), & (1,3), & (1,4),\\
 & (2,2), & (2,3), & (2,4),\\
 &  & (3,3), & (3,4),\\
 &  &  & (4,4).
\end{array}
\]
\end{itemize}
\end{remark}

\begin{definition}
Following the last definitions and remarks, we introduce some additional notations.
\begin{itemize}
\item For every type $(i,j)$ let $T_{ij}$ be the set of $(i,j)$-pairs.
Also denote $t_{ij}=\left|T_{ij}\right|$.
\item For every type $(i,j)$, let $h_{ij}$ denote the minimum possible
size of an $(i,j)$-orbit. The values of $h_{ij}$ for all distinct
possible types are arranged in the table below.
\[
\begin{array}{cccc}
 & h_{02}=1, &  & h_{04}=2,\\
h_{11}=1, & h_{12}=2, & h_{13}=3, & h_{14}=4,\\
 & h_{22}=2, & h_{23}=6, & h_{24}=4,\\
 &  & h_{33}=3, & h_{34}=6,\\
 &  &  & h_{44}=4.
\end{array}
\]
\end{itemize}
\end{definition}

\subsection{Lists of parameters}

We shall bound $P_{\sigma}$ by taking into account the number of
edges of each type in ${\bf G}_{n}$; we call these numbers the \emph{parameters}
of the graph. Actually, our definition excludes the trivial type $(1,1)$
from the list of parameters, since it will not contribute anything
to the bound. The following definitions are still with respect to
a given $\sigma\in\Sigma_{n}$.
\begin{definition}
Let ${\cal I}$ be the set of all $11$ non-trivial distinct types:
\[
{\cal I}=\left\{ \begin{array}{ccc}
(0,2), &  & (0,4),\\
(1,2), & (1,3), & (1,4),\\
(2,2), & (2,3), & (2,4),\\
 & (3,3), & (3,4),\\
 &  & (4,4)
\end{array}\right\} .
\]
\end{definition}

\begin{definition}
Let $G=(V,E)$ be a graph. Its \emph{list of parameters} is the $11$-tuple
$(k_{ij})_{(i,j)\in{\cal I}}$ where $k_{ij}=\left|E\cap T_{ij}\right|$
for every $(i,j)\in{\cal I}$. That is, $k_{ij}$ is the number of
edges of type $(i,j)$ in $G$.
\end{definition}

For convenience, 
 lists of parameters are denoted $(k_{ij})$. We also let $k_{ji}=k_{ij}$
for every $(i,j)$. At this point we slightly deviate from Bollob\'{a}s's
original definitions, in order to adjust to the non-regular case.
\begin{prop}
\label{prop:permissible}Let $G=(V,E)$ be a graph with the list of parameters
$(k_{ij})$. For every $i\in\{2,3,4\}$ let $r_{i}$ be the average
degree of vertices from $S_{i}$ in $G$. Then the following equations
hold:
\begin{equation}
\begin{cases}
r_{2}s_{2} & =2k_{20}+k_{21}+2k_{22}+k_{23}+k_{24},\\
r_{3}s_{3} & =k_{31}+k_{32}+2k_{33}+k_{34},\\
r_{4}s_{4} & =2k_{40}+k_{41}+k_{42}+k_{43}+2k_{44}.
\end{cases}\label{eq:permissible}
\end{equation}
\end{prop}

\begin{proof}
The three equations come from counting the sum of degrees of $S_{i}$
for $i=2,3,4$. From the definition of $r_{i}$ this sum is $r_{i}s_{i}$.
For the right hand side, note that it indeed counts all the edges
incident to $S_{i}$, with edges with both endpoints in $S_{i}$ counted
twice. The counting is done by considering the different possible
types of the edges. 
\end{proof}

\subsection{Bounding $P_{\sigma}$}

We are now ready for the main part of the proof: bounding $P_{\sigma}$,
the probability that $\sigma$ is an automorphism of ${\bf G}_{n}$.
We shall prove a bound which implies Equation (\ref{eq:sum of P_sigma}),
and that will finish the proof.

Let $({\bf k}_{ij})$ be the list of parameters of the random graph
${\bf G}_{n}$.
\begin{definition}
A list $(k_{ij})$ is called \emph{$\sigma$-valid} if 
\[
P_{\sigma,(k_{ij})} := \mathbb{P}\left(\sigma \text{ is an automorphism of } {\bf G}_{n}\;{\rm and}\;({\bf k}_{ij})=(k_{ij})\right)>0.
\]
\end{definition}

By definition, $P_{\sigma}=\sum_{(k_{ij})}P_{\sigma,(k_{ij})}$ where
the sum is over all $\sigma$-valid lists. To bound $P_{\sigma}$,
we first bound $P_{\sigma,(k_{ij})}$ for a given $\sigma$-valid
list $(k_{ij})$. 

We fix a $\sigma$-valid list $(k_{ij})$. Then $P_{\sigma,(k_{ij})}$ is the probability that, for every $(i,j)\in{\cal I}$,
${\bf G}_{n}$ contains exactly $k_{ij}$ edges of type $(i,j)$,
and also does not split $(i,j)$-orbits by Claim~\ref{cl:appA_asym}.
\begin{itemize}
\item Let $f_{ij}$ denote the number of sets $E_{ij}$ of size $k_{ij}$
which can be written as unions of $(i,j)$-orbits.
\item Let $k=\sum_{(i,j)\in{\cal I}}k_{ij}$ be the number of edges of non-trivial
types. Equivalently, $k$ is the number of edges incident to $S$. 
\end{itemize}
\begin{prop}
Following the above definitions, we have
\[
P_{\sigma,(k_{ij})}\leq\left(\frac{c_{0}}{n}\right)^{k}\prod_{(i,j)\in{\cal I}}f_{ij}
\]
where $c_{0}$ is the constant from property \emph{(P3)}.
\end{prop}

\begin{proof}
There are $\prod_{(i,j)\in{\cal I}}f_{ij}$ ways to choose, for every
$(i,j)\in{\cal I}$, a set $E_{ij}\subseteq T_{ij}$ of size $k_{ij}$
which is a union of $(i,j)$-orbits. Given such a choice of sets $E_{ij}$,
let $E_{0}=\bigcup_{(i,j)\in{\cal I}}E_{ij}$. From (P3) we have 
\[
\mathbb{P}\left(E_{0}\subseteq E({\bf G}_{n})\right)\leq\left(\frac{c_{0}}{n}\right)^{\left|E_{0}\right|}=\left(\frac{c_{0}}{n}\right)^{k}.
\]
The union bound completes the proof.
\end{proof}
The following simple combinatorial lemma allows us to bound the terms $f_{ij}$.
\begin{lemma}
\label{lem:combinatorial}Let ${\cal A}$ be a finite collection of
mutually disjoint finite sets, which contains exactly $m_{i}$ sets
of size $i$ for every $i\geq1$. Suppose there exists a positive
integer $h$ such that $m_{i}=0$ for every $i<h$. Let $f(k;{\cal A})$
denote the number of sets of size $k$ which are unions of some of
the sets from ${\cal A}$. Then $f(k;{\cal A})\leq\binom{\left\lfloor \nicefrac{t}{h}\right\rfloor }{\left\lfloor \nicefrac{k}{h}\right\rfloor }$
where $t=\sum_{i}im_{i}$ is the total size of all sets from ${\cal A}$. 
\end{lemma}

The proof is fairly straightforward and is given in \cite[Equation (5)]{Bollobas-asymmetry}.

To effectively bound the terms $f_{ij}$ with Lemma \ref{lem:combinatorial},
in a way which applies generally for all lists $(k_{ij})$, we must
separate ``large'' $k_{ij}$ from ``small'' $k_{ij}$, which motivates
the use of $\varepsilon_{ij}$ below. 
\begin{prop}
For every $(i,j)\in{\cal I}$ let $\ell_{ij}=\frac{k_{ij}}{h_{ij}}$
and 
\[
\varepsilon_{ij}=\begin{cases}
1, & \left\lfloor \ell_{ij}\right\rfloor >\frac{s}{22000};\\
0, & \left\lfloor \ell_{ij}\right\rfloor \leq\frac{s}{22000}.
\end{cases}
\]
Then $\prod_{(i,j)\in{\cal I}}f_{ij}\leq n^{\sum_{j=2}^{4}\varepsilon_{1j}\ell_{1j}}\cdot s^{\sum_{2\leq i\leq j\leq4}\varepsilon_{ij}\ell_{ij}}\cdot c_{1}^{k}\cdot n^{\frac{s}{1000}}$
for some positive constant $c_{1}$. 
\label{prop:app_A_prop}
\end{prop}

\begin{proof}
For every $(i,j)\in{\cal I}$, Lemma \ref{lem:combinatorial} applies
where ${\cal A}$ is the set of $(i,j)$-orbits, $k=k_{ij}$, $h=h_{ij}$
and $t=t_{ij}$. Therefore
\[
f_{ij}\leq\binom{\left\lfloor \nicefrac{t_{ij}}{h_{ij}}\right\rfloor }{\left\lfloor \nicefrac{k_{ij}}{h_{ij}}\right\rfloor }=\binom{\left\lfloor \nicefrac{t_{ij}}{h_{ij}}\right\rfloor }{\left\lfloor \ell_{ij}\right\rfloor }\leq\binom{t_{ij}}{\left\lfloor \ell_{ij}\right\rfloor }.
\]
When $\varepsilon_{ij}=0$, we are satisfied with the very coarse
bound
\[
\binom{t_{ij}}{\left\lfloor \ell_{ij}\right\rfloor }\leq t_{ij}^{\left\lfloor \ell_{ij}\right\rfloor }\leq(n^{2})^{\nicefrac{s}{22000}}=n^{\nicefrac{s}{11000}}.
\]
When $\varepsilon_{ij}=1$, we apply the bound
\[
\binom{t_{ij}}{\left\lfloor \ell_{ij}\right\rfloor }\leq\left(\frac{{\rm e}t_{ij}}{\left\lfloor \ell_{ij}\right\rfloor }\right)^{\left\lfloor \ell_{ij}\right\rfloor }\leq\left(\frac{{\rm e}t_{ij}}{\nicefrac{s}{22000}}\right)^{\ell_{ij}}=\left(c_{1}\frac{t_{ij}}{s}\right)^{\ell_{ij}}
\]
where $c_{1}=22000{\rm e}$. The bound on $t_{ij}$ itself depends on the type $(i,j)$. We have
three cases:
\begin{itemize}
\item $2\leq i\leq j\leq4$. In this case $t_{ij}\leq s_{i}s_{j}\leq s^{2}$
and $f_{ij}\leq\left(c_{1}s\right)^{\ell_{ij}}$. 
\item $i=1$ and $2\leq j\leq4$. In this case $t_{1j}\leq s_{1}s_{j}\leq ns$
and $f_{ij}\leq\left(c_{1}n\right)^{\ell_{ij}}$. 
\item $i=0$ and $j\in\{2,4\}$. In this case $t_{0j}=\frac{s_{j}}{2}\leq\frac{s}{2}$
and $f_{ij}\leq\left(\frac{c_{1}}{2}\right)^{\ell_{ij}}$. 
\end{itemize}
Overall,
\begin{eqnarray*}
\prod_{(i,j)\in{\cal I}}f_{ij} & \leq & n^{\sum_{j=2}^{4}\varepsilon_{1j}\ell_{1j}}\cdot s^{\sum_{2\leq i\leq j\leq4}\varepsilon_{ij}\ell_{ij}}\cdot c_{1}^{\sum_{(i,j)\in{\cal I}}\ell_{ij}}\cdot n^{\sum_{(i,j)\in{\cal I}}\frac{s}{11000}}\\
 & \leq & n^{\sum_{j=2}^{4}\varepsilon_{1j}\ell_{1j}}\cdot s^{\sum_{2\leq i\leq j\leq4}\varepsilon_{ij}\ell_{ij}}\cdot c_{1}^{k}\cdot n^{\frac{s}{1000}}.
\end{eqnarray*}
\end{proof}
\begin{corollary}
\label{cor:P_sigma,k bound}For large enough $n$.
\[
P_{\sigma,(k_{ij})}\leq n^{\frac{s}{100}}n^{-k}\cdot n^{\sum_{j=2}^{4}\varepsilon_{1j}\ell_{1j}}s^{\sum_{2\leq i\leq j\leq4}\varepsilon_{ij}\ell_{ij}}.
\]
\end{corollary}

\begin{proof}
Due to Proposition~\ref{prop:app_A_prop}, we get
\[
P_{\sigma,(k_{ij})}\leq\left(\frac{c_{0}}{n}\right)^{k}\prod_{(i,j)\in{\cal I}}f_{ij}\leq\left(\frac{c_{0}}{n}\right)^{k}\cdot c_{1}^{k}\cdot n^{\frac{s}{1000}}\cdot n^{\sum_{j=2}^{4}\varepsilon_{1j}\ell_{1j}}\cdot s^{\sum_{2\leq i\leq j\leq4}\varepsilon_{ij}\ell_{ij}}.
\]
In addition,
$$
\left(\frac{c_{0}}{n}\right)^{k}\cdot c_{1}^{k}\cdot n^{\frac{s}{1000}}  =  (c_{0}c_{1})^{k}\cdot n^{\frac{s}{1000}}n^{-k} \overset{(*)}{\leq}  (c_{0}c_{1})^{\Delta s}\cdot n^{\frac{s}{1000}}n^{-k}
  \overset{(**)}{=}  c_{2}^{s}\cdot n^{\frac{s}{1000}}n^{-k} \leq  n^{\frac{s}{100}}n^{-k}
 $$
for large enough $n$. Inequality $(*)$ follows from Property (P2) of the model (and our
assumption that $(k_{ij})$ is $\sigma$-valid). Indeed, recall that
$k$ is the number of edges incident to $S$; in a graph with maximum
degree at most $\Delta$, this number is at most $\Delta s$. In Equality
$(**)$ we simply define $c_{2}=(c_{0}c_{1})^{\Delta}$. 
\end{proof}
The next assertion is the main technical part of the proof.
\begin{lemma}
\label{thm:P_sigma,k bound}
For every $\sigma$-valid list $(k_{ij})$,
\[
P_{\sigma,(k_{ij})}\leq\frac{1}{n^{(1+\frac{1}{100})s}}s_{2}^{\nicefrac{s_{2}}{2}}s_{3}^{\nicefrac{s_{3}}{3}}.
\]
\end{lemma}

Note that the bound depends only on $\sigma$, and not on $(k_{ij})$.
\begin{proof}
From Corollary \ref{cor:P_sigma,k bound}, it suffices to prove
\[
\left[n^{\frac{s}{100}}n^{-k}\cdot n^{\sum_{j=2}^{4}\varepsilon_{1j}\ell_{1j}}s^{\sum_{2\leq i\leq j\leq4}\varepsilon_{ij}\ell_{ij}}\right]\cdot n^{s}s_{2}^{\nicefrac{-s_{2}}{2}}s_{3}^{\nicefrac{-s_{3}}{3}}\leq n^{-\frac{s}{100}}.
\]
Let us denote
\[
A=n^{s-k}\cdot n^{\sum_{j=2}^{4}\varepsilon_{1j}\ell_{1j}}s^{\sum_{2\leq i\leq j\leq4}\varepsilon_{ij}\ell_{ij}}\cdot s_{2}^{\nicefrac{-s_{2}}{2}}s_{3}^{\nicefrac{-s_{3}}{3}}.
\]
Then we need to prove $n^{\frac{s}{100}}A\leq n^{-\frac{s}{100}}$.
Equivalently, we need to prove
\begin{equation}
-\frac{\ln A}{\ln n}\geq\frac{2}{100}s.\label{eq:-ln(A)/ln(n)_required_bound}
\end{equation}
Let us bound $-\frac{\ln A}{\ln n}$ from below. Denote $\alpha=\frac{\ln s}{\ln n}$;
then
\begin{equation}
-\frac{\ln A}{\ln n}=(k-s)-\sum_{j=2}^{4}\varepsilon_{1j}\ell_{1j}-\alpha\sum_{2\leq i\leq j\leq4}\varepsilon_{ij}\ell_{ij}+\frac{1}{2}s_{2}\frac{\ln s_{2}}{\ln n}+\frac{1}{3}s_{3}\frac{\ln s_{3}}{\ln n}.\label{eq:ln(A)/ln(n)}
\end{equation}
At this point we invoke Equations (\ref{eq:permissible}). 
 Then
\begin{eqnarray*}
k-s & = & \sum_{(i,j)\in{\cal I}}k_{ij}-(s_{2}+s_{3}+s_{4})\\
 & = & \left(1-\frac{2}{r_{2}}\right)k_{20}+\left(1-\frac{1}{r_{2}}\right)k_{21}+\left(1-\frac{2}{r_{2}}\right)k_{22}\\
 & + & \left(1-\frac{1}{r_{2}}-\frac{1}{r_{3}}\right)k_{23}+\left(1-\frac{1}{r_{2}}-\frac{1}{r_{4}}\right)k_{24}+\left(1-\frac{1}{r_{3}}\right)k_{31}+\left(1-\frac{2}{r_{3}}\right)k_{33}\\
 & + & \left(1-\frac{1}{r_{3}}-\frac{1}{r_{4}}\right)k_{34}+\left(1-\frac{2}{r_{4}}\right)k_{40}+\left(1-\frac{1}{r_{4}}\right)k_{41}+\left(1-\frac{2}{r_{4}}\right)k_{44}.
\end{eqnarray*}
In addition,
\begin{eqnarray*}
\sum_{j=2}^{4}\varepsilon_{1j}\ell_{1j} & = & \varepsilon_{12}\frac{k_{12}}{2}+\varepsilon_{13}\frac{k_{13}}{3}+\varepsilon_{14}\frac{k_{14}}{4},\\
\alpha\sum_{2\leq i\leq j\leq4}\varepsilon_{ij}\ell_{ij} & = & \alpha\varepsilon_{22}\frac{k_{22}}{2}+\alpha\varepsilon_{23}\frac{k_{23}}{6}+\alpha\varepsilon_{24}\frac{k_{24}}{4}+\alpha\varepsilon_{33}\frac{k_{33}}{3}+\alpha\varepsilon_{34}\frac{k_{34}}{6}+\alpha\varepsilon_{44}\frac{k_{44}}{4}.
\end{eqnarray*}
As for $\frac{1}{2}s_{2}\frac{\ln s_{2}}{\ln n}$, we have
\begin{eqnarray*}
\frac{1}{2}s_{2}\frac{\ln s_{2}}{\ln n} & = & \frac{1}{2\ln n}\left[s_{2}\ln s-s_{2}\ln\left(\frac{s}{s_{2}}\right)\right]\\
 & \geq & \frac{1}{2\ln n}\left[s_{2}\ln s-s_{2}\cdot\frac{s}{s_{2}}\right]\\
 & = & \frac{\alpha s_{2}}{2}-\frac{s}{2\ln n}
\end{eqnarray*}
and similarly 
\begin{eqnarray*}
\frac{1}{3}s_{3}\frac{\ln s_{3}}{\ln n} & \geq & \frac{\alpha s_{3}}{3}-\frac{s}{3\ln n}.
\end{eqnarray*}
Again, we apply Equations (\ref{eq:permissible}) to write
\begin{eqnarray*}
\frac{\alpha s_{2}}{2} & = & \frac{\alpha}{2r_{2}}\left(2k_{20}+k_{21}+2k_{22}+k_{23}+k_{24}\right),\\
\frac{\alpha s_{3}}{3} & = & \frac{\alpha}{3r_{3}}\left(k_{31}+k_{32}+2k_{33}+k_{34}\right).
\end{eqnarray*}
Overall, putting everything back into (\ref{eq:ln(A)/ln(n)}), we
have the following lower bound on $-\frac{\ln A}{\ln n}$. We use
colors to help keeping track of the terms.
\begin{eqnarray*}
-\frac{\ln A}{\ln n} & \geq & -\frac{s}{\ln n}\textcolor{red}{+(k-s)}\textcolor{blue}{-\sum_{j=2}^{4}\varepsilon_{1j}\ell_{1j}}\textcolor{green}{-\alpha\sum_{2\leq i\leq j\leq4}\varepsilon_{ij}\ell_{ij}}\textcolor{orange}{+\frac{\alpha s_{2}}{2}}\textcolor{purple}{+\frac{\alpha s_{3}}{3}}\\
 & = & -\frac{s}{\ln n}+\left(\textcolor{red}{1-\frac{2}{r_{2}}}\textcolor{orange}{+\frac{\alpha}{r_{2}}}\right)k_{20}+\left(\textcolor{red}{1-\frac{1}{r_{2}}}\textcolor{blue}{-\frac{\varepsilon_{21}}{2}}\textcolor{orange}{+\frac{\alpha}{2r_{2}}}\right)k_{21}\\
 & + & \left(\textcolor{red}{1-\frac{2}{r_{2}}}\textcolor{green}{-\frac{\alpha\varepsilon_{22}}{2}}\textcolor{orange}{+\frac{\alpha}{r_{2}}}\right)k_{22}+\left(\textcolor{red}{1-\frac{1}{r_{2}}-\frac{1}{r_{3}}}\textcolor{green}{-\frac{\alpha\varepsilon_{23}}{6}}\textcolor{orange}{+\frac{\alpha}{2r_{2}}}\textcolor{purple}{+\frac{\alpha}{3r_{3}}}\right)k_{23}\\
 & + & \left(\textcolor{red}{1-\frac{1}{r_{2}}-\frac{1}{r_{4}}}\textcolor{green}{-\frac{\alpha\varepsilon_{24}}{4}}\textcolor{orange}{+\frac{\alpha}{2r_{2}}}\right)k_{24}+\left(\textcolor{red}{1-\frac{1}{r_{3}}}\textcolor{blue}{-\frac{\varepsilon_{31}}{3}}\textcolor{purple}{+\frac{\alpha}{3r_{3}}}\right)k_{31}\\
 & + & \left(\textcolor{red}{1-\frac{2}{r_{3}}}\textcolor{green}{-\frac{\alpha\varepsilon_{33}}{3}}\textcolor{purple}{+\frac{2\alpha}{3r_{3}}}\right)k_{33}+\left(\textcolor{red}{1-\frac{1}{r_{3}}-\frac{1}{r_{4}}}\textcolor{green}{-\frac{\alpha\varepsilon_{34}}{6}}\textcolor{purple}{+\frac{\alpha}{3r_{3}}}\right)k_{34}\\
 & + & \left(\textcolor{red}{1-\frac{2}{r_{4}}}\right)k_{40}+\left(\textcolor{red}{1-\frac{1}{r_{4}}}\textcolor{blue}{-\frac{\varepsilon_{41}}{4}}\right)k_{41}+\left(\textcolor{red}{1-\frac{2}{r_{4}}}\textcolor{green}{-\frac{\alpha\varepsilon_{44}}{4}}\right)k_{44}.
\end{eqnarray*}
We write simpler lower bounds on the terms in the brackets. In negative
summands with $\varepsilon_{ij}$ we replace it with $1$. In
addition, we omit some positive orange and purple summands.
\begin{eqnarray*}
-\frac{\ln A}{\ln n} & \geq & -\frac{s}{\ln n}+\left(\textcolor{red}{1-\frac{2}{r_{2}}}\right)k_{20}+\left(\textcolor{red}{1-\frac{1}{r_{2}}}\textcolor{blue}{-\frac{1}{2}}\right)k_{21}\\
 & + & \left(\textcolor{red}{1-\frac{2}{r_{2}}}\textcolor{green}{-\frac{\alpha}{2}}\textcolor{orange}{+\frac{\alpha}{r_{2}}}\right)k_{22}+\left(\textcolor{red}{1-\frac{1}{r_{2}}-\frac{1}{r_{3}}}\textcolor{green}{-\frac{\alpha}{6}}\textcolor{orange}{+\frac{\alpha}{2r_{2}}}\right)k_{23}\\
 & + & \left(\textcolor{red}{1-\frac{1}{r_{2}}-\frac{1}{r_{4}}}\textcolor{green}{-\frac{\alpha}{4}}\textcolor{orange}{+\frac{\alpha}{2r_{2}}}\right)k_{24}+\left(\textcolor{red}{1-\frac{1}{r_{3}}}\textcolor{blue}{-\frac{1}{3}}\right)k_{31}\\
 & + & \left(\textcolor{red}{1-\frac{2}{r_{3}}}\textcolor{green}{-\frac{\alpha}{3}}\textcolor{purple}{+\frac{2\alpha}{3r_{3}}}\right)k_{33}+\left(\textcolor{red}{1-\frac{1}{r_{3}}-\frac{1}{r_{4}}}\textcolor{green}{-\frac{\alpha}{6}}\textcolor{purple}{+\frac{\alpha}{3r_{3}}}\right)k_{34}\\
 & + & \left(\textcolor{red}{1-\frac{2}{r_{4}}}\right)k_{40}+\left(\textcolor{red}{1-\frac{1}{r_{4}}}\textcolor{blue}{-\frac{1}{4}}\right)k_{41}+\left(\textcolor{red}{1-\frac{2}{r_{4}}}\textcolor{green}{-\frac{\alpha}{4}}\right)k_{44}.
\end{eqnarray*}
In $k_{22},k_{23},k_{24},k_{33},k_{34},k_{44}$ there are summands
involving $\alpha$. In all six cases, the sum of these summands is
of the form $\alpha x$ where $x\leq0$. This is because $r_{2},r_{3},r_{4}\geq 3$ (unless the respective set $S_i$ is empty, but in that case the corresponding $k$-factor equals 0),
which follows from Property (P1) of the model (and our assumption
that $(k_{ij})$ is $\sigma$-valid). Also note that $\alpha\leq1$
by definition. Therefore, replacing every appearance of $\alpha$
with $1$ can only decrease the sum. By doing so (and simplifying)
we obtain
\begin{eqnarray*}
-\frac{\ln A}{\ln n} & \geq & -\frac{s}{\ln n}+\left(1-\frac{2}{r_{2}}\right)k_{20}+\left(\frac{1}{2}-\frac{1}{r_{2}}\right)k_{21}\\
 & + & \left(\frac{1}{2}-\frac{1}{r_{2}}\right)k_{22}+\left(\frac{5}{6}-\frac{1}{2r_{2}}-\frac{1}{r_{3}}\right)k_{23}\\
 & + & \left(\frac{3}{4}-\frac{1}{2r_{2}}-\frac{1}{r_{4}}\right)k_{24}+\left(\frac{2}{3}-\frac{1}{r_{3}}\right)k_{31}\\
 & + & \left(\frac{2}{3}-\frac{4}{3r_{3}}\right)k_{33}+\left(\frac{5}{6}-\frac{2}{3r_{3}}-\frac{1}{r_{4}}\right)k_{34}\\
 & + & \left(1-\frac{2}{r_{4}}\right)k_{40}+\left(\frac{3}{4}-\frac{1}{r_{4}}\right)k_{41}+\left(\frac{3}{4}-\frac{2}{r_{4}}\right)k_{44}.
\end{eqnarray*}
The right-hand side can be equivalently rewritten as
\begin{eqnarray*}
-\frac{\ln A}{\ln n} & \geq & -\frac{s}{\ln n}+\left(\frac{r_{2}-2}{r_{2}}\right)k_{20}+\left(\frac{r_{2}-2}{2r_{2}}\right)k_{21}\\
 & + & \left(\frac{r_{2}-2}{2r_{2}}\right)k_{22}+\left(\frac{2r_{2}-3}{6r_{2}}+\frac{r_{3}-2}{2r_{3}}\right)k_{23}\\
 & + & \left(\frac{r_{2}-2}{4r_{2}}+\frac{r_{4}-2}{2r_{4}}\right)k_{24}+\left(\frac{2r_{3}-3}{3r_{3}}\right)k_{31}\\
 & + & \left(\frac{2r_{3}-4}{3r_{3}}\right)k_{33}+\left(\frac{r_{3}-2}{3r_{3}}+\frac{r_{4}-2}{2r_{4}}\right)k_{34}\\
 & + & \left(\frac{r_{4}-2}{r_{4}}\right)k_{40}+\left(\frac{3r_{4}-4}{4r_{4}}\right)k_{41}+\left(\frac{3r_{4}-8}{4r_{4}}\right)k_{44}.
\end{eqnarray*}
Since $r_{2},r_{3},r_{4}\geq 3$ (unless the respective $k_i$ equals 0), all the numerators in all the fractions
above are $\geq1$, so replacing them with $1$ only decreases the
sum. By doing so, and rearranging summands, we obtain
\begin{eqnarray*}
-\frac{\ln A}{\ln n} & \geq & -\frac{s}{\ln n}+\frac{1}{r_{2}}\left[k_{20}+\frac{k_{21}}{2}+\frac{k_{22}}{2}+\frac{k_{23}}{6}+\frac{k_{24}}{4}\right]\\
 & + & \frac{1}{r_{3}}\left[\frac{k_{31}}{3}+\frac{k_{32}}{2}+\frac{k_{33}}{3}+\frac{k_{34}}{3}\right]\\
 & + & \frac{1}{r_{4}}\left[k_{40}+\frac{k_{41}}{4}+\frac{k_{42}}{2}+\frac{k_{43}}{2}+\frac{k_{44}}{4}\right]\\
 & \overset{\eqref{eq:permissible}}{\geq} & -\frac{s}{\ln n}+\frac{s_{2}}{6}+\frac{s_{3}}{6}+\frac{s_{4}}{8}\\
 & \geq & -\frac{s}{\ln n}+\frac{s}{8}=\left(\frac{1}{8}-o(1)\right)s\\
 & \geq & \frac{2}{100}s.
\end{eqnarray*}
  In conclusion, we have proved Inequality (\ref{eq:-ln(A)/ln(n)_required_bound}),
and as explained this completes the proof. 
\end{proof}
\begin{corollary}
\[
P_{\sigma}\leq(\Delta s+1)^{11}\frac{1}{n^{(1+\frac{1}{100})s}}s_{2}^{\nicefrac{s_{2}}{2}}s_{3}^{\nicefrac{s_{3}}{3}}.
\]
\label{cor:app_A_cor}
\end{corollary}

\begin{proof}
We have $P_{\sigma}=\sum_{(k_{ij})}P_{\sigma,(k_{ij})}$ where the
sum is over all $\sigma$-valid lists $(k_{ij})$. Due to Lemma \ref{thm:P_sigma,k bound},
 it only remains
to explain why there at most $(\Delta s+1)^{11}$ $\sigma$-valid
lists $(k_{ij})$. 

Every $\sigma$-valid list is a list of $11$ non-negative integers,
and their sum $k$ must satisfy $k\leq\Delta s$ (due to Property
(P2) of the model). Therefore, the number of $\sigma$-valid lists
is bounded by the number of $11$-tuples of integers, each between $0$
and $\Delta s$, which is $(\Delta s+1)^{11}$.
\end{proof}
We are now ready to prove the asymmetry of ${\bf G}_{n}$.
\begin{proof}[Proof of Theorem \ref{thm:asymmetry_bollobas}]
As previously explained, it suffices to prove Equation (\ref{eq:sum of P_sigma}). Let $M(s,s_{2},s_{3})$ be the set of permutations $\sigma$ with
$s$ non-fixed vertices, $s_{2}$ vertices in $2$-cycles and $s_{3}$
vertices in $3$-cycles. Then we shall use the following fairly coarse bound
\[
\left|M(s,s_{2},s_{3})\right|\leq(n)_{s}\cdot\frac{1}{\left(\frac{s_{2}}{2}\right)!}\cdot\frac{1}{\left(\frac{s_{3}}{3}\right)!}\leq\frac{n^{s}}{\left(\frac{s_{2}}{2}\right)!\left(\frac{s_{3}}{3}\right)!}.
\]
From Corollary~\ref{cor:app_A_cor} we deduce
\begin{eqnarray*}
\sum_{\sigma\in M(s,s_{2},s_{3})}P_{\sigma} & \leq & \frac{n^{s}}{\left(\frac{s_{2}}{2}\right)!\left(\frac{s_{3}}{3}\right)!}\cdot(\Delta s+1)^{11}\frac{1}{n^{(1+\frac{1}{100})s}}s_{2}^{\nicefrac{s_{2}}{2}}s_{3}^{\nicefrac{s_{3}}{3}}\\
 & \leq & c_{3}^{s}n^{-\frac{s}{100}}
\end{eqnarray*}
where $c_{3}$ is some positive constant, since $s>1$ provided by $\sigma\neq\mathrm{id}$. Every $\sigma\neq{\rm id}$
belongs to some set $M(s,s_{2},s_{3})$ with $2\leq s\leq n$, and
for every $2\leq s\leq n$ there at most $s^{2}$ choices for $s_{2},s_{3}$
such that $M(s,s_{2},s_{3})$ is non-empty. Overall
\[
\sum_{\sigma\in\Sigma_{n}\setminus\{{\rm id}\}}P_{\sigma}\leq\sum_{s=2}^{n}s^{2}c_{3}^{s}n^{-\frac{s}{100}}=O\left(\sum_{s=2}^{n}n^{-\frac{s}{200}}\right)=o(1).
\]
Thus the proof is completed. 
\end{proof}

\section{Properties of extensions in random graphs}
\label{AppendixG}

The following properties of safe and rigid extensions will be useful.
For their proofs, see  \cite[Chapter 4]{spencer}.
\begin{prop}
Let $(R,H)$ be a rooted graph.\label{prop:extension_properties}
\begin{enumerate}
\item If $(R,H)$ is not safe then it has a rigid subextension.
\item Assume that $(R,H)$ is rigid, $H$ is a subgraph of some larger graph $G$, and 
$X\subseteq V(G)$. If $R\cup X\neq V(H)\cup X$ then $(R\cup X,G[V(H)\cup X])$
is rigid.
\end{enumerate}
\end{prop}

Recall that ${\bf G}\sim G(n,n^{-\alpha})$, $M=M(n)=(\ln n)^{1/(10d)}$, and $r,v,e=O(M)$.

Given an $r$-tuple of vertices $\vec{x}$, we let ${\bf N}_{\vec{x}}^{(R,H)}$
denote the random variable which counts $(R,H)$-extensions of $\vec{x}$
in ${\bf G}$. By definition,
\[
\mathbb{E}\left({\bf N}_{\vec{x}}^{(R,H)}\right)=\frac{1}{a}(n-r)_{v}p^{e}\sim\frac{1}{a}n^{v}p^{e}=\frac{1}{a}n^{v-\alpha e}
\]
where $a$ is the number of automorphisms of the rooted graph $(R,H)$
which preserve the root vertices.
\begin{theorem}
\label{thm:Counting_Safe_Extensions}W.h.p.\ in ${\bf G}$, for every
safe $(R,H)$ with $r\leq M$ roots and $v\leq M$ non-roots, and
for every $r$-tuple of vertices $\vec{x}$, we have ${\bf N}_{\vec{x}}^{(R,H)}\sim\mathbb{E}{\bf N}_{\vec{x}}^{(R,H)}$
uniformly. More explicitly, there exists $\delta=o(1)$ (which is
uniform w.r.t.\ $(R,H)$ and $\vec{x}$) such that w.h.p.
\begin{equation}
\frac{{\bf N}_{\vec{x}}^{(R,H)}}{\mathbb{E}{\bf N}_{\vec{x}}^{(R,H)}}\in[1-\delta,1+\delta]\label{eq:concentration_N}
\end{equation}
\end{theorem}

\begin{proof}
To prove the theorem, we shall prove that there exists $\delta=o(1)$
such that for any $(R,H)$ with $r\leq M$ roots and $v\leq M$ non-roots
and any $r$-tuple $\vec{x}$, (\ref{eq:concentration_N}) holds with
probability $1-\exp\left(-\lambda\right)$, where $\lambda=\exp\left(\Omega\left(\frac{\ln n}{M^{d}}\right)\right)$.
Then it would only remain to apply the union bound over $(R,H)$ and
$\vec{x}$. Indeed, the number of possible safe rooted graphs with
$r\leq M$ roots and $v\leq M$ non-roots is trivially bounded by
$\exp\left(O(M^{2})\right)$. The number of $r$-tuples of vertices
$\vec{x}$ is $O(n^{r})=\exp\left(O(M)\ln n\right)$. So the probability
of having $(R,H)$ and $\vec{x}$ such that $N_{\vec{x}}^{(R,H)}$
does not satisfy (\ref{eq:concentration_N}) is
\[
\exp\left(O(M^{2})+O(M)\ln n-\lambda\right)=o(1).
\]
So, from now on let us fix a safe rooted graph $(R,H)$ of type $(v,e)$
with $r\leq M$ roots and $v\leq M$, and also fix an $r$-tuple of
vertices $\vec{x}$. Denote ${\bf N}={\bf N}_{\vec{x}}^{(R,H)}$.
We prove concentration of ${\bf N}$ via Kim and Vu's powerful concentration
result (see \cite{Kim-Vu}). Note that for the rest of the proof,
$e$ denotes the number of edges in the rooted graph (not the Napier's
constant). 

Let $\mathcal{Y}$ be the set of all $\frac{1}{a}(n-r)_{v}$ $v$-tuples
$\vec{y}$ which are potential $(R,H)$-extensions of $\vec{x}$.
Let $\mathcal{E}$ be the set of all $\binom{n}{2}-\binom{r}{2}$ pairs
of vertices which do not connect two vertices of $\vec{x}$. For every
$\vec{y}\in\mathcal{Y}$, let $P_{\vec{y}}$ be the set of all edges
required by the $(R,H)$-extension $(\vec{x},\vec{y})$. This is a
subset of $\mathcal{E}$ of size $e$. Finally, for every $i\in\mathcal{E}$
let ${\bf I}_{i}$ be the indicator of the event $i\in E({\bf G})$.
Then we can write ${\bf N}$ as the value of a multivariate polynomial
on $E({\bf G})$ as follows: 
\[
{\bf N}=\sum_{\vec{y}\in\mathcal{Y}}\left[\prod_{i\in P_{\vec{y}}}{\bf I}_{i}\right].
\]
For a set of edges $A\subseteq\mathcal{E}$ with $\left|A\right|=e_{0}\leq e$,
we define
\begin{eqnarray*}
\mathcal{Y}_{A}  =  \left\{ \vec{y}\in\mathcal{Y}:A\subseteq P_{\vec{y}}\right\} ,\quad
{\bf N}_{A}  =  \sum_{\vec{y}\in\mathcal{Y}_{A}}\left[\prod_{i\in P_{\vec{y}}\setminus A}{\bf I}_{i}\right].
\end{eqnarray*}
For every $0\leq e_{0}\leq e$ let 
\[
E_{e_{0}}=\max_{A\subseteq\mathcal{E},\left|A\right|=e_{0}}\mathbb{E}\left({\bf N}_{A}\right).
\]
Also define 
\[
E=\max_{e_{0}\geq0}E_{e_{0}},\quad E'=\max_{e_{0}\geq1}E_{e_{0}}.
\]
Notice that $E_{0}=\mathbb{E}{\bf N}$, thus $E=\max\left\{ E',\mathbb{E}{\bf N}\right\} $.
Kim-Vu's concentration result asserts that for every $\lambda>1$,
\begin{align}
\mathbb{P}\left(\left|{\bf N}-\mathbb{E}{\bf N}\right|>\frac{8^{e}}{\sqrt{e!}}\sqrt{E\cdot E'}\cdot\lambda^{e}\right) & =O\left(\exp\left(-\lambda+(e-1)\ln\left|\mathcal{E}\right|\right)\right).\label{eq:Kim-Vu}
\end{align}
To make use of this result we first bound $\mathbb{E}({\bf N}_{A})$
for every set $A$ of $e_{0}\leq e$ edges. It is easy to see that
$\mathbb{E}({\bf N}_{A})=\left|\mathcal{Y}_{A}\right|p^{e-e_{0}}$; let
us bound $\left|\mathcal{Y}_{A}\right|$. Consider the set of vertices
incident to the edges of $A$, excluding the vertices of $\vec{x}$;
denote it $V_{0}$ and its size $v_{0}$. We can bound $\left|\mathcal{Y}_{A}\right|$
by the number of $v$-tuples $\vec{y}$ which contain all the vertices
of $V_{0}$ (in any positions), which is
\[
(v)_{v_{0}}\cdot(n-r-v_{0})_{v-v_{0}}\leq v!\cdot n^{v-v_{0}}.
\]
If the vertices of $V_{0}$ and the edges of $A$ cannot be completed
into an $(R,H)$-extension of $\vec{x}$, then actually $\mathcal{Y}_{A}$
is empty and $\mathbb{E}({\bf N}_{A})=0$ trivially. Otherwise, from
safeness of $(R,H)$ we deduce $v_{0}-\alpha e_{0}>0$. Overall $\mathbb{E}({\bf N}_{A})\leq v!\cdot\frac{n^{v-\alpha e}}{n^{v_{0}-\alpha e_{0}}}$.

Comparing $\mathbb{E}({\bf N}_{A})$ to $\mathbb{E}({\bf N})$ yields
\[
\frac{\mathbb{E}({\bf N}_{A})}{\mathbb{E}({\bf N})}\leq\frac{av!}{n^{v_{0}-\alpha e_{0}}}\leq\frac{(v!)^{2}}{n^{v_{0}-\alpha e_{0}}}.
\]

Maximizing over all non-empty sets $A$ of at most $e$ edges, we
obtain the bound
\[
E'\leq\frac{(v!)^{2}}{n^{\beta}}\cdot\mathbb{E}({\bf N})
\]
where $\beta=\min_{v_{0},e_{0}}(v_{0}-\alpha e_{0})$, with the minimum
taken over $1\leq v_{0}\leq v$ and $e_{0}<\frac{v_{0}}{\alpha}$. 

It is at this point that we encounter the main difference from the
original argument: since $v$ may now grow with $n$, we get that
$\beta$ is no longer constant, but approaches $0$ in a way that
reflects the behavior of rational approximations of $\alpha$. To
provide a lower bound on $\beta$ we invoke the assumption that $\left|\alpha-\frac{p}{q}\right|\geq\frac{1}{q^{d}}$ for all but finitely many rationals $\frac{p}{q}$.
Taking any $1\leq v_{0}\leq v$ and $e_{0}<\frac{v_{0}}{\alpha}$
(with a finite number of possible exceptions) we have
\[
v_{0}-\alpha e_{0}=e_{0}\left(\frac{v_{0}}{e_{0}}-\alpha\right)=e_{0}\left|\alpha-\frac{v_{0}}{e_{0}}\right|\geq\frac{1}{e_{0}^{d-1}}.
\]
Overall, this proves the lower bound $\beta\geq\frac{c}{e^{d-1}}$
where $c$ is a constant (which comes from the finite number of exceptions).
For convenience, we rewrite this lower bound in terms of $v$: since
$e<\frac{v}{\alpha}$, we have $\beta\geq\frac{c}{v^{d-1}}$ (where
now $c$ might be a different constant).

We now return to Equation (\ref{eq:Kim-Vu}). Let us first take care
of the left hand side. From what we have seen so far,
\begin{align*}
\frac{E'}{\mathbb{E}{\bf N}}\leq\frac{(v!)^{2}}{n^{\beta}}\leq\frac{v^{2v}}{n^{\beta}}&=\exp\left(2v\ln v-\beta\ln n\right)\\
&\leq\exp\left(2v\ln v-\frac{c}{v^{d-1}}\ln n\right).
\end{align*}
Since $v=O(M)=O\left((\ln n)^{\frac{1}{10d}}\right)$, we have
\[
E'\leq\exp\left(-(c-o(1))\frac{\ln n}{v^{d-1}}\right)\cdot\mathbb{E}{\bf N}.
\]
We can also trivially bound $\frac{8^{e}}{\sqrt{e!}}$ from above
by a constant, and for every $\lambda>1$, we have $\lambda^{e}\leq\lambda^{\frac{1}{\alpha}v}$.
Let us take $\lambda=\exp\left(\frac{\alpha c}{4}\cdot\frac{\ln n}{v^{d}}\right).$
This choice guarantees 
\[
\frac{8^{e}}{\sqrt{e!}}\sqrt{E\cdot E'}\cdot\lambda^{e}\leq\exp\left(-\left(\frac{c}{4}-o(1)\right)\frac{\ln n}{v^{d-1}}\right)\cdot\mathbb{E}{\bf N}.
\]
In particular $\frac{8^{e}}{\sqrt{e!}}\sqrt{E\cdot E'}\cdot\lambda^{e}\leq\delta\mathbb{E}{\bf N}$
for $\delta=o(1)$ which does not depend on the specific $(R,H)$
or $\vec{x}$.

Finally, from Equation (\ref{eq:Kim-Vu}) with $\lambda,\delta$ as
above, we have
\begin{eqnarray*}
\mathbb{P}\left(\left|N-\mathbb{E}N\right|>\delta\mathbb{E}N\right) & = & O\left(\exp\left(-\lambda+(e-1)\ln\left|\mathcal{E}\right|\right)\right)\\
 & = & \exp\left[-\lambda+O(v\ln n)\right]\\
 & = & \exp\left(-(1+o(1))\lambda\right)
\end{eqnarray*}
with $\lambda=\exp\left(\Omega\left(\frac{\ln n}{M^{d}}\right)\right).$
That finishes the proof. 
\end{proof}

\begin{remark}
\label{rem:counting_safe_extensions_upgraded}For technical reasons, we also need
the following straightforward generalization of Theorem \ref{thm:Counting_Safe_Extensions}.
{\it Following the same notation, w.h.p. for every tuple $\vec{x}$ and every set $U_{\vec{x}}$ of ``forbidden'' vertices of size
$O(M^{d+1})$ the conclusion of Theorem \ref{thm:Counting_Safe_Extensions}
holds for the random variable ${\bf P}_{\vec{x}}^{(R,H)}$ counting
only $(R,H)$-extensions $\vec{y}$ of $\vec{x}$ which do not use
any vertices from $U_{\vec{x}}$.} 
This generalization can be proved as a direct corollary of Theorem
\ref{thm:Counting_Safe_Extensions}, by adding the forbidden vertices
$U_{\vec{x}}$ to the rooted graph as roots, without any edges. Technically,
this generalization requires the replacement of $r\leq M$ with $r=O(M^{d})$,
but it has negligible effect on computations since $d$ is a constant.
\end{remark}

\subsection{Proof of Bounded Closure Lemma}

The proof is similar to the proof of the Finite Closure Theorem from
\cite{spencer}. For convenience, we present the full argument. 

We start by fixing $0\leq r,t\leq M$. As in the original argument,
set
\[
\beta=\min_{v_{0},e_{0}}\left(\frac{\alpha e_{0}-v_{0}}{v_{0}}\right)
\]
where the minimum is over all $1\leq v_{0}\leq t$ and $e_{0}>\frac{1}{\alpha}v_{0}$.
Again, this is no longer a positive constant; we now have $\beta\geq\frac{c}{t^{d}}$ where $c=c_{\alpha}$
is a positive constant.

Define $K=\left\lceil \frac{r}{\beta}\right\rceil +1$ and let us
show that with probability $1-o(M^{-2})$, the $t$-closure of any
$r$-tuple contains at most $r+K$ vertices. Then, taking a union
bound over all possible $0\leq r,t\leq M$ will finish the proof.

Assume that a set $R$ of $r$ vertices has $\left|{\rm cl}_{t}(R)\right|>r+K$.
This means there exists a sequence $R=S_{0}\subsetneq S_{1}\subsetneq S_{2}\subsetneq\dots\subsetneq S_{u}\subseteq{\rm cl}_{t}(R)$
where each $(S_{i-1},S_{i})$ is a dense rooted graph of type $(v_{i},e_{i})$
with $v_{i}\leq t$ and also the number of non-roots of $(R,S_{u})$
satisfies $\sum_{i=1}^{u}v_{i}\in[K,K+t]$.

The induced subgraph on $S_{u}$ has $v_{*}=r+\sum_{i=1}^{u}v_{i}$
vertices and at least $e_{*}=\sum_{i=1}^{u}e_{i}$ edges. Notice that
\[
v_{*}-\alpha e_{*}=r+\sum_{i=1}^{u}(v_{i}-\alpha e_{i})\leq r-\sum_{i=1}^{u}\beta v_{i}\leq r-\beta K\leq-\beta.
\]
It is therefore sufficient to prove the following statement. W.h.p.,
there exists no subgraph $H$ of ${\bf G}$ such that $v(H)\in[r+K,r+K+t]$ and $v(H)-\alpha e(H)\leq-\beta$. Fix a graph $H$ which satisfies the above conditions. Its
expected number of copies is 
\[
\mathbb{E}\left({\bf X}_{H}\right)\leq n^{v(H)-\alpha e(H)}\leq n^{-\beta}=\exp\left(-\beta\ln n\right)=\exp\left(-\Omega\left(\frac{\ln n}{M^{d}}\right)\right)
\]
since $\beta\geq\frac{c}{t^{d}}$ and $t=O(M)$. Let us recall that $r=O(M)$ and $K=O\left(rt^{d}\right)=O\left(M^{d+1}\right)$. It is sufficient to take the union bound only over those $H$ with
minimal number of edges; that is, we may assume $e(H)=\frac{1}{\alpha}v(H)+O(1)$.
Therefore, the probability of existence of a subgraph $H$ of ${\bf G}$
as above is at most
\begin{gather*}
\exp\left[O\left((r+K+t)\ln(r+K+t)\right)-\Omega\left(\frac{\ln n}{M^{d}}\right)\right]\\
=\exp\left(O\left(M^{d+1}\ln M\right)-\Omega\left(\frac{\ln n}{M^{d}}\right)\right).
\end{gather*}
This is $o(M^{-2})$ since $M=(\ln n)^{\frac{1}{10d}}$.

\subsection{Proof of Lemma~\ref{lem:generic_extensions}}

The idea is to use
Theorem \ref{thm:Counting_Safe_Extensions} to claim that there are
many $(R,H)$-extensions, and also that the majority of them must
be generic. 

From Theorem \ref{thm:Counting_Safe_Extensions}, Remark \ref{rem:counting_safe_extensions_upgraded},
and the Bounded Closure Lemma (with $M$ replaced by $2M$), w.h.p.\ the
following properties hold true in ${\bf G}$:
\begin{enumerate}
\item For every safe rooted graph $(R',H')$ with $r'\leq2M$ roots and
$v'\leq2M$ non-roots, and for every $r'$-tuple $\vec{x}'$,
\begin{equation}
{\bf N}_{\vec{x}'}^{(R',H')}\sim\mathbb{E}\left({\bf N}_{\vec{x}'}^{(R',H')}\right).\label{eq:concentration_N'}
\end{equation}
\item For every $r$-tuple $\vec{x}$, $r\leq M$, letting $U_{\vec{x}}={\rm cl}_{2M}(\vec{x})\setminus\vec{x}$,
we have $\left|U_{\vec{x}}\right|=O\left(M^{d+1}\right)$.
\item For every safe rooted graph $(R,H)$ with $r\leq M$ roots and $v\leq M$
non-roots, and for every $r$-tuple $\vec{x}$ of vertices, 
\[
{\bf P}_{\vec{x}}^{(R,H)}\sim\mathbb{E}\left({\bf P}_{\vec{x}}^{(R,H)}\right)
\]
where ${\bf P}_{\vec{x}}^{(R,H)}$ counts $(R,H)$-extensions $\vec{y}$
which do not use any vertex from $U_{\vec{x}}$ (we call such $(R,H)$-extensions
\emph{good}). 
\end{enumerate}
Let $\mathcal{A}$ be the event that the three conditions above hold;
then $\mathbb{P}\left(\mathcal{A}\right)=1-o(1)$. We now fix a safe
rooted graph $(R,H)$ with $r\leq M$ roots and $v\leq M$ non-roots,
a non-negative integer $t\leq M$ and an $r$-tuple of vertices $\vec{x}$.
To complete the proof, it suffices to show that given the event $\mathcal{A}$,
$\vec{x}$ has a $t$-generic $(R,H)$-extension $\vec{y}$.

Let $(v,e)$ denote the type of $(R,H)$. Since $\mathcal{A}$ holds,
we know that 
\[
{\bf P}_{\vec{x}}^{(R,H)}\sim\mathbb{E}\left({\bf P}_{\vec{x}}^{(R,H)}\right)\sim\frac{n^{v-\alpha e}}{a}
\]
where $a$ is the number of automorphisms of $(R,H)$ which preserve
the roots. We prove that, given that $\mathcal{A}$ holds, the number
of good $(R,H)$-extensions $\vec{y}$ of $\vec{x}$ which fail to be $t$-generic
is $o\left(\frac{n^{v-\alpha e}}{a}\right)$. 
First, let us bound the number of good $(R,H)$-extensions $\vec{y}$
which violate the first part of Definition \ref{def:generic}, meaning
that they contain additional edges except those specified by $H$.
Assume that $\vec{y}$ is such an extension; then it forms an $(R,H^{+})$-extension
of $\vec{x}$ of type $(v,e^{+})$ with $e^{+}>e$. This extension
cannot have any rigid subextensions, because $\vec{y}$ shares no
vertices with $U_{\vec{x}}$. From part 1 of Proposition \ref{prop:extension_properties}
it is safe. The number of $(R,H^{+})$-extensions of $\vec{x}$ is
therefore
\[
O\left(n^{v}p^{e^{+}}\right)=O\left(an^{-\alpha}\cdot\frac{n^{v-\alpha e}}{a}\right)=n^{-\alpha+o(1)}\cdot\frac{n^{v-\alpha e}}{a}.
\]
The last equality follows from the bound $a\leq v!=n^{o(1)}$. Summing
over $\exp\left(O(M^{2})\right)$ possibilities for the rooted graph
$(R,H^{+})$, we obtain a bound of $o\left(\frac{n^{v-\alpha e}}{a}\right)$
on the number of $(R,H)$-extensions of $\vec{x}$ which violate
the first part of Definition \ref{def:generic}.

Now we bound the number of good $(R,H)$-extensions $\vec{y}$ which
violate the second part of Definition \ref{def:generic}. Assume that
$\vec{y}$ is such an extension, and let $\vec{z}$ be a rigid extension
of $\vec{x}\cup\vec{y}$ with at most $t$ nonroots and at least one
edge between $\vec{z}$ and $\vec{y}$. Let $(V(H),H_{1})$ denote the
rooted graph corresponding to this extension, and let $(v_{1},e_{1})$
denote its type. Without loss of generality, we may assume that $\vec{z}$
is minimal rigid extension of $\vec{x}\cup\vec{y}$, in the sense
that $(V(H),H_{1})$ has no non-trivial rigid subextensions. Note that
$\vec{z}\cup\vec{y}$ is an $(R,H_{1})$-extension of $\vec{x}$, which is of
type $(v+v_{1},e+e_{1})$, with $v+v_{1}\leq v+t\leq2M$. We then consider cases.

\emph{Case 1.} $(R,H_{1})$ is safe. The number of $(R,H_{1})$-extensions
of $\vec{x}$ is 
\[
O\left(n^{v+v_{1}}p^{e+e_{1}}\right)=O(n^{v_{1}-\alpha e_{1}}n^{v-\alpha e}).
\]
$(V(H),H_{1})$ is rigid, so $v_{1}-\alpha e_{1}<0$. Actually, we have $v_{1}-\alpha e_{1}=-\Omega\left(\frac{1}{M^{d-1}}\right)$.
Therefore
\[
O(n^{v_{1}-\alpha e_{1}}n^{v-\alpha e})=\exp\left(-\Omega\left(\frac{\ln n}{M^{d-1}}\right)\right)n^{v-\alpha e}.
\]
Summing over $\exp\left(O(M^{2})\right)$ possibilities for $(V(H),H_{1})$, we obtain a bound of $o\left(\frac{n^{v-\alpha e}}{a}\right)$
on the number of $\vec{y}$ which violate the second part of Definition
\ref{def:generic}.

\emph{Case 2. }$(R,H_{1})$ is not safe. We first show that in this
case, $\vec{z}$ is a rigid extension of $\vec{x}$. Since $(R,H_{1})$
is not safe, from part 1 of Proposition \ref{prop:extension_properties}
it has a rigid subextension $(R,H_{2})$. The vertices of $\vec{y}\cup\vec{z}$
which form this $(R,H_{2})$-extension of $\vec{x}$ are actually
all in $\vec{z}$; we denote them by $\vec{z}'$. Indeed, by definition
these vertices are contained in $U_{\vec{x}}$ (as they form a rigid
extension of $\vec{x}$ with at most $v+t\leq2M$ vertices) but $\vec{y}$
contains no vertices from $U_{\vec{x}}$. From part 2 of Proposition
\ref{prop:extension_properties} we deduce that $\vec{x}\cup\vec{y}\cup\vec{z}'$
is a rigid extension of $\vec{x}\cup\vec{y}$. It corresponds to a
rigid subextension of $(V(H),H_{1})$, but from minimality it must be
$(V(H),H_{1})$ itself, so $\vec{z}'=\vec{z}$. Therefore $\vec{z}$
is a rigid extension of $\vec{x}$.

Now, since $\vec{z}$ is a rigid extension of $\vec{x}$, its vertices
are all in $U_{\vec{x}}$. Since $\vec{z}$ has at most $t$ vertices
and $\left|U_{\vec{x}}\right|=O(M^{d+1})$, there are $O\left(M^{(d+1)t}\right)=\exp\left(O(M\ln M)\right)$
possibilities for $\vec{z}$.

Fix a possible $\vec{z}$. Consider $\vec{y}$ as an extension of
$\vec{x}\cup\vec{z}$. It does not have a rigid subextension, since
$\vec{y}$ contains no vertex from $U_{\vec{x}}$, and from part 1
of Proposition \ref{prop:extension_properties} it is safe. Let $(R',H_1)$
denote the rooted graph corresponding to this extension. Its type
is $(v,e^{+})$ with $e^{+}>e$, because there is at least one edge
between $\vec{y}$ and $\vec{z}$. Its number of roots is at most
$r+t\leq2M$. Therefore the number of $(R',H')$-extensions of $\vec{x}\cup\vec{z}$
is
\[
O\left(n^{v}p^{e^{+}}\right)=O(n^{-\alpha})n^{v-\alpha e}.
\]
Summing over $\exp\left(O(M^{2})\right)$ possibilities for the rooted
graph $(R',H_1)$ and $\exp\left(O(M\ln M)\right)$ possibilities
for $\vec{z}$, we get a bound of $o\left(\frac{n^{v-\alpha e}}{a}\right)$
on the number of $\vec{y}$ which violate the second part of Definition
\ref{def:generic}. That completes the
proof.

\end{document}